\documentclass[final,onefignum,onetabnum]{siamart190516}

\usepackage{amssymb}
\usepackage{amsmath}
\usepackage{xcolor}
\usepackage{graphicx}
\usepackage{epstopdf}
\usepackage{enumitem}
\usepackage{comment}
\usepackage{psfrag}
\usepackage{subfig}

\usepackage[textsize=tiny,color=yellow!60!green]{todonotes}
\usepackage{multirow}
\usepackage{bm}
\usepackage{mathtools}
\usepackage{stmaryrd}
\usepackage{soul}
\usepackage{multirow}
\usepackage{diagbox}
\usepackage{tabularx}
\usepackage{cite} 

\usepackage{cleveref}
\usepackage{mathbbol}
\usepackage{amssymb}             

\DeclareSymbolFontAlphabet{\amsmathbb}{AMSb}%
\newcommand{\defeq}{\vcentcolon=}
\newcommand{\textn}[1]{\textnormal{#1}}

\DeclareMathAlphabet{\mathbfit}{OML}{cmm}{b}{it}
\newcommand{\jump}[1]{\left\llbracket #1 \right\rrbracket}

\newcommand{\RR}{{\amsmathbb R}}

\newcommand{\Omegai}{\Omega_{i}}

\newcommand{\vecv}{{\mathbf{v}}}
\newcommand{\vecr}{{\mathbf{r}}}
\newcommand{\vecu}{{\mathbf{u}}}

\newcommand{\vK}{\mathbf K}
\newcommand{\II}{{\textbf I}}

\newcommand{\hg}{h,\gamma}

\newcommand{\vecn}{{\mathbf{n}}}
\newcommand{\vecx}{{\mathbf{x}}}

\newcommand{\hi}{h,i}

\newcommand{\eq}{\defeq}

\newcommand{\Tau}{\mathcal{T}}

\newcommand{\Assum}{(\textbf{A1})--(\textbf{A5})}

\DeclareMathOperator*{\argmin}{arg\,min}

\usepackage{tabularx}

\newtheorem{thm}{Theorem}
\numberwithin{thm}{section}

\newtheorem{lem}[thm]{Lemma}
\newtheorem{prop}[thm]{Proposition}
\newtheorem{cor}[thm]{Corollary}
\newtheorem{df}[thm]{Definition}

\newtheorem{rem}[thm]{Remark}

\newtheorem{algo}[thm]{Algorithm}

\newcommand{\bse}{\begin{subequations}}
\newcommand{\ese}{\end{subequations}}

\newcommand{\red}[1]{{\color{red}{#1}}}

\makeatletter
\newcommand{\eqnum}{\refstepcounter{equation}\textup{\tagform@{\theequation}}}
\makeatother

\graphicspath{{fig/}}

\newcommand{\FIX}[2]{{{#2}}}

\newcommand{\FIXX}[2]{{{#2}}}

\title{Robust linear domain decomposition  schemes for\\
 reduced  non-linear fracture flow models}

\author{Elyes Ahmed\footnotemark[2] \footnotemark[4]
\and Alessio Fumagalli\footnotemark[1]
\and Ana Budi\v{s}a\footnotemark[2]\\
\and Eirik Keilegavlen\footnotemark[2]
\and Jan M. Nordbotten\footnotemark[2]
\and Florin A. Radu\footnotemark[2]
}
\date{\today}

\begin{document}

\maketitle

\renewcommand{\thefootnote}{\fnsymbol{footnote}}

\footnotetext[2]{Department of Mathematics, University of Bergen, P. O. Box 7800, N-5020 Bergen, Norway.
\href{mailto:ana.budisa@uib.no}{ana.budisa@uib.no},
\href{mailto:eirik.keilegavlen@uib.no}{eirik.keilegavlen@uib.no},
\href{mailto:jan.nordbotten@uib.no}{jan.nordbotten@uib.no},
\href{mailto:florin.radu@uib.no}{florin.radu@uib.no},
}
\footnotetext[4]{SINTEF Digital, Mathematics and Cybernetics, Oslo, Norway (current address).
\href{mailto:elyes.ahmed@sintef.no}{elyes.ahmed@sintef.no},
}
\footnotetext[1]{Department of Mathematics, Politecnico di Milano, piazza Leonardo da Vinci 32, 20133 Milan, Italy.
\href{mailto:alessio.fumagalli@polimi.it}{alessio.fumagalli@polimi.it}.
}
\renewcommand{\thefootnote}{\arabic{footnote}}

\numberwithin{equation}{section}

\begin{abstract}
 In this work, we consider compressible single-phase flow problems in a porous
 media containing a fracture. In the latter, a non-linear pressure-velocity
 relation is prescribed. Using a non-overlapping domain decomposition procedure,
 we reformulate the global problem into a non-linear interface problem. {We then
 introduce two\FIX{ new}{} algorithms that are able to
 efficiently handle the non-linearity and the     coupling between the fracture
 and\FIX{ the}{} matrix,  both based on
linearization by the so-called L-scheme}. The first algorithm, named MoLDD, uses
the L-scheme \FIX{to resolve}{for}
 the non-linearity, requiring at each  iteration to solve the dimensional
 coupling via a domain decomposition approach.  The second algorithm, called
 ItLDD, uses a sequential approach in which the dimensional coupling is part of
 the linearization iterations.  For both algorithms, the computations are
 reduced only to the fracture by pre-computing, in an offline phase, a
 multiscale flux basis (the linear Robin-to-Neumann co-dimensional map), that
 represent the flux exchange between the fracture and the matrix.  We present
 extensive theoretical findings\FIX{ and in particular, t}{. T}he stability and the
 convergence of both schemes are obtained, where user-given parameters  are
 optimized to  minimise the number of iterations.  Examples on two important
 fracture models are computed with the library PorePy and agree with the
 developed theory.
\end{abstract}


\vspace{3mm}




\pagestyle{myheadings} \thispagestyle{plain} \markboth{E. Ahmed, A. Fumagalli and A. Budi\v{s}a}{}




%
\section{Introduction}\label{sec:Intro}
 {Fractures  are  ubiquitous  in  porous  media  and  strongly affect  the  flow
 and transport. Several energy and environmental applications including carbon
 sequestration, geothermal energy,  and ground-water contamination involve
 flow  and transport problems  in a porous medium containing fractures.
 Typically, fractures are thin and long formations that correspond to a fast
 pathway along which the medium properties, such as permeability or porosity, differ
 from the adjacent formations (the rocks) \cite{MR1911534,MR3264361,Martin2005,MR3846260}.
 Specifically, the permeability of the fracture can be significantly higher than that of
 the host rock. As a consequence, while flow in the host rock can be well represented by the linear
 Darcy's law, flow in the fractures can potentially exhibit non-linear effects.
Models for such flow will thus be non-linear, but with the non-linear effects confined to specific parts of the domain,
moreover, these regions are characterized by an extreme aspect ratio.

\FIX{}{
In this paper, we consider models with fracture flow represented by an unstready Forchheimer's law ~\cite{MR1729811,MR3264361},
as an extension of the model in~\cite{ahmed2018multiscale,MR3264361}.
Our approach can straightforwardly be broadened to cover viscosity models for generalized Newtonian fluids \cite{MR3388812,MR2842139}.
We  also  refer
to~\cite{MR2776916,AHMED2019103431,MR3671645,MR3624734,list2018upscaling}  for  extensions  to other
flow models.
To limit complications relating to mesh construction and computational cost,
we represent the fracture as a lower-dimensional object embedded in the full domain,
as introduced in~\cite{MR1911534, Boon2020}, and
we refer to the resulting model as mixed-dimensional.
We then apply a domain decomposation (DD) approach, with the fracture forming an interface between subdomains.
The domain decomposition approach is beneficial for both
modeling, discretization and the formulation non-linear and linear solvers

Considerable research efforts have been conducted to mixed-dimensional fracture
models. Several numerical schemes for steady-state models have been proposed,  such
as the cell-centered finite volume scheme \cite{HAEGLAND20091740}, the  extended
finite element method  \cite{MR3631391}, the mimetic finite
difference \cite{MR3507274}, the block-centered finite difference method
\cite{MR3693346} and the mixed finite element (MFE) methods \cite{Martin2005,angelo_scotti_2012,Frih2012,MR3829517}.
Herein, we discretize the generalized mixed-dimensional Forchheimer problem with a mortar mixed finite element method
(MMFEM)~\cite{MR2557486,MR3869664,MR3765858}, combined with backword Euler in time.
See also \cite{Berre2018} for a review on fracture
models and discretization approaches.

While the dimension reduction reduces the number of cells necessary to represent the fracture,
the computational cost in solving the discrete non-linear problem can still be significant,
in particular for complex fracture geometries.
This calls for the construction of efficient solvers, and domain decomposition
facilitates the exploitation of both the geometric structure of the problem,
and the spatial localization of non-linearities.
See, for example, the application of DD to reduced Darcy \cite{MR1911534,MR3457700} and Darcy-Forchheimer \cite{MR2386967} fracture models.
In this work, we develop DD schemes \cite{ahmed2018multiscale,MR1911534,MR2434950}
to solve the non-linear problem resulting from our discretization.
To exploit the geometric structure of the problem, we reformulate it
as an interface problem by eliminating the subdomain variables, obtaining a non-linear system  to  solve at each time step. That is, the
resulting  problem posed only on the fracture is a \textit{superposition} of a
\textit{non-linear local} flow operator   within the fracture and a
\textit{linear non-local} one (Robin-to-Neumann type) handling the  flux contributions from the subdomains.
For this problem, two schemes are proposed, both based on the so-called $L$-scheme
method, a robust quasi-Newton method with a parameter $L>0$ mimicking the Jacobian ~\cite{MR2079503,MR3489128}.
Our two approaches differ in the way they handle the non-linearity, and in the degree of coupling between the
fracture interface and the surrounding subdomains.

The first algorithm named the Monolithic LDD-scheme
(MoLDD) employs the $L$-scheme as a linearization procedure. At each $L$-scheme iteration, an inner algorithm is used to solve the linear interface problem~\cite{MR3022024}. It can be a direct or an iterative method (e.g. a Krylov method). The action of the interface operator
requires solving subdomain problems with Robin boundary conditions on the
fracture.  This algorithm is \textit{Jacobian-free}, solving subdomain problems can be done in parallel and is later shown to be \textit{unconditionally stable}. We also obtain the condition number estimates of the inner DD system, the
contraction estimates and rates of convergence for the outer scheme.
However, there is still a computational overhead
associated with its non-local part \cite{ahmed2018multiscale,MR3022024,AHMED2019103431}, that is, the subdomain solves.
Increasing the non-linearity strength, the number of subdomains and refining the grids all
lead to an increase in the number of iterations and the number of subdomain
solves.

More recently, the $L$-scheme  has gained attention as an efficient
solver to  treat simultaneously  non-linear and coupling  effects in complex
problems \cite{MR3771899, MR3827264}.
Building on this idea, we propose the second algorithm, referred as the Iterative
LDD-scheme (ItLDD). In ItLDD, the
$L$-scheme is now synchronizing  linearization  and domain decomposition
through one-loop algorithm \cite{ahmed2019adaptive_2,brun2019monolithic}. At each iteration it has the
cost of the sequential approach, yet it converges to the fully monolithic
approach. This way we reduce the computational cost as no inner DD solver is
required  and  only a modest number  of subdomain solves that can be done in parallel are needed at each
iteration. This algorithm increases
local to non-local cooperation and saves computational time if one process is dominating
the whole problem. This  approach   differs  from  the  one  commonly  used  in DD methods for
non-linear interface problems \cite{ahmed2017posteriori,MR3261611}.

The second contribution of this paper concerns the robust and efficient
implementation of the two LDD schemes. The dominant computational
costs in these schemes comes from the subdomain solves and, to reduce this,
we use the multiscale flux basis framework
from~\cite{MR2557486}. The fact that   the non-linearity
in the system appears
within the local operator on the fracture motivates that the
linear non-local contribution from the subdomains can be expressed  as a
 superposition of multiscale basis functions~\cite{MR2557486,MR3577939,ahmed2018multiscale}, in
 spirit of \textit{reduced basis}~\cite{URBAN2012716, MR3644446,MR3853612}.
\FIX{These}{This} multiscale flux basis consists of the flux (velocity trace) response from each fracture pressure degrees of freedom.  They are
computed by  solving a \textit{fixed} number of \textit{steady} Robin subdomain
problems, equal to the number of fracture pressure degrees of freedom
per subdomain.  An inexpensive linear combination of the multiscale flux basis
functions  then replaces the  subdomain solves in any inner/outer iteration of
the  algorithms. This step of freezing  the contributions
from the rock matrices can be cheaply evaluated and easily implemented in the algorithms.
\FIX{That is,  i}{I}t permits
reusing the same basis functions  to compare  MoLDD with  ItLDD,  to simulate various linear and non-linear
models for flow in the fracture and to vary the fracture permeability.
In case of a fixed time step, the multiscale flux basis is constructed only once
in the offline phase. 
Numerical results are computed with the library PorePy \cite{Keilegavlen2019a}.

\FIX{This paper is organized as follows: }{} \FIX{Firstly, t}{In
\Cref{sec:model_problem} the model problem is presented. T}he   approximation of
problem~\eqref{weak_mixed_formulation}  using
the MMFEM in space  and  a backward Euler scheme in time is given in~\Cref{sec:problem_formulation}. Also, the reduction of this
mixed-dimensional scheme into a non-linear interface one is introduced.  The
LDD-schemes  are formulated in~\Cref{sec:algorithms}.
In~\Cref{sec:analysis_moldd} and~\Cref{sec:analysis_itldd}, the  analysis of the
schemes is presented. \Cref{sec:Mufbi} describes the implementation based on the
multiscale flux basis framework.  Finally, we \FIX{showcase}{show} the performance of  our methods on
several numerical examples in~\Cref{sec:examples} and draw the conclusions
in~\Cref{sec:conclusion}.}

\section{Model problem}\label{sec:model_problem}

Let $\Omega$ be a
bounded domain in $\RR^{d}$, $d\in\{2,3\}$,  with \FIX{}{Lipschitz continuous}
boundary $\Gamma\eq\partial\Omega$. Furthermore, let $T$ be the final time
simulation and $I\eq(0,T)$. Suppose that $\gamma \subset\Omega$ is a
$(d-1)$-dimensional \FIX{}{non self-intersecting} surface
\FIX{}{of class $C^2$}
that divides  $\Omega$ into two subdomains:
$\Omega=\Omega_{1}\cup\Omega_{2}\cup\gamma$, where $\gamma\eq\partial
\Omega_{1}\cap\partial \Omega_{2}$ and $\Gamma_{i}\eq\partial \Omega_{i}\cap
\partial \Omega$, $i\in\{1,2\}$.  Assume \FIX{that}{} the flow in $I\times\Omega_{i}$,
$i\in\{1,2\}$, is \FIX{described}{given} by\FIX{the system of equations}{}
\bse\label{Initial_system_porous}\begin{alignat}{5}
    \label{Initial_system_d_p}&\vK_{i}^{-1}\vecu_{i}+\nabla p_{i}=\mathbf{0}&&\quad \textn{in} \;  I\times\Omega_{i},\\
    \label{Initial_system_c_p}&\partial_{t}p_{i}+\nabla\cdot\vecu_{i} =f_{i} &&\quad \textn{in} \; I\times\Omega_{i},\\
    \label{Initial_system_bd_p}&p_{i} =0 &&\quad \textn{in} \;I\times\Gamma_{i},\\
    \label{Initial_system_IC_p}&p_{i}(\cdot,0) =p^{0}_{i} &&\quad \textn{in} \;   \Omega_{i},
\end{alignat}\ese
 and  in $I\times\gamma$ by the following equations

\bse\label{Initial_system_fracture}\begin{alignat}{4}
    \label{Initial_system_d_f}&\xi(\vecu_{\gamma})+\vK^{-1}_{\gamma}\vecu_{\gamma}+\nabla_{\tau} p_{\gamma}=\mathbf{0} &&\quad \textn{in} \;  I\times\gamma,&\\
    \label{Initial_system_c_f}&\partial_{t}p_{\gamma}+\nabla_{\tau}\cdot\vecu_{\gamma} =f_{\gamma}+\left(\vecu_{1}\cdot\vecn_{1}+\vecu_{2}\cdot\vecn_{2}\right) &&\quad \textn{in} \;  I\times\gamma,&\\
    \label{Initial_system_bd_f}&p_{\gamma} =0 &&\quad \textn{in} \;I\times\partial\gamma,&\\
    \label{Initial_system_IC_f}&p_{\gamma}(\cdot,0) =p^{0}_{\gamma} &&\quad \textn{in} \;   \gamma,
\end{alignat}\ese
where the transmission  conditions \FIX{for $i\in\{1,2\}$, are prescribed}{ for $i\in\{1,2\}$ are}
\begin{alignat}{4}
\label{Initial_system_r_interface}-\vecu_{i}\cdot\vecn_{i}+\alpha_{\gamma} p_{i} &=\alpha_{\gamma} p_{\gamma} &&\quad \textn{on} \;  I\times \gamma.
\end{alignat}
Here, $\nabla_{\tau}$ denotes  the $(d-1)$-dimensional   gradient  operator  in  the  plane  of $\gamma$, $\vK_{\gamma}$  is the hydraulic conductivity tensor in the fracture, $\vK_{i}$ is the hydraulic conductivity tensor in the subdomain
$\Omegai$ and
$\vecn_{i}$ is the outward unit normal vector to $\partial\Omegai$, $i\in\{1,2\}$.
The function  $\xi$ is a  non-linear
function extending the classical Forchheimer flow
to more general laws.
\FIX{In~\eqref{Initial_system_r_interface}, t}{T}he coefficient $\alpha_{\gamma}$ is proportional to  the normal component of the \FIX{of the physical fracture}{fracture} permeability and inversely proportional to the fracture width/aperture.
The functions $f_{\gamma}$ and $f_{i}$, $i\in\{1,2\}$, are source terms in the fracture and in the matrix,
respectively.  For simplicity, we have imposed a homogeneous Dirichlet condition
on the boundary $\partial\Omega$.
Finally,  $p^{0}_{\gamma}$ and  $p^{0}_{i}$, $i\in\{1,2\}$, are  initial conditions.

The system~\eqref{Initial_system_porous}--\eqref{Initial_system_r_interface} is
a mixed-dimensional model for flow in fractured porous media: the
equations~\eqref{Initial_system_d_p}-\eqref{Initial_system_c_p} are Darcy's law and mass
conservation equations in the subdomain $\Omegai$,
while \eqref{Initial_system_d_f}-\eqref{Initial_system_c_f} are \FIX{lower-dimensional non-Darcy flow }{}generalized
Forchheimer's law and mass conservation in the fracture of co-dimension one. Together these
equations form a non-standard transmission problem where the fracture system
sees the  surrounding matrix system through the source term
$\vecu_{1}\cdot\vecn_{1}+\vecu_{2}\cdot\vecn_{2}$
in~\eqref{Initial_system_c_f} and the\FIX{rock}{} matrix system communicates to the
fracture through Robin \FIX{type}{}\FIX{boundary}{interface}
conditions~\eqref{Initial_system_r_interface}. Note that the restriction to only
one fracture is made \FIX{for the ease of}{to simplify the} presentation, but the model and the
analysis below can \FIX{straightforwardly}{easily} be extended to fracture \FIX{}{networks} \cite{ahmed2018multiscale,Martin2005}.

\subsection{Assumptions on the data and weak formulation}
Let $D\subseteq\Omega$. For $s\geq0$, $||\cdot||_{s,D}$ stands for the usual
Sobolev norm on $H^{s}(D)$. If $s=0$, $||\cdot||_{D}$ is simply the $ L^{2}$
norm and $(\cdot,\cdot)_{D}$  stands for the $ L^{2}$ scalar product.  We define
the weak spaces in $\Omega_i$ for $i\in\{1,2\}$ as
\begin{align*}
    \mathbf{V}_{i}\defeq \left\{ \mathbf{v} \in \mathbf{H}(\textn{div}, \Omega_{i}):
    \mathbf{v}\cdot\mathbf{n}_i \in L^2(\gamma) \right\}
    \quad \text{and} \quad
    M_{i}\defeq L^{2}(\Omega_{i}),
\end{align*}
where we have implicitly considered the trace operator of $\mathbf{v} \cdot \mathbf{n}_i$.
Moreover, we introduce their global versions by
 $\mathbf{V}\defeq\bigoplus_{i=1}^{2}\mathbf{V}_{i}\textn{ and } M\defeq\bigoplus_{i=1}^{2} M_{i}.$ The mixed spaces on the fracture $\gamma$, are
$\mathbf{V}_{\gamma}\defeq\mathbf{H}(\textn{div}_{\tau}, \gamma)\textn{ and }  M_{\gamma}\defeq L^{2}(\gamma)$.  For simplicity of notation, we introduce the jump $\jump{\cdot}$
   given  by $\jump{\vecu \cdot
   \vecn}\defeq\vecu_{1}\cdot\vecn_{1}+\vecu_{2}\cdot\vecn_{2}$.
 and the functions $\vK$ and $f$  in $\Omega_{1}\cup \Omega_{2}$  such that $\vK_{i}\eq\vK|_{\Omegai}$, and
 $f_{i}\eq f|_{\Omegai}$, $i\in\{1,2\}$. \FIX{Throughout the paper, we assume that the following
 assumptions hold true:}{We assume:}
 \begin{enumerate}[label=(\textbf{A\arabic*})]\label{assumpA}
\item $\xi:\RR\rightarrow\RR$ is $C^{1}$, strictly increasing and
Lipschitz continuous, i.e., there exist $\xi_{m}>0$ and $L_{\xi}$ such that
$\xi_{m}\leq \xi^\prime(\vecu) \leq L_{\xi}<\infty$.
 Otherwise, we ask bounded flux \FIX{for the differential
problem}{in}~\eqref{Initial_system_porous}--\eqref{Initial_system_r_interface}, i.e,
$\vecu\in \left[L^{\infty}(\Omega)\right]^{d}$, when $\xi$ is\FIX{ simply}{} an
increasing function ($\xi^\prime \geq 0$), and let
$L_{\xi}\eq\sup_{|\vecu|\leq C_{\xi}}{\xi^\prime(\vecu)}$,
where $C_{\xi}\eq\sup_{\vecx\in \overline{\Omega}}|\vecu(\vecx)|$. \label{ass_xi}
\item $\vK:\RR^{d}\rightarrow\RR^d$ is assumed to be constant in time and bounded: there exist
$c_{\vK}>0$ and $C_{\vK}$ such that $ \zeta^{\textn{T}}\vK^{-1}(\vecx)\zeta  \geq
c_{\vK}|\zeta|^{2}$ and $|\vK^{-1}(\vecx)\zeta|\leq  C_{\vK}|\zeta|$ for a.e.
$\vecx\in\Omega_{1}\cup\Omega_{2}$, $\forall\zeta\in\RR^{d}$.\label{ass_K_omega}
\item $\vK_{\gamma}:\RR^{d-1}\rightarrow\RR^{d-1}$ is assumed to be constant in time and bounded; there exist
$c_{\vK,\gamma}>0$ and $C_{\vK,\gamma}$ such that $ \zeta^{\textn{T}}\vK^{-1}_{\gamma}(\vecx)\zeta \geq
c_{\vK,\gamma}|\zeta|^{2}$ and $|\vK^{-1}_{\gamma}(\vecx)\zeta|\leq  C_{\vK,\gamma}|\zeta|$ for a.e. $\vecx\in\gamma$,
$\forall\zeta\in\RR^{d-1}.$\label{ass_K_gamma}
\item The Robin parameter $\alpha_{\gamma}$ is a  strictly positive constant: $\alpha_{\gamma} > 0$.
 \item The initial conditions are such that $p_{i}^{0}\in L^{2}(\Omegai)$, $i\in\{1,2\}$, and
$p^{0}_{\gamma}\in L^{2}(\gamma)$. The source terms are such that $f_{i}\in
L^{2}(0,T;L^{2}(\Omegai))$, $i\in\{1,2\}$, and
$f_{\gamma}\in L^{2}(0,T;L^{2}(\gamma))$. For simplicity we further assume that $f$ and $f_{\gamma}$ are piecewise constant in
    time with respect to the temporal mesh introduced\FIX{ in
    \Cref{subsec:Discretization}}{}.
\end{enumerate}
\begin{rem}[On assumptions] The Lipschitz-continuity  of~$\xi$ is not true when
the function $\xi$ (therefore the flux) is unbounded, as it is the case
for generalized Forchheimer's law. However, for bounded flux $\vecu$, this can
be verified. Otherwise, this assumption can be recovered by truncating
the original function $\xi$. Obviously, the solution of the truncated
problem will not in general solve the original one. See, for example~ \cite{MITRA20191722}. \end{rem}
  We
introduce   the  bilinear forms $a_{i}:\mathbf{V}_{i}\times \mathbf{V}_{i}\rightarrow\RR$, $b_{i}:\mathbf{V}_{i}\times M_{i}\rightarrow\RR$ and
 $c_{i}:M_{i}\times M_{i}\rightarrow\RR$, $i\in\{1,2\}$,
\begin{alignat}{2}\label{bil:omegai}
& a_{i}(\vecu,\vecv)\eq (\vK^{-1}\vecu,\vecv)_{\Omega_{i}}+{\alpha_{\gamma}^{-1}}(
\vecu\cdot\vecn_{i},\vecv\cdot\vecn_{i})_{\gamma},\quad b_{i}(\vecu,q)\eq(\nabla\cdot\vecu,q)_{\Omega_{i}},\quad c_{i}(p,q)\eq(p,q)_{\Omega_{i}}.
\end{alignat}
On the fracture, we define the bilinear forms $a_{\gamma}:\mathbf{V}_{\gamma}\times \mathbf{V}_{\gamma}\rightarrow\RR$, $b_{\gamma}:\mathbf{V}_{\gamma}\times M_{\gamma}\rightarrow\RR$ and
 $c_{\gamma}:M_{\gamma}\times M_{\gamma}\rightarrow\RR$,
 \begin{alignat}{2}
& a_{\gamma}(\vecu,\vecv)\eq (\vK^{-1}_{\gamma}\vecu,\vecv)_{\gamma},\quad
b_{\gamma}(\vecu,\mu)\eq(\nabla_{\tau}\cdot\vecu,\mu)_{\gamma},\quad c_{\gamma}(\lambda,\mu)\eq(\lambda,\mu)_{\gamma}.
\end{alignat}
With the above notations, a weak solution
of~\eqref{Initial_system_porous}--\eqref{Initial_system_r_interface} is given
in the following.
\begin{df}[Mixed-dimensional weak solution]\label{def:mixed_weak_solution}
Assume that~\textn{\Assum}  hold true. We say that
$(\vecu, p)\in  L^{2}(0,T;\mathbf{V})\times H^{1}(0,T;M)$ and
$(\vecu_{\gamma}, p_{\gamma})\in  L^{2}(0,T;\mathbf{V}_{\gamma})\times H^{1}(0,T;M_{\gamma})$ form a weak solution
of~\eqref{Initial_system_porous}--\eqref{Initial_system_r_interface}
if it satisfies \FIX{}{weakly the initial
conditions~\eqref{Initial_system_IC_p} and
\eqref{Initial_system_IC_f}}, and for each $i \in \{1,2\}$,
\bse \label{weak_mixed_formulation} \begin{alignat}{2}
\label{weak_mixed_formulation_m_d}&a_{i}(\vecu,\vecv)-b_{i}(\vecv,p)+(
p_{\gamma}, \vecv\cdot\vecn_{i})_{\gamma}=0\quad&&\forall\vecv\in \mathbf{V}_{i},\\
\label{weak_mixed_formulation_m_c}&c_{i}(\partial_{t}p,q)+b_{i}(\vecu,q)=(f,q)_{\Omega_{i}} \quad&&\forall q\in  M_{i},\\
\label{weak_mixed_formulation_f_f}& (\xi(\vecu_{\gamma}),\vecv)_{\gamma}+a_{\gamma}(\vecu_{\gamma},\vecv)
-b_{\gamma}( \vecv,p_{\gamma} ) =0\quad&& \forall\vecv\in \mathbf{V}_{\gamma},\\
\label{weak_mixed_formulation_f_c}&c_{\gamma}(\partial_{t}p_{\gamma},\mu)+b_{\gamma}( \vecu_{\gamma},\mu )-(\jump{\vecu \cdot \vecn}, \mu )_\gamma =
( f_{\gamma}, \mu )_\gamma \quad && \forall \mu\in M_{\gamma}.
\end{alignat}
\ese
\end{df}
\FIXX{}{
\begin{rem}
In this paper we assume that a weak solution of
Definition~\ref{def:mixed_weak_solution} exists. For the \FIX{static}{steady-state} model and\FIX{and when
$\xi$ stems from}{}
the classical Forchheimer's law,
the existence and uniqueness of a weak solution was
shown in~\cite{MR3264361}. That of the linear case, i.e., $\xi\eq0$, was
studied in~\cite{MR3457700}. Through the paper, we will also consider the case
of continuous pressure across $\gamma$ by letting
$\alpha_{\gamma}\rightarrow\infty$ in \eqref{Initial_system_r_interface}. For
this\FIX{case}{}, we will use Definition~\ref{def:mixed_weak_solution} for the \FIX{week}{weak}
formulation  \FIX{keeping in mind that}{changing} in~\eqref{bil:omegai} \FIX{$a_{i}$ is
simply}{to be} $a_{i}(\vecu,\vecv)\eq (\vK^{-1}\vecu,\vecv)_{\Omega_{i}}$ and
$\mathbf{V}_{i} \defeq \mathbf{H}(\textn{div}, \Omega_i)$ for $i\in\{1,2\}$.
\end{rem}
}

\FIX{\subsection{Goal and positioning of the paper}
The mixed-dimensional
problem~\eqref{Initial_system_porous}--\eqref{Initial_system_r_interface} is an
alternative to the possibility  to use   fine  grids  of  the  spatial
discretization  in the (physical)  fracture and thus reduces the computational
cost. This idea was introduced in~\cite{MR1911534} for highly permeable fractures
and in~\cite{MR2512496}  for fractures that may be much more permeable than the
surrounding medium or nearly impermeable~\cite{MR3264347,MR3448951,MR3307584}.
Particularly,   for  ``fast-path'' fractures,
Darcy's law is   replaced  by  the classical Darcy-Forchheimer's law as
established in~\cite{MR3264361}. We  also  refer
to~\cite{MR2776916,AHMED2019103431,MR3671645,MR3624734,list2018upscaling}  for  extensions  to other
flow models.  Here, we extend the model in~\cite{ahmed2018multiscale,MR3264361}
to unsteady non-Darcy flow generalized Forchheimer's law. {The work can be extended straightforwardly to viscosity models for generalized Newtonian fluids,
including the Power law, the Cross model and the Carreau model~\cite{MR3388812,MR2842139}.}

Considerable research efforts have been conducted to mixed-dimensional fracture
models.   Several numerical schemes for steady models have been proposed,  such
as a cell-centered finite volume scheme in~\cite{HAEGLAND20091740}, an  extended
finite element method  in~\cite{MR3631391}, a mimetic finite
difference~\cite{MR3507274} and a  block-centred finite difference method
in~\cite{MR3693346}. We also mention several contributions on the application of
mixed methods,   on conforming and non-conforming
grids~\cite{Martin2005,angelo_scotti_2012,Frih2012,MR3829517};
see~\cite{Berre2018}  for detailed  account of major contribution on fracture
models and discretization approaches.  The  aforementioned  numerical
approaches     solve \FIX{coupled}{}  fracture-matrix models  monolithically, which
leads to a large system, particularly \FIX{if}{for} mixed finite element (MFE)
methods\FIX{are
adopted~}{}\cite{MR3693346}. This is especially the case when incorporating
different \FIX{}{type of} equations\FIX{ varied in type}{}, \FIX{such
as}{like} coupling linear and non-linear systems, and   where often interface
conditions involve additional   variables. {Domain decomposition (DD) is an
elegant tool for modeling \FIX{such a}{} multi-physics problem\FIX{}{s} and can provide an
effective tool for reducing computational \FIX{complexity}{cost} and performing parallel
calculations.} See~\cite{MR3450068,MR1857663}\FIX{ for a
general introduction of the subject}{}. In~\cite{MR2386967}, the authors combine
\FIX{domain decomposition}{DD} techniques with \FIX{mixed finite element}{MFE} methods for the
reduced Darcy-Forchheimer fracture model (see~\cite{MR1911534,MR3457700} for the
linear case).

In this paper, we propose  efficient DD methods to
solve~\eqref{weak_mixed_formulation}  combining
the   mortar mixed finite element method
(MMFEM)~\cite{MR2557486,MR3869664,MR3765858} with non-overlapping domain
decomposition~\cite{ahmed2018multiscale,MR1911534,MR2434950} and  the $L$-scheme
method~\cite{MR2079503,MR3489128}. Our  method first
reformulates~\eqref{weak_mixed_formulation}
into   an  interface problem by eliminating the  subdomain variables. The
resulting  problem posed only on the fracture is a \textit{superposition} of a
\textit{non-linear local} flow operator   within the fracture and a
\textit{linear non-local} one handling the  flux contribution from the subdomains
(Robin-to-Neumann type operator).  {After approximating this problem with   the
MMFEM in space and the backward Euler scheme in time,  we obtain a non-linear system  to  solve at each time step. A first algorithm is
then built with the $L$-scheme employed as a linearization procedure; a robust quasi-Newton method with a parameter $L>0$ mimicking the Jacobian from the Newton method~\cite{MR2079503,MR3489128}.} At each iteration
of the $L$-scheme, an \textit{inner iterative
algorithm}, such as  GMRes or any Krylov solver, is  used to solve the linear
interface problem~\cite{MR3022024}. The action of the interface operator
requires solving subdomain problems with Robin boundary condition on the
fracture.  This algorithm referred to henceforth as the Monolithic LDD-scheme
(MoLDD) is  \textit{Jacobian-free} and   subdomain solves are  done in parallel.
This LDD scheme will be  shown to be \textit{unconditionally stable}. Stability
and condition number estimates of the inner DD system are obtained as well as
contraction estimates and rates of convergence for the outer scheme. While MoLDD offers an elegant  outer-inner approach  to
solve the interface-fracture problem, there is a computational overhead
associated with its non-local part (DD), see
\textit{e.g.}~\cite{ahmed2018multiscale,MR3022024}. Precisely,  the dominant
computational cost in this approach is measured by the number of subdomain
solves; increasing the non-linearity or DD strength   and refining the grids both
lead to an increase in the number of iterations and the number of subdomain
solves.

More recently, the $L$-scheme  has gained attention as an efficient
solver to  treat simultaneously  non-linear and coupling  effects in complex
problems. See for example \cite{MR3771899} for an application of the $L$-scheme
on a non-linear DD problem and~\cite{MR3827264} on a  non-linear coupling one.
Building further on this idea, we propose a   second algorithm, in which   the
DD  step is part of the linearization iterations (see~\cite{ahmed2019adaptive_2,brun2019monolithic} for related works). In other words,  the
$L$-scheme is now synchronizing  linearization  and domain decomposition
through one-loop algorithm.    This  approach  referred as the Iterative
LDD-scheme (ItLDD) differs  from  the  one  commonly  used  when  dealing  with
non-linear interface problems in the context of DD~\cite{ahmed2017posteriori,MR3261611}. At each iteration it has the
cost of the sequential approach, yet it converges to the fully monolithic
approach. This approach reduces the computational costs as no inner DD solver is
required  and  only a modest number  of subdomain solves  is  required at each
iteration, which still done in parallel. This algorithm increases
\textit{local to non-local cooperation}  and saves time if one process is dominating
the whole problem.

The second contribution of this paper concerns the robust and efficient
implementation of the LDD schemes above. Precisely, the dominant computational
costs in these schemes comes from the subdomain solves. To  reduce this
computational cost, we make use of the multiscale flux basis framework
from~\cite{MR2557486}. The fact that   the non-linearity
in~\eqref{weak_mixed_formulation}   is only
within the local operator on the fracture, we can adopt the notion that the
linear non-local contribution from the rock subdomains can be expressed  as a
 superposition of multiscale basis functions~\cite{MR2557486,MR3577939,ahmed2018multiscale}.
These multiscale flux basis consists of the flux (or velocity trace) response from each fracture pressure degrees of freedom.  They are
computed by  solving a \textit{fixed} number of \textit{steady} Robin subdomain
problems, which is equal to the number of fracture pressure degrees of freedom
per subdomain.  An inexpensive linear combination of the multiscale flux basis
functions  then replaces the  subdomain solves in any inner/outer iteration of
the  algorithms. This step of \textit{freezing}  the contributions on the flow
from the rock matrices can be  \textit{easily coded,  cheaply evaluated},  and
\textit{efficiently used} in  \textit{all the algorithms}.
\FIX{That is,  i}{I}t permits
reusing the same basis functions  to compare  MoLDD with  ItLDD  as well as to simulate various linear and non-linear
models for flow in the fracture by varying $\xi$ and finally  {exploring
high and low permeable fractures}. This is in total \textit{conformity} with the
spirit of \textit{reduced basis}~\cite{URBAN2012716, MR3644446,MR3853612}.
Crucially, if a fixed time step is used,  our multiscale flux basis applied to
a non-linear time-dependent problem are constructed only once
in the offline phase. This  should  be  kept  in  mind  also for our numerical
results reported in the last section. Numerical results are computed with the
library PorePy \cite{Keilegavlen2017a}.

\subsection{Outline of the paper}
\FIX{This paper is organized as follows: }{} \FIX{Firstly, t}{T}he   approximation of
problem~\eqref{weak_mixed_formulation}  using
the MMFEM in space  and  a backward Euler scheme in time is given in
Section~\ref{sec:problem_formulation}. Also, the reduction of this
mixed-dimensional scheme into a non-linear interface one is introduced.  The
LDD-schemes  are formulated in~\Cref{sec:algorithms}.
In~\Cref{sec:analysis_moldd} and~\Cref{sec:analysis_itldd}, the  analysis of the
schemes is presented. \Cref{sec:Mufbi} describes the implementation based on the
multiscale flux basis framework.  Finally, we \FIX{showcase}{show} the performance of  our methods on
several numerical examples in~\Cref{sec:examples} and draw the conclusions
in~\Cref{sec:conclusion}.
}
%

\FIX{\section{The DD formulation}}{}
\section{The domain decomposition formulation}
\label{sec:problem_formulation}
As explained earlier, it is natural to solve the mixed-dimensional
problem~\eqref{weak_mixed_formulation}  using
domain decomposition techniques, especially  as  these  methods  make it
possible  to  take  different   time  grids  in the  subdomains  and  in  the
fracture.

\subsection{Discretization in space and time}\label{subsec:Discretization}

We introduce in this section the partitions of $\Omega$ and $(0,T)$, basic
notation, and the  mortar mixed finite element discretization of the
mixed-dimensional
problem~\eqref{weak_mixed_formulation} .

Let $\Tau_{\hi}$ be a partition of the subdomain $\Omegai$ into either
$d$-dimensional simplicial or rectangular elements.   Moreover, we assume that
these meshes are such that $\Tau_{h}=\displaystyle{\cup_{i=1}^{2}\Tau_{\hi}}$
forms a conforming finite element mesh on  $\Omega$.  We also let
$\Tau_{\hg}$  be either a partition of the fracture $\gamma$ induced
by~$\Tau_{h}$ or slightly coarser.  Denote $h$  as the maximal
mesh size of both  $\Tau_{h}$ and $\Tau_{\hg}$.
For an integer $N\geq 0$, let
$\left(\tau^{n}\right)_{0\leq n\leq N}$ denote a sequence of positive real
numbers corresponding to the discrete time steps such that
$T=\sum_{n=1}^{N}\tau^{n}$. Let $t^{0}\eq0$, and $t^{n}\eq\sum_{j=1}^{n}\tau^{j}, \
1 \leq n\leq N$, be the discrete times. Let $I^n\eq(t^{n-1},t^n], \ 1 \le n \le
N$.

\FIX{\subsubsection{Finite-dimensional spaces and projection operators}}{}
For  the approximation of scalar unknowns, we
introduce \FIX{the   approximation spaces}{}
$M_{h}\eq M_{h,1}\times M_{h,2}$ and $M_{\hg}$, where $M_{\hi}$, $i\in\{1,2\}$, and
$M_{\hg}$ are the spaces of piecewise constant functions associated with $\Tau_{\hi}$,
$i\in\{1,2\}$ and $\Tau_{\hg}$, respectively. For  the vector unknowns, we
introduce \FIX{the   approximation spaces}{}
$\mathbf{V}_{h}\eq \mathbf{V}_{h,1}\times \mathbf{V}_{h,2}$ and $\mathbf{V}_{\hg}$, where
$\mathbf{V}_{\hi}$, $i\in\{1,2\}$ and
$\mathbf{V}_{\hg}$, are the lowest-order Raviart-Thomas-N\'{e}d\'{e}lec finite elements spaces
associated with  $\Tau_{\hi}$,  $i\in\{1,2\}$ and $\Tau_{\hg}$, respectively.
Thus, $\mathbf{V}_{h}\times M_{h}\subset \mathbf{V}\times M$ and
$\mathbf{V}_{\hg}\times M_{\hg}\subset \mathbf{V}_{\gamma}\times M_{\gamma}$. For all of the above spaces,
\begin{equation}\label{conservation}
 \nabla\cdot \mathbf{V}_{h}=M_{h},\quad\textnormal{ and }\quad\nabla_{\tau}\cdot \mathbf{V}_{\hg}=M_{\hg},
\end{equation}
and there exists a  projection
$\tilde{\Pi}_{i}:\mathbf{H}^{1/2+\epsilon}(\Omegai)\cap\mathbf{V}_{i}\rightarrow\mathbf{V}_{\hi}$, $i\in\{1,2\}$,
for any $\epsilon>0$, satisfying among other properties \cite{MR3577939} that for any $\vecu\in \mathbf{H}^{1/2+\epsilon}(\Omegai)\cap\mathbf{V}_{i}$
\begin{align}\label{properties_Pi}
&(\nabla\cdot (\vecu-\tilde{\Pi}_{i}\vecu),q)_{\Omegai} = 0 &&\forall q\in M_{\hi},\\
&((\vecu-\tilde{\Pi}_{i}\vecu)\cdot\vecn_{i},\vecv\cdot\vecn_{i})_{\partial\Omegai
}  = 0 &&\forall \vecv\in \mathbf{V}_{\hi}.
\end{align}
We also note that if $\vecu\in\mathbf{H}^{\epsilon}(\Omegai)\cap\mathbf{V}_{i}$, $0< \epsilon <1$, $\tilde{\Pi}_{i}\vecu$ is well-defined~\cite{MR2842713} and
\begin{alignat}{2}\label{estimate_regularity_Pi}
 &||\tilde{\Pi}_{i}\vecu||_{\Omegai}\lesssim ||\vecu||_{\epsilon,\Omegai}+||\nabla\cdot\vecu||_{\Omegai}.
\end{alignat}
We introduce  $\mathcal{Q}_{\hi}$ the $L^{2}$-projection onto $\mathbf{V}_{\hi}\cdot\vecn_{i}$
 and denote $\mathcal{Q}_{\hi}^{\textn{T}}:\mathbf{V}_{\hi}\cdot\vecn_{i}\rightarrow M_{\hg}$
as the $L^{2}$-projection  from  the  normal
velocity  trace on the subdomains onto  the  mortar space $M_{\hg}$. Thus, for
all $\lambda\in M_{\hg}$ \FIX{the condition}{}
\begin{equation}\label{solvability}
|| \lambda||_{\gamma}\lesssim || \mathcal{Q}_{h,1}\lambda||_{\gamma}+|| \mathcal{Q}_{h,2}\lambda||_{\gamma},
\end{equation}
can \FIX{easily}{} be verified  if the mesh on the fracture $\Tau_{\hg}$ matches the one
resulting from the surrounding subdomains, or if $\Tau_{\hg}$ is chosen
slightly coarser~\cite{MR3829517,MR2306414}. Note that~\eqref{conservation} can
be   satisfied by choosing any of the usual MFE pairs.
The
condition~\eqref{solvability} can be  satisfied even if  the space $M_{\hg}$ is not much richer  than  the  space  of  normal  traces  on $\gamma$ of elements of $\mathbf{V}_{h}$~\cite{MR3577939,MR2557486}.\FIX{\subsubsection{The discrete scheme}}{}
The  fully  discrete scheme of the mixed-dimensional
formulation~\eqref{weak_mixed_formulation}
based on  the MMFEM in space and the backward Euler scheme in time is defined
through the following.
\begin{df}[The mixed-dimensional scheme]\label{Def:discrete_scheme}
 At each time step $n\geq1$,  assuming $(p^{n-1}_{\hg},p^{n-1}_{h})$ is given,
 we look for  $(\vecu^{n}_{h}, p^{n}_{h})\in \mathbf{V}_{h}\times M_{h}$ and
 $(\vecu^{n}_{\hg}, p^{n}_{\hg})\in \mathbf{V}_{\hg}\times M_{\hg}$ such that, for $i\in\{1,2\},$
 \bse\label{md_discrete_system_fracture}\begin{alignat}{4}
  \label{md_semi_system_d_p}&a_{i}(\vecu^{n}_{h},\vecv)-b_{i}(\vecv,p^{n}_{h})=-(
p^{n}_{\hg}, \vecv\cdot\vecn_{i})_{\gamma}  &&\quad \forall  \vecv\in \mathbf{V}_{h}.&&\\
  \label{md_semi_system_c_p}&c_{i}(p^{n}_{h}-p^{n-1}_{h},q)+\tau^{n}b_{i}(\vecu^{n}_{h},q)=\tau^{n}(f^{n},\mu)_{\Omegai} &&\quad \forall  q\in M_{h},&&\\
   \label{md_semi_system_d_f}&(\xi(\vecu^{n}_{\hg}),\vecv)+a_{\gamma}(\vecu^{n}_{\hg},\vecv)-b_{\gamma}(\vecv,p^{n}_{\hg})=0  &&\quad \forall  \vecv\in \mathbf{V}_{\hg},&&\\
 \label{md_semi_system_c_f}&c_{\gamma}(p^{n}_{\hg}-p^{n-1}_{\hg},\mu)+\tau^{n}b_{\gamma}(\vecu^{n}_{\hg},\mu)-
 \tau^{n}(\jump{\vecu_{h}^{n}\cdot\vecn},\mu)_{\gamma}=\tau^{n}(f_{\gamma}^{n},\mu)_\gamma &&\quad \forall  \mu\in M_{\hg}.&
 \end{alignat}\ese
 \end{df}

\subsection{Reduction into an interface problem}
Following the algorithm in~\cite{ahmed2018multiscale}, we reduce the
mixed-dimensional scheme in Definition~\ref{Def:discrete_scheme} to a
non-linear interface one\FIX{posed}{} on $\gamma$.\FIX{which can be solved using an appropriate combination of a linearization method and an iterative Krylov solver. }{} For $i\in\{1,2\}$, we let
\begin{alignat}{2}\label{SolSub:decomposition}
 p^{n}_{\hi}=p^{*}_{\hi}(\lambda^{n}_{\hg})+
 \bar{p}^{n}_{\hi}\quad\text{and}\quad
 \vecu^{n}_{\hi}={\vecu}^{*}_{\hi}(\lambda^{n}_{\hg})+
 \bar{\vecu}^{n}_{\hi},\quad 1\leq n\leq N,
\end{alignat}
where for $\lambda^{n}_{\hg}\in M_{\hg}$,
$(\vecu^{*}_{\hi}(\lambda^{n}_{\hg}), p^{*}_{\hi}(\lambda^{n}_{\hg}))\in
 \mathbf{V}_{\hi}\times M_{\hi}$  solves
\bse \label{subdo1_weak_mixed_formulation_disc} \begin{align}
    &a_{i}(\vecu^{*}_{\hi}(\lambda^{n}_{\hg}),\vecv)-b_{i}(\vecv,p^{*}_{\hi}(\lambda^{n}_{\hg})) =-(
    \lambda_{\hg}^{n}, \vecv\cdot\vecn_{i})_{\gamma} && \forall\vecv\in
    \mathbf{V}_{\hi},\\
    &c_{i}(p^{*}_{\hi}(\lambda^{n}_{\hg}),q)+\tau^{n}b_{i}(\vecu^{*}_{\hi}(\lambda^{n}_{\hg}),q)
    = 0 && \forall q\in  M_{\hi},
 \end{align}
\ese
and $(\bar{\vecu}^{n}_{\hi}, \bar{p}^{n}_{\hi})\in \mathbf{V}_{\hi} \times M_{\hi}$ solves
\bse \label{subdo2_weak_mixed_formulation_disc} \begin{align}
&a_{i}(\bar{\vecu}^{n}_{i},\vecv)-b_{i}(\vecv,\bar{p}^{n}_{\hi}) =0 && \forall\vecv\in \mathbf{V}_{\hi},\\
&c_{i}(\bar{p}^{n}_{\hi}-p^{n-1}_{\hi},q)+\tau^{n}b_{i}(\bar{\vecu}^{n}_{\hi},q)
=\tau^{n}(f^{n},\mu)_{\Omega_{i}} && \forall q\in  M_{\hi},\\
&(\bar{p}^{0}_{\hi},\mu)_{\Omegai} = (p_{\hi}^{0},\mu)_{\Omegai} && \forall \mu\in  M_{\hi}.
\end{align}
\ese
Define the forms
$s_{\gamma,i}:M_{\hg}\times M_{\hg}\rightarrow\RR$, $i\in\{1,2\}$,
$s_{\gamma}:M_{\hg}\times M_{\hg}\rightarrow\RR$, and
$g^{n}_{\gamma}: M_{\hg}\rightarrow\RR$ as
\bse\label{bilinear_forms_discrete}\begin{align}
s_{\gamma,i}(\lambda_{\hg}^{n},\mu) & \eq
 (\mathcal{S}^{\textn{RtN}}_{\gamma,i}(\lambda_{\hg}^{n}),\mu)_{\gamma}\eq -(\vecu_{\hi}^{*}(\lambda_{\hg}^{n})\cdot\vecn_{i},\mu)_{\gamma},\\
s_{\gamma}(\lambda_{\hg}^{n},\mu) & \eq(\mathcal{S}^{\textn{RtN}}_{\gamma}(\lambda_{\hg}^{n}),\mu)_{\gamma}\eq\sum_{i=1}^{2}s_{\gamma,i}(\lambda_{\hg}^{n},\mu),\\
g^{n}_{\gamma}(\mu) & \eq(g_{\gamma}^{n},\mu)_{\gamma}\eq \sum_{i=1}^{2}(\bar{\vecu}^{n}_{\hi}\cdot\vecn_{i},\mu)_{\gamma},
\end{align}\ese
where $\mathcal{S}^{\textn{RtN}}_{\gamma,i}:M_{\hg}
\rightarrow M_{\hg}$, $1\leq i\leq 2$, and
$\mathcal{S}^{\textn{RtN}}_{\gamma}\eq\sum_{i=1}^{2}\mathcal{S}^{\textn{RtN}}_{\gamma,i}$
are   Robin-to-Neumann type operators.
\FIX{Obviously}{Consequently}, the  operator
$\mathcal{S}^{\textn{RtN}}_{\gamma,i}$
 is linear. 
It is \FIX{easy}{possible} to verify that the non-linear mixed-dimensional
scheme~\eqref{md_discrete_system_fracture} is equivalent to the  non-linear interface scheme.
 \begin{df}[The  reduced scheme]\label{def:reduced scheme}
\FIX{Given}{For} $n\geq1$ and  $\lambda^{n-1}_{\hg}$,
find $(\vecu^{n}_{\hg}, \lambda^{n}_{\hg})\in \mathbf{V}_{\hg}\times M_{\hg}$
such that%
 \bse\label{discrete_system_fracture}\begin{alignat}{4}
 \label{semi_system_d_f}&(\xi(\vecu^{n}_{\hg}),\vecv)_{\gamma}+a_{\gamma}(\vecu^{n}_{\hg},\vecv)-b_{\gamma}(\vecv,\lambda^{n}_{\hg})=0
 &&\quad \forall  \vecv\in \mathbf{V}_{\hg},\\
 \label{semi_system_c_f}&c_{\gamma}(\lambda^{n}_{\hg}-\lambda^{n-1}_{\hg},\mu)+\tau^{n}b_{\gamma}(\vecu^{n}_{\hg},\mu)+
 \tau^{n}s_{\gamma}(\lambda^{n}_{\hg},\mu)=\tau^{n}(f_{\gamma}^{n}+g^{n}_{\gamma},\mu)_{\gamma}
 &&\quad \forall  \mu\in M_{\hg}.&
 \end{alignat}\ese
\end{df}

\section{Robust L-type Domain-Decomposition (LDD) schemes}
\label{sec:algorithms}
\FIX{}{In this section, we propose two  iterative approaches  based on the
$L$-scheme to solve~\eqref{discrete_system_fracture}. The first approach entails an inner-outer
procedure of the form $\textit{linearize}\rightarrow \textit{solve the
    DD}\rightarrow update$, so that the $L$-scheme is used for the outer loop and
an inner solver (direct or iterative) for the inner loop. The second approach is a
one-loop procedure in which   the $L$-scheme acts  iteratively and
simultaneously on the linearization and DD.} For the  presentation of the algorithms, we shall denote the time step simply by $\tau$, keeping in mind it may depend on $n$.
\subsection{A monolithic LDD scheme}\label{subsec:monolithic}
  The monolithic LDD scheme (MoLDD)  \FIX{used to solve  the interface
  problem}{for}~\eqref{discrete_system_fracture} reads:
 \begin{algo}[The MoLDD scheme]~\label{monolithic_ldd}
Given \FIX{the initial data}{$ n=0 $}, $(\lambda^{0}_{\hg},p^{0}_{h})\in M_{\hg}\times M_{h}$, stabilization parameter $L_{\gamma}>0$ and tolerance $\epsilon>0$,

\textn{\textbf{Do}}
{\setlist[enumerate]{topsep=0pt,itemsep=-1ex,partopsep=1ex,parsep=1ex,leftmargin=1.5\parindent,font=\upshape}
\begin{enumerate}
 \item Increase $n\eq n+1$.
\item Choose an initial approximation $\vecu^{n,-1}_{\hg}\in \mathbf{V}_{\hg}$ of $\vecu^{n}_{\hg}$. Set $k\eq-1$.
\item \textn{\textbf{Do}}
\begin{enumerate}
 \item Increase $k\eq k+1$.
\item Compute
$(\vecu^{n,k}_{\hg}, \lambda^{n,k}_{\hg})\in   \mathbf{V}_{\hg}\times M_{\hg}$
such that, for all $(\vecv, \mu)\in \mathbf{V}_{\hg}\times M_{\hg}$,
\bse\label{lscheme_system_fracture}\begin{alignat}{4}
 \label{lscheme_system_d_f}&(\xi(\vecu^{n,k-1}_{\hg})+L_{\gamma}(\vecu^{n,k}_{\hg}-\vecu^{n,k-1}_{\hg}),\vecv)_{\gamma}+a_{\gamma}(\vecu^{n,k}_{\hg},\vecv)
 -b_{\gamma}(\vecv,\lambda^{n,k}_{\hg})=0,\\
\label{lscheme_system_c_f}
&c_{\gamma}(\lambda^{n,k}_{\hg}-\lambda^{n-1}_{\hg},\mu)+\tau b_{\gamma}(\vecu^{n,k}_{\hg},\mu)+ \tau
s_{\gamma}(\lambda^{n,k}_{\hg},\mu)=\tau (f_{\gamma}^{n}+g^{n}_{\gamma},\mu)_{\gamma}.
\end{alignat}\ese
\end{enumerate}
\textn{\textbf{while}}
${\|(\vecu^{n,k}_{\hg}, \lambda^{n,k}_{\hg}) -
(\vecu^{n,k-1}_{\hg}, \lambda^{n,k-1}_{\hg})\|_{\gamma}}\geq \epsilon\|(\vecu^{n,k-1}_{\hg}, \lambda^{n,k-1}_{\hg})\|_{\gamma}$.
\item Update the subdomain solutions  via~\eqref{SolSub:decomposition}.
\end{enumerate}
 \textn{\textbf{while}} $n\leq N$.
}
\end{algo}
\FIX{\begin{rem}
\end{rem}}{The  advantages of Algorithm~\ref{monolithic_ldd} are multiple:  (i) the
algorithm is Jacobian-free and independent of the initialization,
     (ii) we can reuse the existing $d$- and $(d-1)$-dimensional
codes for solving linear Darcy problem, and (iii) optimal convergence rate is obtained
with a stabilization amount determined\FIX{ efficiently }{} through $L_{\gamma}$.}

The  MoLDD scheme involves the solution  of a linear  Darcy interface
problem~\eqref{lscheme_system_fracture} at each iteration\FIX{ $k\geq0$}{}.
\FIX{To see that,
w}{W}e introduce the linear operators
$\mathbf{A}_{L,\gamma}: \mathbf{V}_{\hg}\rightarrow \mathbf{V}_{\hg}$ and
$\mathbf{B}_{\gamma}: \mathbf{V}_{\hg}\rightarrow M_{\hg}$, defined as
$(\mathbf{A}_{L,\gamma}\vecu,\vecv)_{\gamma}\eq a_{\gamma}(\vecu,\vecv)+L_{\gamma}(\vecu,\vecv)_{\gamma}$,
$\forall \vecu,\vecv\in \mathbf{V}_{\hg}$,
and $(\mathbf{B}_{\gamma}\vecu, q)\eq b_{\gamma}(\vecu,q)$, $\forall \vecv\in \mathbf{V}_{\hg}$, $\forall q\in M_{\hg}$.
\FIX{Now~}{}\eqref{lscheme_system_fracture} becomes
\begin{align}\label{compact_system}
\mathcal{A}_{\textn{DD}}\begin{bmatrix}
 \vecu^{n,k}_{\hg}\\[2mm]
 \lambda^{n,k}_{\hg}\\[2mm]
 \end{bmatrix}:=
\begin{bmatrix}
\mathbf{A}_{L,\gamma}& \mathbf{B}_{\gamma}^{\textn{T}}\\[2mm]
 \mathbf{B}_{\gamma}& \mathcal{S}^{\textn{RtN}}_{\gamma}+ \mathbf{I}/\tau\\[2mm]
\end{bmatrix}\begin{bmatrix}
 \vecu^{n,k}_{\hg}\\[2mm]
 \lambda^{n,k}_{\hg}\\[2mm]
 \end{bmatrix} =  \begin{bmatrix}
 L_{\gamma}\vecu^{n,k-1}_{\hg}-\xi(\vecu^{n,k-1}_{\hg})\\[2mm]
 g^{n}_{\gamma}+f^{n}_{\gamma}+ \lambda^{n-1}_{\hg} / \tau
\end{bmatrix}:=\mathcal{F}_{\gamma},
\end{align}
which can \FIX{}{be solved using a direct or a Krylov type method, such as GMRes or MINRes.}
\FIX{At each GMRes iteration $m\geq1$}{Regardless of the choice of inner method}, we need to evaluate
the action of the Robin-to-Neumann type operator $ \mathcal{S}^{\textn{RtN}}_{\gamma}$
via \eqref{bilinear_forms_discrete}, representing physically the flow
contributions from the subdomains by
solving Robin subdomain problems~\eqref{subdo1_weak_mixed_formulation_disc}.
We  summarize the evaluation of the interface operator by the following steps:
\begin{algo}[Evaluating  the action of $ \mathcal{S}^{\textn{RtN}}_{\gamma}$]\label{Eval_op}~
{
\setlist[enumerate]{topsep=0pt,itemsep=-1ex,partopsep=1ex,parsep=1ex,leftmargin=1.5\parindent,font=\upshape}
\begin{enumerate}
    \item Enter interface data $\lambda_{\hg}$.
    \item \textn{\textbf{For}} $i=1:2$
    \begin{enumerate}
        \item Project mortar pressure onto subdomain boundary, i.e.,
        $\varphi_{\hg,i}=\mathcal{Q}_{\hi}(\lambda_{\hg}).$
        \item Solve the subdomain problem~\eqref{subdo1_weak_mixed_formulation_disc}
        with Robin data~$\varphi_{\hg,i}$.
        \item Project  the resulting  flux onto  the  space $M_{\hg}$, i.e.,
        $\mathcal{S}^{\textn{RtN}}_{\gamma,i}(\lambda_{\hg})=
        -\mathcal{Q}_{\hi}^{\textn{T}}\vecu^{*}_{\hi}(\varphi_{\hg,i})\cdot\vecn_{i}.$
    \end{enumerate}
    \textn{\textbf{EndFor}}
    \item Compute the flow contribution from the subdomains to \FIX{the
    fracture}{$\gamma$} given by the flux jump across the fracture,
    $$\mathcal{S}^{\textn{RtN}}_{\gamma}(\lambda_{\hg})=\displaystyle\sum_{i\in\{1,2\}}\mathcal{S}^{\textn{RtN}}_{\gamma,i}(\lambda_{\hg}).$$
\end{enumerate}}
\end{algo} 
\vskip -2mm
\FIX{\begin{rem}
\end{rem}}{The evaluation of  $\mathcal{S}^{\textn{RtN}}_{\gamma}$  dominates the total
computational costs in Algorithm~\ref{monolithic_ldd} (step 2(b) of
Algorithm~\ref{Eval_op}).  The number of subdomain solves required by this
method at each time step $n\geq 1$ is approximately equal to
$\sum_{k=1}^{N^{n}_{\textnormal{Lin}}} N^{k}_{\textnormal{DD}}$, where
$N_{\textnormal{Lin}}$ is the number of  iterations of the $L$-scheme, and
$N^{k}_{\textnormal{DD}}$ denotes the number of inner DD iterations. To set up the
right-hand side term $f_{\gamma}^{n}$, we  need to solve once in the subdomains at
each time step $n\geq 1$.}
\subsection{An \FIX{ robust }{}iterative LDD-scheme}
An alternative \FIX{   LDD-scheme to solve  the interface
problem~\eqref{discrete_system_fracture}}{to MoLDD scheme} is to let
the $L$-scheme  act iteratively not only on the non-linearity\FIX{as in Algorithm~\ref{monolithic_ldd}}{},
but also on the fracture-matrix coupling. Additional stabilization term is then required for
the inter-dimensional coupling. 
\begin{algo}[The ItLDD scheme]~\label{splitting_ldd}
Given \FIX{}{n=0}, $(\lambda^{0}_{\hg},p^{0}_{h})\in M_{\hg}\times M_{h}$, the stabilization parameters $(L_{\gamma,p},L_{\gamma,u})>0$, and the tolerance $\epsilon>0$.

\textn{\textbf{Do}}
{\setlist[enumerate]{topsep=0pt,itemsep=-1ex,partopsep=1ex,parsep=1ex,leftmargin=1.5\parindent,font=\upshape}}
\begin{enumerate}
 \item Increase $n\eq n+1$.
\item Choose an initial approximation $(\vecu^{n,-1}_{\hg}, \lambda^{n,-1}_{\hg})\in  \mathbf{V}_{\hg}\times M_{\hg}$ of $(\vecu^{n}_{\hg}, \lambda^{n}_{\hg})$. Set $k\eq-1$.
\item \textn{\textbf{Do}}
\begin{enumerate}
 \item Increase $k\eq k+1$.
\item Compute
$(\vecu^{n,k}_{\hg}, \lambda^{n,k}_{\hg})\in  \mathbf{V}_{\hg}\times M_{\hg}$
such that, for all $(\vecv, \mu)\in \mathbf{V}_{\hg}\times M_{\hg}$,
\bse\label{lddscheme_system_fracture}\begin{alignat}{4}
\label{lddscheme_system_d_f} &(\xi(\vecu^{n,k-1}_{\hg}) +
L_{\gamma,u}(\vecu^{n,k}_{\hg}-\vecu^{n,k-1}_{\hg}),\vecv)_{\gamma}
 +a_{\gamma}(\vecu^{n,k}_{\hg},\vecv)
 -b_{\gamma}(\vecv,\lambda^{n,k}_{\hg})=0.\\
\nonumber &c_{\gamma}(\lambda^{n,k}_{\hg}-\lambda^{n-1}_{\hg},\mu)+\tau
L_{\gamma,p}(\lambda^{n,k}_{\hg}-\lambda^{n,k-1}_{\hg},\mu)_{\gamma}
+ \tau
s_{\gamma}(\lambda^{n,k-1}_{\hg},\mu)\\
\label{lddscheme_system_c_f}&\qquad\qquad\qquad\qquad\qquad\qquad\qquad\qquad\quad
+\tau b_{\gamma}(\vecu^{n,k}_{\hg},\mu)=\tau (f_{\gamma}^{n}+g^{n}_{\gamma},\mu)_{\gamma},
 \end{alignat}\ese
\end{enumerate}
\textn{\textbf{while}} ${\|(\vecu^{n,k}_{\hg}, \lambda^{n,k}_{\hg}) -
(\vecu^{n,k-1}_{\hg}, \lambda^{n,k-1}_{\hg})\|_{\gamma}}\geq \epsilon{\|(\vecu^{n,k-1}_{\hg}, \lambda^{n,k-1}_{\hg})\|_{\gamma}}$.
\item Update the subdomain solutions  via~\eqref{SolSub:decomposition}.
\end{enumerate}
\textn{\textbf{while}} $n\leq N$.
\end{algo} 
\FIX{The linear problem~\eqref{lddscheme_system_fracture} is solved with the GMRes iterations~\eqref{gmres_pseudocode}. It requires at  each iteration $k\geq 1$ only one solve per subdomain to evaluate the action of $ \mathcal{S}^{\textn{RtN}}_{\gamma}$ via Algorithm~\ref{Eval_op} at the previous iteration, and this at each time step $n\geq 1$.}{The linear problem~\eqref{lddscheme_system_fracture} can again be solved by a direct or iterative method, and it requires applying the operator $ \mathcal{S}^{\textn{RtN}}_{\gamma}$ via Algorithm~\ref{Eval_op}, at each time step $n\geq 1$.} 
\FIX{\begin{rem}[Advantages of ItLDD-scheme]\end{rem}}{The advantages of the Algorithm~\ref{splitting_ldd} are:
 (i) at each iteration $k\geq 1$,  the systems in the fracture and the rock
 matrices cooperate sequentially  in one loop \FIX{and}{}
    (ii) optimal convergence rate is obtained with
 precise stabilization parameters $(L_{\gamma,p},L_{\gamma,u})$, and (iii) existing codes for $d$- and $(d-1)$-dimensional
 Darcy  problems can be cheaply reused\FIX{ for practical
 simulations}{}.}


%



\section{Analysis of  MoLDD-scheme}\label{sec:analysis_moldd}
The complete analysis of~Algorithm~\ref{monolithic_ldd} will be carried out in
two steps: (i) we first study  the stability of the iterate DD scheme (inner
solver) and  estimate the condition number, and (ii) we prove the convergence of
the LDD scheme (outer solver), show  the well-posedness of the discrete scheme,
estimate the  convergence rate and subsequently determine the optimal
stabilization parameter.\FIX{Throughout the paper, we will frequently use the
standard identity
\begin{equation}
 \label{identity}
 (a-b) \cdot a=\dfrac{1}{2}\left(a^{2}-b^{2}+(a-b)^{2}\right),\quad a,b\in\RR,
\end{equation}
and inequality
\begin{equation}
 \label{young}
 |ab|\leq \dfrac{1}{2\delta}a^2+\dfrac{2}{\delta}b^2,\quad a,b,\delta\in\RR,\,\delta>0.
\end{equation}}{}
A key point in the analysis of the methods below
 are  inverse inequalities.
\begin{lem}[Inverse inequalities]
There exist positive constants $C_{\textn{dTr}},C_{\textn{inv}}>0$ depending only on the shape regularity of the mesh such that
 \begin{alignat}{2}\label{dtrace_ineq}
 &||\vecu_{h}\cdot\vecn||_{\partial\Omegai}\leq C_{\textn{dTr}} h^{-1/2}||\vecu_{h}||_{\Omegai} \qquad &&\forall\vecu_{h}\in\mathbf{V}_{\hi},\\
 \label{inv_ineq}&||\nabla_{\tau}\cdot\vecu_{\hg}||_{\gamma}\leq C_{\textn{inv}} h^{-1}||\vecu_{\hg}||_{\gamma} &&\forall\vecu_{\hg}\in\mathbf{V}_{\hg}.
\end{alignat}
\end{lem}

\subsection{Analysis of the  DD step}
To simplify\FIX{ the analysis}{}, we rewrite \FIX{problem~}{}\eqref{lscheme_system_fracture}:
find $(\vecu^{n,k}_{\hg}, \lambda^{n,k}_{\hg})$ $\in\mathbf{V}_{\hg}\times M_{\hg}$ so that
\begin{alignat}{2}\label{compact_form_lscheme}
 \mathcal{A}_{\gamma}((\vecu^{n,k}_{\hg}, \lambda^{n,k}_{\hg}),(\vecv, \mu))+s_{\gamma}(\lambda^{n,k}_{\hg},\mu)=
 \mathcal{F}^{n,k-1}_{\gamma}(\vecv, \mu) \qquad \forall (\vecv, \mu)\in \mathbf{V}_{\hg} \times M_{\hg},
\end{alignat}
where \FIX{}{$\mathcal{A}_{\gamma}$ is the linearized flow system on the fracture and $s_{\gamma}$ is the flow contribution from the rock matrices}
\vskip -5mm
\bse\begin{alignat}{2}
\label{bilinear:A}&\mathcal{A}_{\gamma}((\vecu_{\hg}, \lambda_{\hg}),(\vecv, \mu))\eq a_{\gamma}(\vecu_{\hg},\vecv)+
 L_{\gamma}(\vecu_{\hg},\vecv)_{\gamma}+
 \dfrac{1}{\tau }(\lambda_{\hg},\mu)_{\gamma}
 +b_{\gamma}(\vecu_{\hg},\mu)
-b_{\gamma}(\vecv,\lambda_{\hg}),\\
\label{linear:F}&\mathcal{F}^{n,k-1}_{\gamma}(\vecv,
\mu)\eq
(\xi(\vecu^{n,k-1}_{\hg}) +
L_{\gamma}\vecu^{n,k-1}_{\hg},\vecv)_{\gamma}+
(f_{\gamma}^{n}+g^{n}_{\gamma},\mu)_{\gamma}.
\end{alignat}\ese
The first result concerns the properties of the coupling term  $s_{\gamma}$.
\begin{lem}[Properties of the DD operator]\label{lem:bound_on_s}
 The interface bilinear form
 $s_{\gamma}$
 \FIX{\begin{itemize}
  \item
  $s_{\gamma}$ is symmetric positive  and
  semi-definite on $L^{2}(\gamma)$.
  \item There exists a constant $C_{1}>0$ independent of $h$  such that,
for all $\lambda_{\hg}\in M_{\hg}$,
 \begin{equation}\label{def_inter_oper_disc}
    {\left(C_{1}\dfrac{C_{\vK}}{\sqrt{c_{\vK}}}+\dfrac{1}{\sqrt{\alpha_{\gamma}}}\right)^{-2}}||\lambda_{\hg}||_{\gamma}^{2}\leq
    s_{\gamma}(\lambda_{\hg},\lambda_{\hg})\leq  \alpha_{\gamma}||\lambda_{\hg}||_{\gamma}^{2}.
 \end{equation}

  \end{itemize}}{is symmetric positive  and
    semi-definite on $L^{2}(\gamma)$, and there exists a constant $C_{1}>0$ independent of $h$  such that,
    for all $\lambda_{\hg}\in M_{\hg}$,
    \begin{equation}\label{def_inter_oper_disc}
    {\left(C_{1}\dfrac{C_{\vK}}{\sqrt{c_{\vK}}}+\dfrac{1}{\sqrt{\alpha_{\gamma}}}\right)^{-2}}||\lambda_{\hg}||_{\gamma}^{2}\leq
    s_{\gamma}(\lambda_{\hg},\lambda_{\hg})\leq  \alpha_{\gamma}||\lambda_{\hg}||_{\gamma}^{2}.
    \end{equation}
}
\end{lem}
\begin{proof}
 Recalling~\eqref{bilinear_forms_discrete} and
 taking $\vecv=\vecu^{*}_{\hi}(\mu)$ and $ q = p^{*}_{\hi}(\mu) $
 in~\eqref{subdo1_weak_mixed_formulation_disc}\FIX{ to see that the bilinear
 form}{, }
 $s_{\gamma}$ can be expressed as
 \begin{alignat}{2}\label{def:s_g}
    s_{\gamma}(\lambda_{\hg},\mu)=\sum_{i=1}^{2}
    \{a_{i}(\vecu^{*}_{\hi}(\lambda_{\hg}),\vecu^{*}_{\hi}(\mu))+
    c_{i}(p^{*}_{\hi}(\lambda_{\hg}),p^{*}_{\hi}(\mu))\}.
  \end{alignat}
  It is now easy to see that the bilinear form
 $s_{\gamma}$ is  symmetric and positive
 semi-definite on $L^{2}(\gamma)$.   We now show that if
$s_{\gamma}(\lambda_{\hg},\lambda_{\hg})=0$, then $\lambda_{\hg}=0$ on $M_{\hg}$.  Note that
$s_{\gamma}(\lambda_{\hg},\lambda_{\hg})=0$ implies that
$\vecu^{*}_{\hi}(\lambda_{\hg})=p^{*}_{\hi}(\lambda_{\hg})=0$.
Again,~\eqref{subdo1_weak_mixed_formulation_disc} implies $(\mathcal{Q}_{\hi}
\lambda_{\hg},\vecv\cdot\vecn_{i})_{\gamma}=(
\lambda_{\hg},\vecv\cdot\vecn_{i})_{\gamma}=0$ for any $\vecv\in\mathbf{V}_{\hi}$.
Thus, we can find some $\vecv$ so that $\vecv\cdot\vecn_{i}=\mathcal{Q}_{\hi}
\lambda_{\hg}$ and then $||\mathcal{Q}_{\hi} \lambda_{\hg}||_{\gamma}=0$.
Finally,~\eqref{solvability} shows that $\lambda_{\hg}=0$ on $\gamma$.
We now infer the upper bound on
 $s_{\gamma}$. The assumption~\ref{ass_K_omega} directly implies
  \begin{equation}\label{coer:a_i}
   c_{\vK}||\vecu_{\hi}||_{\Omegai}^{2}+{\alpha_{\gamma}^{-1}}||\vecu_{\hi}\cdot\vecn_{i}||_{\gamma}^{2}\leq a_{i}(\vecu_{\hi},\vecu_{\hi}),\quad\forall \vecu_{\hi} \in \mathbf{V}_{h,i}.
  \end{equation}
 The definition~\eqref{bilinear_forms_discrete}
 of $s_{\gamma}$ gives
\FIX{\begin{alignat}{2}
    \nonumber& s_{\gamma}(\lambda_{\hg},\lambda_{\hg})&&=-\sum_{i=1}^{2}(
    \lambda_{\hg}, \vecu^{*}_{\hi}(\lambda_{\hg})\cdot\vecn_{i})_{\gamma}\leq \sum_{i=1}^{2}||\vecu^{*}_{\hi}(\lambda_{\hg})\cdot\vecn_{i}||_{\gamma}||\lambda_{\hg}||_{\gamma}\\
    &&& \leq \sum_{i=1}^{2}\alpha_{\gamma}^{1/2}a_{i}(\vecu^{*}_{\hi}(\lambda_{\hg}),\vecu^{*}_{\hi}(\lambda_{\hg}))^{1/2}||\lambda_{\hg}||_{\gamma}.
    \end{alignat}}{
 \begin{equation}
    s_{\gamma}(\lambda_{\hg},\lambda_{\hg})
    \leq \sum_{i=1}^{2}||\vecu^{*}_{\hi}(\lambda_{\hg})\cdot\vecn_{i}||_{\gamma}||\lambda_{\hg}||_{\gamma}\leq \sum_{i=1}^{2}\alpha_{\gamma}^{1/2}a_{i}(\vecu^{*}_{\hi}(\lambda_{\hg}),\vecu^{*}_{\hi}(\lambda_{\hg}))^{1/2}||\lambda_{\hg}||_{\gamma}.
 \end{equation}}This result together with~\eqref{def:s_g} leads
 to the upper bound in~\eqref{def_inter_oper_disc}.
 We prove the lower bound by induction.\FIX{. To this aim}{} We
 let $(\psi_{i},\vecr_{i})$, $i\in\{1,2\}$,  be the solution of the
 auxiliary subdomain problem
\begin{gather*}
\begin{aligned}
& \vecr_{i}+\vK_{i}\nabla \psi_{i}=\bm{0}\quad\textn{in }\Omegai,
&&\FIX{}{\psi_{i}=0\quad\textn{on }\Gamma_{i},}\\
&\nabla\cdot \vecr_{i}=0\quad\textn{in }\Omegai,
\quad&&\FIX{}{\vecr_{i}\cdot\vecn_{i}=\mathcal{Q}_{\hi}\lambda_{\hg}\quad\textn{on
}\gamma.}
\FIX{,\\&\psi_{i}=0,\quad&&\textn{on }\Gamma_{i},\\
&\vecr_{i}\cdot\vecn_{i}=\mathcal{Q}_{\hi}\lambda_{\hg}\quad&&\textn{on
}\gamma.}{}
\end{aligned}
\end{gather*}
{For fracture network with immersed fractures or for subdomains with $\Gamma_{i}=\emptyset$,  $\lambda_{\hg}$ approximates the pressure on $\gamma$\FIX{, which is }{} determined up to a constant. This constant is fixed by a zero  mean value constraint for $M_{\hg}$~\cite{Frih2012,MR2306414}. Thus, the auxiliary problem is well-posed since $(\vecr_{i}\cdot\vecn_{i},1)_{\partial\Omega_{i}}=(\mathcal{Q}_{\hi}\lambda_{\hg},1)_{\partial\Omega_{i}}=0$.} Now, we choose $\vecv=\tilde{\Pi}_{i} \vecr_{i}$ in~\eqref{subdo1_weak_mixed_formulation_disc},
 to obtain
 \begin{gather}
\nonumber||\mathcal{Q}_{\hi}\lambda_{\hg}||_{\gamma}^{2}=(\lambda_{\hg}, \tilde{\Pi}_{i} \vecr_{i}\cdot\vecn_{i})_{\gamma}
=-a_{i}(\vecu^{*}_{\hi}(\lambda_{\hg}),\tilde{\Pi}_{i} \vecr_{i})
+b_{i}(\tilde{\Pi}_{i} \vecr_{i},p^{*}_{\hi}(\lambda_{\hg}))
=-a_{i}(\vecu^{*}_{\hi}(\lambda_{\hg}),\tilde{\Pi}_{i} \vecr_{i})\\\nonumber
\leq C C_{\vK}||\vecu^{*}_{\hi}(\lambda_{\hg})||_{\Omegai}|| \vecr_{i}||_{1/2,\Omegai}
+\alpha_{\gamma}^{-1}|| \vecu^{*}_{\hi}(\lambda_{\hg})\cdot\vecn_{i}||_{\gamma}||\mathcal{Q}_{\hi} \lambda_{\hg}||_{\gamma}
\leq C C_{\vK}||\vecu^{*}_{\hi}(\lambda_{\hg})||_{\Omegai}||
\mathcal{Q}_{\hi}\lambda_{\hg}||_{\gamma}+\\\label{bound_low_s}
\alpha_{\gamma}^{-1}|| \vecu^{*}_{\hi}(\lambda_{\hg})\cdot\vecn_{i}||_{\gamma}
|| \mathcal{Q}_{\hi}\lambda_{\hg}||_{\gamma}
\leq \left(C\dfrac{C_{\vK}}{\sqrt{c_{\vK}}}+\dfrac{1}{\sqrt{\alpha_{\gamma}}}\right)\sqrt{a_{i}(\vecu^{*}_{\hi}(\lambda_{\hg}),\vecu^{*}_{\hi}(\lambda_{\hg}))}|| \mathcal{Q}_{\hi}\lambda_{\hg}||_{\gamma},
 \end{gather}
where we used~\eqref{coer:a_i}, assumption~\ref{ass_K_omega} and the
elliptic regularity~\eqref{estimate_regularity_Pi} \FIX{together with}{and $|| \vecr_{i}||_{1/2,\Omegai}\lesssim || \mathcal{Q}_{\hi}\lambda_{\hg}||_{\gamma}$.}
%
The  bound~\eqref{bound_low_s} in combination
with~\eqref{def:s_g}-\eqref{coer:a_i}~and~\eqref{solvability} delivers the lower bound in~\eqref{def_inter_oper_disc}.
\end{proof}
\FIX{ As announced in the introduction, i}{I}t is interesting to study the robustness of Algorithm~\ref{monolithic_ldd} and Algorithm~\ref{splitting_ldd} for the
limiting case\FIX{in which the coefficient }{} $\alpha_{\gamma}\rightarrow\infty$ in the
 transmission conditions~\eqref{Initial_system_r_interface}\FIX{. This case is physically}{, which is} corresponding to a continuous pressure over the fracture interface.

\begin{lem}[Parameter-robustness
($\alpha_{\gamma}\rightarrow\infty$)]In the case of continuous pressure across $\gamma$,
there exists  a constant $C_{2} > 0$
such that, for all $\lambda_{\hg}\in M_{\hg}$,
 \begin{equation}\label{def_inter_oper_cont}
 C_{2}{c_{\vK}}{C_{\vK}^{-2}}||\lambda_{\hg}||_{\gamma}^{2}\leq
 s_{\gamma}(\lambda_{\hg},\lambda_{\hg})\leq {C_{\textn{dTr}}^{2}}{ c_{\vK}^{-1}}h^{-1}||\lambda_{\hg}||_{\gamma}^{2}.
 \end{equation}
\end{lem}
\begin{proof}
 Recalling~the definition~\eqref{bilinear_forms_discrete}
 of $s_{\gamma}$ and using \eqref{dtrace_ineq}, we have
\FIX{\begin{alignat*}{2}
 &0\leq s_{\gamma}(\lambda_{\hg},\lambda_{\hg})&&=-\sum_{i=1}^{2}(
 \lambda_{\hg}, \vecu^{*}_{\hi}(\lambda_{\hg})\cdot\vecn_{i})_{\gamma}
\leq \sum_{i=1}^{2}||\vecu^{*}_{\hi}(\lambda_{\hg})\cdot\vecn_{i}||_{\gamma}||\lambda_{\hg}||_{\gamma},\\
&&&\leq \sum_{i=1}^{2}C_{\textn{dTr}}h^{-1/2}||\vecu_{\hi}(\lambda_{\hg})||_{\Omegai}||\lambda_{\hg}||_{\gamma},
\end{alignat*}}{
\begin{equation}
    0\leq s_{\gamma}(\lambda_{\hg},\lambda_{\hg})
    \leq \sum_{i=1}^{2}||\vecu^{*}_{\hi}(\lambda_{\hg})\cdot\vecn_{i}||_{\gamma}||\lambda_{\hg}||_{\gamma} \leq \sum_{i=1}^{2}C_{\textn{dTr}}h^{-1/2}||\vecu_{\hi}(\lambda_{\hg})||_{\Omegai}||\lambda_{\hg}||_{\gamma}.
\end{equation}}\FIX{where in that case we used the discrete trace inequality}{}This result together with~\eqref{def:s_g} and \eqref{coer:a_i}
leads to the upper bound
in~\eqref{def_inter_oper_cont}. By inspection of the
proof of~Lemma~\ref{lem:bound_on_s}, starting as in~\eqref{bound_low_s} we
promptly get the
lower bound of~\eqref{def_inter_oper_cont}.
\end{proof}
 In the following, we denote by $||\cdot||_{s,\gamma}$ the  induced  semi-norm from $s_{\gamma}$ on
$L^{2}(\gamma)$,
\begin{equation}\label{seminorm}
 ||\mu||_{s,\gamma}:=s_{\gamma}(\mu,\mu)^{1/2},\quad\forall\mu\in L^{2}(\gamma).
\end{equation}
We will also consider the following discrete norms:
\bse\label{discrete_norms}\begin{alignat}{2}
 &||(\vecv_{\hg}, \mu_{\hg})||^{2}_{0,\tau,\star}\eq ||\vK_{\gamma}^{-\frac{1}{2}}\vecv_{\hg}||^{2}_{\gamma}
+||L_{\gamma}^{\frac{1}{2}}\vecv_{\hg}||^{2}_{\gamma}+||\tau^{-\frac{1}{2}}\mu_{\hg}||^{2}_{\gamma},\\
 &||\vecv_{\hg}||^{2}_{\mathbf{V}_{\hg}}\eq||\vK^{-\frac{1}{2}}_{\gamma}\vecv_{\hg}||^{2}_{\gamma}
+||L_{\gamma}^{\frac{1}{2}}\vecv_{\hg}||^{2}_{\gamma}+
||\tau^{\frac{1}{2}}\nabla_{\tau}\cdot\vecv_{\hg}||^{2}_{\gamma},\\
&||\mu_{\hg}||^{2}_{M_{\hg}}\eq||\mu_{\hg}||^{2}_{s,\gamma}+
||\tau^{-\frac{1}{2}}\mu_{\hg}||^{2}_{\gamma},\\
&||(\vecv_{\hg}, \mu_{\hg})||^{2}_{1,\tau,\star}\eq||\vecv_{\hg}||^{2}_{\mathbf{V}_{\hg}}+||\mu_{\hg}||^{2}_{M_{\hg}}.
\end{alignat}\ese
\FIX{We  start with the estimate below.}{The following estimates are obtained.}
\begin{lem} [Inverse energy estimates]
 There  holds  for  all
$(\vecu_{\hg}, \lambda_{\hg})\in  \mathbf{V}_{\hg}\times M_{\hg}$,
\begin{equation}
\label{estim_1_0_alpha} ||(\vecu_{\hg}, \lambda_{\hg})||_{1,\tau,\star}
 \leq\sqrt{\max((1+C_{\textn{inv}} c_{\vK,\gamma}\tau h^{-2}),
 (1+\alpha_{\gamma}\tau))}||(\vecu_{\hg}, \lambda_{\hg})||_{0,\tau,\star}.
\end{equation}
Furthermore, if $\alpha_{\gamma}\rightarrow\infty$, there holds
\begin{equation}
 \label{estim_1_0_noalpha} ||(\vecu_{\hg}, \lambda_{\hg})||_{1,\tau,\star}
 \leq\sqrt{\max((1+C_{\textn{inv}} c_{\vK,\gamma}\tau h^{-2}),(1+C_{\textn{dTr}}^2c_{\vK}^{-1} \tau h^{-1}))}
 ||(\vecu_{\hg}, \lambda_{\hg})||_{0,\tau,\star}.
\end{equation}
\end{lem}
\begin{proof}
 \FIX{Owing to the inverse inequality~}{} With \eqref{inv_ineq}\FIX{, together}{} and~\eqref{def_inter_oper_disc},
 we obtain~\eqref{estim_1_0_alpha}\FIX{, and i}{. I}f $\alpha_{\gamma}\rightarrow\infty$,
 we make use~\eqref{def_inter_oper_cont} to get~\eqref{estim_1_0_noalpha}.
\end{proof}
The following results  are immediately verified.
\begin{lem}[Boundedness on ${A}_{\gamma}$]There holds   for  all
$(\vecu_{\hg}, \lambda_{\hg})$, $(\vecv_{\hg}, \mu_{\hg})\in  \mathbf{V}_{\hg}\times M_{\hg}$,
\begin{equation}
 \mathcal{A}_{\gamma}((\vecu_{\hg}, \lambda_{\hg}),(\vecv_{\hg}, \mu_{\hg}))\leq
 ||(\vecu_{\hg}, \lambda_{\hg})||_{1,\tau,\star}||(\vecv_{\hg}, \mu_{\hg})||_{1,\tau,\star}.
\end{equation}
\end{lem}
\begin{lem}[Positivity on ${A}_{\gamma}$]
 There  holds for all
 $(\vecu_{\hg}, \lambda_{\hg})\in  \mathbf{V}_{\hg}\times M_{\hg}$,
 \begin{equation}\label{positivity_A}
 \mathcal{A}_{\gamma}((\vecu_{\hg}, \lambda_{\hg}),(\vecu_{\hg}, \lambda_{\hg}))
 = ||\vK^{-\frac{1}{2}}_{\gamma}\vecu_{\hg}||^{2}_{\gamma}
+||L_{\gamma}^{\frac{1}{2}}\vecu_{\hg}||^{2}_{\gamma}+
  ||\tau^{-\frac{1}{2}}\lambda_{\hg}||_{\gamma}^{2}.
\end{equation}
\end{lem}
The  above results are  then used to prove the following
stability estimate for $A_{\gamma}+s_{\gamma}$.
\begin{thm}[Stability results] \label{thm:stability_estimate} \FIX{Let}{For}
$(\vecu_{\hg}, \lambda_{\hg})\in  \mathbf{V}_{\hg}\times M_{\hg}$, we have
\begin{alignat}{2}\label{stability_estimate_alpha}
\dfrac{1}{6(1+\tau\alpha_{\gamma})^2}||(\vecu_{\hg},
\lambda_{\hg})||_{1,\tau,\star}\!\!\leq\!\!
 \sup_{(\vecv_{\hg}, \mu_{\hg})\in  \mathbf{V}_{\hg}\times M_{\hg}}
 \dfrac{\mathcal{A}_{\gamma}((\vecu_{\hg}, \lambda_{\hg}),(\vecv_{\hg}, \mu_{\hg}))
+s_{\gamma}(\lambda_{\hg},\mu_{\hg})}{||(\vecv_{\hg}, \mu_{\hg})||_{1,\tau,\star}}.
\end{alignat}
If $\alpha_{\gamma}\rightarrow\infty$, we have
\begin{alignat}{2}\label{stability_estimate_noalpha}
\dfrac{1}{6(1\!+\!C_{\textn{dTr}}^{2}c_{\vK}^{-1}\frac{\tau}{h
})^2}||\!\left(\vecu_{\hg},
\lambda_{\hg}\right)\!||_{1,\tau,\star}\!\!\leq\!\!\!\!
 \sup_{(\vecv_{\hg}, \mu_{\hg})\in  \mathbf{V}_{\hg} \times M_{\hg}}\!\!\!\!\!
 \dfrac{\mathcal{A}_{\gamma}((\vecu_{\hg}, \lambda_{\hg}),(\vecv_{\hg}, \mu_{\hg}))
\!+\!s_{\gamma}(\lambda_{\hg},\mu_{\hg})}{||(\vecv_{\hg}, \mu_{\hg})||_{1,\tau,\star}}
\end{alignat}
\end{thm}
\begin{proof}
 Let us first recall \FIX{this}{the}  inf-sup  condition;
 given $\lambda_{\hg}\in
 M_{\hg}$, we construct \FIX{an element}{} $\vecr_{\hg}\in \mathbf{V}_{\hg}$
 such that
\begin{alignat}{2}\label{inf_sup_b}
 b_{\gamma}(\vecr_{\hg},\lambda_{\hg})=||\lambda_{\hg}||_{\gamma}^{2},\textn{ and }||\lambda_{\hg}||_{\gamma}\leq C(\gamma) ||\vecr_{\hg}||_{\gamma}.
\end{alignat}
 Let $\Psi_{\gamma}\in H^{2}_{0}(\gamma)$ \FIX{be the solution
    of }{satisfy} $-\Delta_{\tau} \Psi_{\gamma}=\tau^{-1}\lambda_{\hg}$. \FIX{Pose}{Take}
 $\vecr_{\gamma}=-\nabla_{\tau} \Psi_{\gamma}$ and let
 $\vecr_{\hg}=\Pi_{h,\gamma}\vecr_{\gamma}$, where $\Pi_{h,\gamma}$
 is the Raviart-Thomas projection onto
 $\mathbf{V}_{\hg}$~\cite{MR3264361,MR3829517}. Then\FIX{, we have}{}
 $\nabla_{\tau}\cdot \vecr_{\hg}=\Pi_{h,\gamma}\nabla_{\tau}\cdot
 \vecr_{\gamma}=\tau^{-1}\lambda_{\hg}$\FIX{.
 Hence,}{ and} $b_{\gamma}(\vecr_{\hg},\lambda_{\hg})=
 ||\tau^{-\frac{1}{2}}\lambda_{\hg}||_{\gamma}^{2}$.
 \FIX{Furthermore, we have}{Finally,}
 $||\vecr_{\hg}||_{\gamma}^{2}
 \leq C||\vecr_{\gamma}||_{1,\gamma}^{2}
 \leq C||\Psi_{\gamma}||_{2,\gamma}^{2}\leq C(\gamma)||\tau^{-\frac{1}{2}}\lambda_{\hg}||_{\gamma}^{2}$.
 \FIX{Now, consider}{Set} $\delta_{1},\delta_{2}>0$, and let
 $\vecv_{\hg}=
 \vecu_{\hg}-\delta_{2}\vecr_{\hg}$ and
 $\mu_{\hg}=\lambda_{\hg}+\delta_{1}\tau\nabla_{\tau}\cdot\vecu_{\hg}$, where
 $\vecr_{\hg}$ \FIX{is}{} from~\eqref{inf_sup_b}.
 We \FIX{have}{get}
 \begin{alignat}{2}
\nonumber&  \mathcal{A}_{\gamma}((\vecu_{\hg}, \lambda_{\hg}),(\vecv_{\hg}, \mu_{\hg}))
+s_{\gamma}(\lambda_{\hg},\mu_{\hg})=\{\mathcal{A}_{\gamma}((\vecu_{\hg},
\lambda_{\hg}),(\vecu_{\hg}, \lambda_{\hg}))
+s_{\gamma}(\lambda_{\hg},\lambda_{\hg})\}\\
&\label{decompo_A}\qquad+\delta_{1}\{\mathcal{A}_{\gamma}((\vecu_{\hg},
\lambda_{\hg}),\tau(\mathbf{0}, \nabla_{\tau}\cdot\vecu_{\hg}))
+s_{\gamma}(\lambda_{\hg},\tau\nabla_{\tau}\cdot\vecu_{\hg})\}
-\delta_{2}\{\mathcal{A}_{\gamma}((\vecu_{\hg},
\lambda_{\hg}),(\vecr_{\hg}, 0))
\}.
 \end{alignat}
For the first term on the right-hand side of~\eqref{decompo_A}, we obtain using estimate~\eqref{positivity_A} together with~\eqref{seminorm},
\begin{alignat}{2}
 \nonumber\mathcal{A}_{\gamma}((\vecu_{\hg}, \lambda_{\hg}),(\vecu_{\hg}, \lambda_{\hg}))
+s_{\gamma}(\lambda_{\hg},\lambda_{\hg})=||\vK_{\gamma}^{-\frac{1}{2}}\vecu_{\hg}||^{2}_{\gamma}
+||L_{\gamma}^{\frac{1}{2}}\vecu_{\hg}||^{2}_{\gamma}+
  ||\tau^{-\frac{1}{2}}\lambda_{\hg}||_{\gamma}^{2}+
  ||\lambda_{\hg}||_{s,\gamma}^{2}.
\end{alignat}
For the second term,
we get for all $\epsilon_{1}>0$,
 \begin{gather*}
 \mathcal{A}_{\gamma}((\vecu_{\hg}, \lambda_{\hg}),\tau(\mathbf{0}, \nabla_{\tau}\cdot\vecu_{\hg}))
 +s_{\gamma}(\lambda_{\hg},\tau\nabla_{\tau}\cdot\vecu_{\hg})
= ||\tau^{\frac{1}{2}}\nabla_{\tau}\cdot\vecu_{\hg}||^{2}_{\gamma}+
(\lambda_{\hg},\nabla_{\tau}\cdot\vecu_{\hg})_{\gamma}
  \\+s_{\gamma}(\lambda_{\hg},\tau\nabla_{\tau}\cdot\vecu_{\hg})
 \geq ||\tau^{\frac{1}{2}}\nabla_{\tau}\cdot\vecu_{\hg}||^{2}_{\gamma}
  -||\tau^{\frac{1}{2}}\nabla_{\tau}\cdot\vecu_{\hg}||_{\gamma}
  ||\tau^{-\frac{1}{2}}\lambda_{\hg}||_{\gamma} -\alpha_{\gamma}\tau||
  \tau^{\frac{1}{2}}\nabla_{\tau}\cdot\vecu_{\hg}||_{\gamma} ||\tau^{-\frac{1}{2}}\lambda_{\hg}||_{\gamma} \\
  \geq \left(1-\epsilon_{1}\dfrac{\left(1+\tau\alpha_{\gamma}\right)}{2}\right)||
  \tau^{\frac{1}{2}}\nabla_{\tau}\cdot\vecu_{\hg}||^{2}_{\gamma}
 -\dfrac{\left(1+\tau\alpha_{\gamma}\right)}{2\epsilon_{1}}||\tau^{-\frac{1}{2}}\lambda_{\hg}||^{2}_{\gamma},
 \end{gather*}
 where we have used the continuity of $s_{\gamma}$, i.e., $ s_{\gamma}(\lambda_{\hg},\mu_{\hg})\leq ||\lambda_{\hg}||_{s,\gamma}||\mu_{\hg}||_{s,\gamma}\leq \alpha_{\gamma}||\lambda_{\hg}||_{\gamma}||\mu_{\hg}||_{\gamma}. $
%

\noindent For the last term,
using~\ref{ass_K_omega} together with~\eqref{inf_sup_b} (first equation),
we obtain for all $\epsilon_{2}>0$,
\begin{align*}
\mathcal{A}_{\gamma}((\vecu_{\hg}, \lambda_{\hg}),(\vecr_{\hg}, 0)) &\leq
\dfrac{1}{2\epsilon_{2}}(||\vK_{\gamma}^{-\frac{1}{2}}\vecu_{\hg}||_{\gamma}^{2}+||L_{\gamma}^{\frac{1}{2}}\vecu_{\hg}||_{\gamma}^{2})
+\dfrac{\epsilon_{2}}{2}(||\vK_{\gamma}^{-\frac{1}{2}}\vecr_{\hg}||_{\gamma}^{2}+||L_{\gamma}^{\frac{1}{2}}\vecr_{\hg}||_{\gamma}^{2}),\\
-b_{\gamma}(\vecr_{\hg},\lambda_{\hg})
& \leq\dfrac{1}{2\epsilon_{2}}(||\vK_{\gamma}^{-\frac{1}{2}}\vecu_{\hg}||_{\gamma}^{2}+||L_{\gamma}^{\frac{1}{2}}\vecu_{\hg}||_{\gamma}^{2})
+\dfrac{\epsilon_{2}}{2}(C_{\vK,\gamma}+L_{\gamma})C(\gamma)||\vecr_{\hg}||_{\gamma}^{2}
-||\tau^{-\frac{1}{2}}\lambda_{\hg}||_{\gamma}^{2}. \end{align*}
Thus,  with~\eqref{inf_sup_b} (second equation),
\begin{gather*}
-\delta_{2}\mathcal{A}_{\gamma}((\vecu_{\hg},
\lambda_{\hg}),(\vecr_{\hg},0)\!\geq\!
\delta_{2}\!\left(1-\epsilon_{2}\dfrac{C(\gamma)(C_{\vK,\gamma}+L_{\gamma})}{2}\right)\!||\tau^{-\frac{1}{2}}\lambda_{\hg}||_{\gamma}^{2}
-\dfrac{\delta_{2}}{2\epsilon_{2}}(||\vK_{\gamma}^{-\frac{1}{2}}\vecu_{\hg}||_{\gamma}^{2}+||L_{\gamma}^{\frac{1}{2}}\vecu_{\hg}||_{\gamma}^{2}).
\end{gather*}
Collecting the previous results we get
 \begin{align*}
\mathcal{A}_{\gamma}((\vecu_{\hg}, &\lambda_{\hg}),(\vecv_{\hg}, \mu_{\hg}))+s_{\gamma}(\lambda_{\hg},\mu_{\hg})\\
&\geq
\left(1-\dfrac{\delta_{2}}{2\epsilon_{2}}\right)\left(||\vK_{\gamma}^{-\frac{1}{2}}\vecu_{\hg}||^{2}_{\gamma}
+||L_{\gamma}^{\frac{1}{2}}\vecu_{\hg}||^{2}_{\gamma}\right)
+\delta_{1}\left(1-\epsilon_{1}\dfrac{(1+\tau\alpha_{\gamma})}{2}\right)||\tau^{\frac{1}{2}}\nabla_{\tau}\cdot\vecu_{\hg}||^{2}_{\gamma}\\
&\quad +\left(1-\delta_{1}\dfrac{(1+\tau\alpha_{\gamma})}{2\epsilon_{1}}\right) ||\tau^{-\frac{1}{2}}\lambda_{\hg}||_{\gamma}^{2}
+||\lambda_{\hg}||_{s,\gamma}^{2}+\delta_{2}\left(1-\epsilon_{2}\dfrac{C(\gamma)(C_{\vK,\gamma}+L_{\gamma})}{2}\right)||\tau^{-\frac{1}{2}}\lambda_{\hg}||_{\gamma}^{2}.
\end{align*}
Now, let us choose  the  parameters
$\epsilon_{i}$ and $\delta_{i}$
such that all the norms in \FIX{~\eqref{estimatee_lower_A_s}}{the previous
inequality}
are multiplied by positive coefficients.  We choose
$\epsilon_{1}=1/(1+\tau\alpha_{\gamma})$ and  $\delta_{1}=1/(1+\tau\alpha_{\gamma})^2$,
and then $\epsilon_{2}={2}/[{C(\gamma)(C_{\vK,\gamma}+L_{\gamma})}]$ and
 $\delta_{2}={2}/[{C(\gamma)(C_{\vK,\gamma}+L_{\gamma})+1}]$, to  get
 \begin{gather*}
\mathcal{A}_{\gamma}((\vecu_{\hg}, \lambda_{\hg}),(\vecv_{\hg}, \mu_{\hg}))+s_{\gamma}(\lambda_{\hg},\mu_{\hg})
\geq
\dfrac{C(\gamma)(C_{\vK,\gamma}+L_{\gamma})+2}{2(C(\gamma)(C_{\vK,\gamma}+L_{\gamma})+1)}\left(||\vK_{\gamma}^{-\frac{1}{2}}\vecu_{\hg}||^{2}_{\gamma}
+||L_{\gamma}^{\frac{1}{2}}\vecu_{\hg}||^{2}_{\gamma}\right)\\
+\dfrac{1}{2(1+\tau\alpha_{\gamma})^2}||\tau^{\frac{1}{2}}\nabla_{\tau}\cdot\vecu_{\hg}||^{2}_{\gamma}
+\dfrac{1}{2} ||\tau^{-\frac{1}{2}}\lambda_{\hg}||_{\gamma}^{2}+||\lambda_{\hg}||_{s,\gamma}^{2}.
\end{gather*}
Thus,
\begin{alignat}{2}
 \label{estimate_lower_A_s} \mathcal{A}_{\gamma}((\vecu_{\hg},
 \lambda_{\hg}),(\vecv_{\hg},
 \mu_{\hg}))+s_{\gamma}(\lambda_{\hg},\mu_{\hg})\geq
 {(4+4\tau\alpha_{\gamma})^{-2}}||(\vecu_{\hg}, \lambda_{\hg})||^{2}_{1,\tau,\star}.
\end{alignat}
\FIX{We also have}{From the choice of $ \vecv_{\hg} $ and $ \mu_{\hg} $, we have that}
\begin{alignat}{2}
 \nonumber||(\vecv_{\hg},
 \mu_{\hg})||_{1,\tau,\star}
 \nonumber
 &\leq
 ||(\vecu_{\hg}, \lambda_{\hg})||_{1,\tau,\star}+
 \delta_{1}||(\mathbf{0}, \tau\nabla_{\tau}\cdot\vecu_{\hg})||_{1,\tau,\star}+\delta_{2}
 ||(\vecr_{\hg}, 0)||_{1,\tau,\star}.
\end{alignat}
With simple calculations, it is inferred that
\bse\begin{alignat}{2}
         &\delta_{1}||(\mathbf{0}, \tau\nabla_{\tau}\cdot\vecu_{\hg})||_{1,\tau,\star}\leq
         {(1+\tau\alpha_{\gamma})^{-\frac{3}{2}}}||(\vecu_{\hg}, \lambda_{\hg})||_{1,\tau,\star},\\
        &\delta_{2}||(\vecr_{\hg}, 0)||_{1,\tau,\star}\leq
        \dfrac{2\sqrt{C(\gamma)(C_{\vK,\gamma}+L_{\gamma})+1}}{C(\gamma)(C_{\vK,\gamma}+L_{\gamma})+2}||(\vecu_{\hg}, \lambda_{\hg})||_{1,\tau,\star}.
        \end{alignat}\ese
 This implies that we have
$
||(\vecv_{\hg}, \mu_{\hg})||_{1,\tau,\star}\leq 3
 ||(\vecu_{\hg}, \lambda_{\hg})||_{1,\tau,\star}$.
This \FIX{result}{} together with~\eqref{estimate_lower_A_s} leads
to~\eqref{stability_estimate_alpha}. For \FIX{the limit case when}{}
$\alpha_{\gamma}\rightarrow\infty$, we repeat the same \FIX{lines as before}{proof} while
using~\eqref{def_inter_oper_cont} instead of~\eqref{def_inter_oper_disc}
\FIX{to promptly arrive to~}{to get} \eqref{stability_estimate_noalpha}.
\end{proof}

\begin{lem}[Well-posedness of the DD scheme]
 The domain decomposition scheme~\eqref{compact_form_lscheme} is well-posed, and all eigenvalues of the induced
   system~$\mathcal{A}_{\gamma}+s_{\gamma}$
 are bounded away from zero.
\end{lem}
\begin{proof}
    \FIX{The matrix associated to~$\mathcal{A}_{\gamma} +s_{\gamma}$ is non-singular,
    \FIX{that is to say that}{so} the system~\eqref{compact_form_lscheme} has a unique
    solution.  \FIX{ Moreover, t}{T}he stability estimate~\eqref{stability_estimate_alpha}
    \FIX{(also~\eqref{stability_estimate_noalpha})}{and ~\eqref{stability_estimate_noalpha}}   guarantee  that  the lowest
    eigenvalue is bounded away from zero.}{This directly follows from non-singularity of $\mathcal{A}_{\gamma} +s_{\gamma}$ and estimates in Theorem~\ref{thm:stability_estimate}.}
\end{proof}
Let us comment on the robustness of the stability estimate in
Theorem~\ref{thm:stability_estimate}.
First, \eqref{stability_estimate_alpha} states that, regardless of the choice of
the space and time discretization, the stability constant with respect to
\FIX{the norm}{} $||(\vecu_{\hg}, \lambda_{\hg})||_{1,\tau,\star}$ is independent of the
coefficients $\vK$,  $\vK_{\gamma}$, and the stabilization parameter
$L_{\gamma}$. One can also show that this estimate is asymptotically robust and
bounded independently of  $(\tau,\alpha_{\gamma},h)\rightarrow0$ and the stability
constant tends to $1/6$. The only issue can happen having a large coefficient
$\alpha_{\gamma}$, but this  case is resolved
in~\eqref{stability_estimate_noalpha}.  Therein, as the ratio
${\tau}/{h}\rightarrow0$,  the stability constant is approximately $1/6$.

Following  the  approach  of  Ern  and  Guermond~\cite{MR2223503},  we  now
provide  an  estimate  for  the  condition  number  of  the  stiffness matrix
associated  with the domain decomposition
system~$\mathcal{A}_{\gamma}+s_{\gamma}$.
This condition number estimate is important in our analysis as any algorithm is
stable if every step is well-conditioned.
This  will also encourage the development of the flux basis
framework in~\Cref{sec:Mufbi}.  Let us first  introduce some basic notation in
order to provide the definition of the condition number. We recall the stiffness
matrix $\mathcal{A}_{\textn{DD}}$ introduced in~\eqref{compact_system}
associated with the domain decomposition scheme~\eqref{compact_form_lscheme},
\begin{alignat}{2} \label{eq:stiffness_matrix}
  (\mathcal{A}_{\textn{DD}}V,W)_{N}\eq \mathcal{A}_{\gamma}((\vecu_{\hg},
  \lambda_{\hg}),(\vecv_{\hg}, \mu_{\hg}))+s_{\gamma}(\lambda_{\hg},\mu_{\hg}),
 \end{alignat}
 for all $(\vecu_{\hg}, \lambda_{\hg})$, $(\vecv_{\hg},\mu_{\hg})\in  \mathbf{V}_{\hg}\times M_{\hg}$,
 where  $(V,W)_{N}\eq\sum_{i=1}^{N}V_{i} W_{i}$ denotes the inner product in
$\RR^{N}$ and $|V|_{N}^{2}\eq (V,V)_{N}$ is
the corresponding Euclidean norm. The condition number is defined by \FIX{
\begin{gather*}
 \kappa(\mathcal{A}_{\textn{DD}})\eq|\mathcal{A}_{\textn{DD}}|_{N}|\mathcal{A}_{\textn{DD}}|^{-1}_{N},
\end{gather*}
}{$\kappa(\mathcal{A}_{\textn{DD}})\eq|\mathcal{A}_{\textn{DD}}|_{N}|\mathcal{A}_{\textn{DD}}|^{-1}_{N}$}
where
\begin{gather*}
 |\mathcal{A}_{\textn{DD}}|_{N}\eq\sup_{V\in \RR^{N}\setminus\mathbf{0}}\sup_{W\in \RR^{N}\setminus\mathbf{0}}\dfrac{(\mathcal{A}_{\textn{DD}}V,W)_{N}}{|V|_{N}|W|_{N}}
 \FIX{}{ =\sup_{V\in \RR^{N}
 \setminus\mathbf{0}}\dfrac{|\mathcal{A}_{\textn{DD}}|_{N}}{|V|_{N}}.
 }
\end{gather*}
\FIX{which is equivalent to
\begin{gather*}
 |\mathcal{A}_{\textn{DD}}|_{N}\eq\sup_{V\in \RR^{N} \setminus\mathbf{0}}\dfrac{|\mathcal{A}_{\textn{DD}}|_{N}}{|V|_{N}}.
\end{gather*}}{}We recall the following estimate that holds true for a conforming,
 quasi-uniform mesh $\Tau_{h}$ \cite{MR2223503}; there exists $c_{\mu},C_{\mu}>0$
 such that the following equivalence holds
\begin{alignat}{2}\label{estimate_ell2}
 c_{\mu}h^{d/2}|V|_{N}\leq ||V||_{0,\tau,\star}\leq C_{\mu}h^{d/2}|V|_{N}.
\end{alignat}
\vskip -5mm
\begin{thm}[Condition number estimate]\! \label{thm:cond_num}
\FIX{The condition number of the domain decomposition scheme}{The condition number $\kappa(\mathcal{A}_{\textn{DD}})$ of \eqref{compact_form_lscheme} is bounded as}
\begin{equation}\label{condition_numb_alpha}
 \kappa(\mathcal{A}_{\textn{DD}}) \lesssim 6(1+\tau\alpha_{\gamma})^2 \max((1+C_{\textn{inv}} c_{\vK,\gamma}\tau h^{-2}),(1+\alpha_{\gamma}\tau)).
\end{equation}
Furthermore, if $\alpha_{\gamma}\rightarrow\infty$,
\begin{equation}\label{condition_numb_noalpha}
 \kappa(\mathcal{A}_{\textn{DD}}) \lesssim 6(1+C_{\textn{dTr}}^2c_{\vK}^{-1} \tau h^{-1})^2 \max((1+C_{\textn{inv}} c_{\vK,\gamma}\tau h^{-2}),(1+C_{\textn{dTr}}^2c_{\vK}^{-1} \tau h^{-1})).
\end{equation}
\end{thm}
\begin{proof} 
By definition \FIX{}{\eqref{eq:stiffness_matrix}, using \eqref{estim_1_0_alpha} and \eqref{estimate_ell2}, we have} for all $V,W\in \RR^{N}$,
\begin{align*}
 (\mathcal{A}_{\textn{DD}}V,W)_{N}
& \leq  ||(\vecu_{\hg}, \lambda_{\hg})||_{1,\tau,\star}||(\vecv_{\hg}, \mu_{\hg})||_{1,\tau,\star},\\
& \leq \max((1+C_{\textn{inv}} c_{\vK,\gamma}\tau
h^{-2}),(1+\alpha_{\gamma}\tau))||(\vecu_{\hg},
\lambda_{\hg})||_{0,\tau,\star}||(\vecv_{\hg}, \mu_{\hg})||_{0,\tau,\star},\\
& \lesssim \max((1+C_{\textn{inv}} c_{\vK,\gamma}\tau h^{-2}),(1+\alpha_{\gamma}\tau)) h^{d}|V|_{N}|W|_{N},
\end{align*}
Consequently, \FIX{}{$|\mathcal{A}_{\textn{DD}}|_{N}\lesssim \max((1+C_{\textn{inv}} c_{\vK,\gamma}\tau h^{-2}),(1+\alpha_{\gamma}\tau)) h^{d}.$}
\FIX{To estimate $|\mathcal{A}_{\textn{DD}}|^{-1}_{N}$, start
from~\eqref{stability_estimate_alpha} and use~\eqref{estimate_ell2} to get}{On the other hand,}
\begin{align*}
(\mathcal{A}_{\textn{DD}}V,W)_{N}&\geq{\left[6(1+\tau\alpha_{\gamma})^2\right]^{-1}}
 ||(\vecu_{\hg}, \lambda_{\hg})||_{1,\tau,\star}||(\vecv_{\hg}, \mu_{\hg})||_{1,\tau,\star},\\
&\geq {\left[6(1+\tau\alpha_{\gamma})^2\right]^{-1}}||(\vecu_{\hg},
\lambda_{\hg})||_{0,\tau,\star}||(\vecv_{\hg}, \mu_{\hg})||_{0,\tau,\star}
\gtrsim {h^{d}}{\left[6(1+\tau\alpha_{\gamma})^2\right]^{-1}}|V|_{N}|W|_{N},
\end{align*}
hence $|V|_{N}\lesssim 6(1+\tau\alpha_{\gamma})^2h^{-d}|\mathcal{A}_{\textn{DD}}V|_{N}$.
\FIX{Now s}{S}etting $V=\mathcal{A}_{\textn{DD}}^{-1} W$, we \FIX{easily }{}conclude that
$ |\mathcal{A}_{\textn{DD}}^{-1}|_{N}\lesssim 6(1+\tau\alpha_{\gamma})^2 h^{-d}$. Combining
estimates for
$ |\mathcal{A}_{\textn{DD}}^{-1}|_{N}$ and $ |\mathcal{A}_{\textn{DD}}|_{N}$
we get~\eqref{condition_numb_alpha}. The estimate \eqref{condition_numb_noalpha}
\FIX{of the limiting case}{for} $ \alpha_{\gamma} \to \infty $ is obtained similarly by
using~\eqref{estim_1_0_noalpha} and~\eqref{stability_estimate_noalpha}\FIX{ in
the proof}{}.
\end{proof}
\subsection{Convergence of   MoLDD-scheme}
The second step of our analysis is to prove the convergence of
Algorithm~\ref{monolithic_ldd}. The idea is to \FIX{prove}{show} that this algorithm is a
contraction and then  apply the Banach fixed-point theorem~\cite{MR2079503}.
For that purpose, define $\delta_{\vecu,h}^{k}=\vecu^{n,k}_{\hg}-\vecu^{n,k-1}_{\hg}$ and
 $\delta_{\lambda,h}^{k}=\lambda^{n,k}_{\hg}-\lambda^{n,k-1}_{\hg}$ \FIX{be the differences
between the solutions at iteration $k$ and $k-1$ of the
problem~\eqref{lscheme_system_fracture}, respectively}{}.
\begin{thm}[Convergence of MoLDD-scheme]\label{thm:cvlscheme}
    Assuming  that Assumptions~\textn{\Assum} hold true and that
    $L_{\gamma}(\zeta)={L_{\xi}}/{2(1-\zeta)}$, with a parameter
    $\zeta\in[0,1)$,   Algorithm~\ref{monolithic_ldd}  defines  a  contraction
    given by
 \begin{alignat}{4}\label{contraction_lscheme}
&||\delta_{\lambda,h}^{k}||_{\gamma}^{2}+\tau ||\delta_{\lambda,h}^{k}||^{2}_{s,\gamma}+\left(\dfrac{L_{\gamma}}{2}+c_{\vK,\gamma}\right)\tau ||\delta_{\vecu,h}^{k}||_{\gamma}^{2}
\leq \left(\dfrac{L_{\gamma}}{2}-\zeta \xi_{m}\right)\tau ||\delta_{\vecu,h}^{k-1}||_{\gamma}^{2},
\end{alignat}
\FIX{}{and the limit is the unique solution of~\eqref{discrete_system_fracture}.}
\FIX{where  $\zeta$ is chosen to
improve the convergence rate of the scheme. Furthermore,
the limit is the unique solution of~\eqref{discrete_system_fracture}.}{}
\end{thm}
\begin{proof}
    By subtracting \eqref{lscheme_system_fracture} at step $k$ from the ones at step $k-1$, we obtain
 \bse\label{lscheme_error_fracture}\begin{alignat}{4}
\label{lscheme_error_d_f}&(\xi(\vecu_{h, \gamma}^{n,k-1}) -\xi(\vecu_{h, \gamma}^{n,k-2}),\vecv)_{\gamma}+L_{\gamma}(\delta_{\vecu,h}^{k}-\delta_{\vecu,h}^{k-1},\vecv)_{\gamma} +a_{\gamma}(\delta_{\vecu,h}^{k},\vecv)-
b_{\gamma}(\vecv,\delta_{\lambda,h}^{k})= 0  &&\quad \forall  \vecv\in
\mathbf{V}_{\hg}, &\\
\label{lscheme_error_c_f}&c_{\gamma}(\delta_{\lambda,h}^{k},\mu)+\tau
b_{\gamma}(\delta_{\vecu,h}^{k},\mu)+ \tau
s_{\gamma}(\delta_{\lambda,h}^{k},\mu)=0 &&\quad \forall  \mu\in M_{\hg}.&
\end{alignat}\ese
 \FIX{}{Taking $\vecv=\tau \delta_{\vecu,h}^{k}$ in \eqref{lscheme_error_d_f}, $\mu=\delta_{\lambda,h}^{k}$ in \eqref{lscheme_error_c_f} and summing the equations, we get}
\bse\label{lscheme_error_estimate1}\begin{alignat}{4}
\nonumber&||\delta_{\lambda,h}^{k}||^{2}+\tau ||\delta_{\lambda,h}^{k}||^{2}_{s,\gamma}+\tau a_{\gamma}(\delta_{\vecu,h}^{k},\vecv)+\tau (\xi(\vecu_{\hg}^{n,k-1})-\xi(\vecu_{h, \gamma}^{n, k-2}),\delta_{\vecu,h}^{k})_{\gamma}
+L_{\gamma}\tau (\delta_{\vecu,h}^{k}-\delta_{\vecu,h}^{k-1},\delta_{\vecu,h}^{k})_{\gamma}=0.
\end{alignat}
Following~\cite{storvik2018optimization}, we let $\zeta\in[0,1)$ and split the third term while applying the lower bound of $\vK_{\gamma}^{-1}$,
\begin{align*}
||\delta_{\lambda,h}^{k}||_{\gamma}^{2}+\tau ||\delta_{\lambda,h}^{k}&||^{2}_{s,\gamma}+c_{\vK,\gamma}\tau ||\delta_{\vecu,h}^{k}||_{\gamma}^{2}
+\zeta\tau (\xi(\vecu_{\hg}^{n,k-1})-\xi(\vecu_{h, \gamma}^{n, k-2}),\delta_{\vecu,h}^{k-1})_{\gamma} +L_{\gamma}\tau (\delta_{\vecu,h}^{k}-\delta_{\vecu,h}^{k-1},\delta_{\vecu,h}^{k})_{\gamma} \\
& +(1-\zeta)\tau (\xi(\vecu_{\hg}^{n,k-1}) -\xi(\vecu_{h, \gamma}^{n, k-2}),\delta_{\vecu,h}^{k-1})_{\gamma} +\tau (\xi(\vecu_{\hg}^{n,k-1})-\xi(\vecu_{h, \gamma}^{n, k-2}),\delta_{\vecu,h}^{k}-\delta_{\vecu,h}^{k-1})_{\gamma} \leq 0.
\end{align*}
We use the identity \FIX{~\eqref{identity}}{$2(a-b) a=a^{2}-b^{2}+(a-b)^{2}$ for $a,b\in\RR$} together with the monotonicity and Lipschitz continuity of $\xi$ given by~\ref{ass_xi} to get
\begin{align*}
||\delta_{\lambda,h}^{k}||_{\gamma}^{2}&+\tau ||\delta_{\lambda,h}^{k}||^{2}_{s,\gamma}+c_{\vK,\gamma}\tau ||\delta_{\vecu,h}^{k}||_{\gamma}^{2}+\zeta \xi_{m}\tau ||\delta_{\vecu,h}^{k-1}||_{\gamma}^{2}
+\dfrac{(1-\zeta)}{L_{\xi}}\tau ||\xi(\vecu_{\hg}^{n,k-1})-\xi(\vecu_{\hg}^{n,k-2})||_{\gamma}^{2}\\
&+\dfrac{L_{\gamma}}{2}\tau ||\delta_{\vecu,h}^{k}||^{2}_{\gamma}+\dfrac{L_{\gamma}}{2}\tau ||\delta_{\vecu,h}^{k}-\delta_{\vecu,h}^{k-1}||_{\gamma}^{2}\leq
\dfrac{L_{\gamma}}{2}\tau ||\delta_{\vecu,h}^{k-1}||_{\gamma}^{2}-\tau (\xi(\vecu_{\hg}^{n,k-1})-\xi(\vecu_{\hg}^{n,k-2}),\delta_{\vecu,h}^{k}-\delta_{\vecu,h}^{k-1})_{\gamma}.
\end{align*}
We apply Young's inequality\FIX{~\eqref{young}}{
$|ab|\leq (2\delta)^{-1} a^2+{2}{\delta^{-1}}b^2$ for
$a,b,\delta\in\RR,\,\delta>0$,} for the last term in the right-hand side to obtain
\begin{align*}
||\delta_{\lambda,h}^{k}||_{\gamma}^{2}&+\tau ||\delta_{\lambda,h}^{k}||^{2}_{s,\gamma}+c_{\vK,\gamma}\tau ||\delta_{\vecu,h}^{k}||_{\gamma}^{2}+\zeta
\xi_{m}\tau ||\delta_{\vecu,h}^{k-1}||_{\gamma}^{2}
+\dfrac{(1-\zeta)}{L_{\xi}}\tau
||\xi(\vecu_{\hg}^{n,k-1})-\xi(\vecu_{\hg}^{n,k-2})||_{\gamma}^{2}
+\dfrac{L_{\gamma}}{2}\tau ||\delta_{\vecu,h}^{k}||_{\gamma}^{2}\\
&+\dfrac{L_{\gamma}}{2}\tau ||\delta_{\vecu,h}^{k}-\delta_{\vecu,h}^{k-1}||_{\gamma}^{2}
\leq \dfrac{L_{\gamma}}{2}\tau ||\delta_{\vecu,h}^{k-1}||^{2}_{\gamma}+\dfrac{L_{\gamma}}{2}\tau ||\delta_{\vecu,h}^{k}-\delta_{\vecu,h}^{k-1}||^{2}_{\gamma}+
\dfrac{1}{2L_{\gamma}}\tau ||\xi(\vecu_{\hg}^{n,k-1})-\xi(\vecu_{h, \gamma}^{n,k-2})||_{\gamma}^{2}.
\end{align*}\ese
\FIX{We choose}{By choosing} $L_{\gamma}={L_{\xi}}/{2(1-\zeta)}$, we immediately obtain~\eqref{contraction_lscheme}.
 The inequality~\eqref{contraction_lscheme} impl\FIX{y}{ies} that the sequence
 $\delta_{\lambda,h}^{n,k}\FIX{}{\rightarrow0}$ \FIX{tends to $0$}{} in
$L^{2}(\gamma)$ and $\delta_{\vecu,h}^{k}\FIX{}{\rightarrow0}$ \FIX{tends to
$0$}{}
in $\mathbf{L}^{2}(\gamma)$.
Now we choose $\mu=\nabla_{\tau}\cdot\delta_{\vecu,h}^{k}$ in
\eqref{lscheme_error_c_f} to obtain
 \begin{alignat*}{2}
 \tau ||\nabla_{\tau}\cdot \delta_{\vecu,h}^{k}||^{2}_{\gamma}
 \!=\!-c_{\gamma}(\delta_{\lambda,h}^{k},\nabla_{\tau}\cdot \delta_{\vecu,h}^{k})
 -\tau s_{\gamma}(\lambda^{n,k}_{\hg},\nabla_{\tau}\cdot
 \delta_{\vecu,h}^{k})
\!\leq\!
 ||\delta_{\lambda,h}^{k}||_{\gamma}||\nabla_{\tau}\cdot\delta_{\vecu,h}^{k}||_{\gamma}+
 \tau \alpha_{\gamma}||\delta_{\lambda,h}^{k}||_{\gamma}||\nabla_{\tau}\cdot\delta_{\vecu,h}^{k}||_{\gamma}.
 \end{alignat*}
 Thus,
\begin{alignat}{2}\label{div_estimate2}
&\tau ||\nabla_{\tau}\cdot\delta_{\vecu,h}^{k}||_{\gamma}\leq (\alpha_{\gamma}\tau +1)||\delta_{\lambda,h}^{k}||_{\gamma}.
\end{alignat}
Hence, by~\eqref{contraction_lscheme}, we have $||\nabla_{\tau}\cdot\delta_{\vecu,h}^{k}||_{\gamma}$ tends  to $0$ in
$L^{2}(\gamma)$. This shows that $\delta_{\vecu,h}^{k}$ tends to $0$
in $\mathbf{H}(\textn{div}_{\tau}, \gamma)$.

\FIXX{}{Since \eqref{contraction_lscheme} defines a contraction, by the Banach
fixed-point theorem we can conclude that the sequence generated by the algorithm converges to the
unique solution of the problem \eqref{discrete_system_fracture}.}
\end{proof}
\begin{cor}[Optimal MoLDD-convergence rate]{If $\xi_{m}>0$, }
the  minimum of the
convergence rate of  Algorithm~\ref{monolithic_ldd} is
reached for  the optimal  parameter
\bse\begin{equation}\label{choice:gamma}
\zeta^{*}=\underset{0<\zeta<1}{\argmin}\rho(\zeta)=1+
\dfrac{L_{\xi}\xi_{m}-\sqrt{(L_{\xi}\xi_{m})^2+4L_{\xi} \xi_{m}^2 c_{\vK,\gamma}+4L_{\xi} \xi_{m}c^{2}_{\vK,\gamma}}}
{4\xi_{m}c_{\vK,\gamma}},
\end{equation}
where $\rho(\zeta)$ is the convergence rate
from~\eqref{contraction_lscheme},
\begin{equation}\label{rate_mono}
  \rho(\zeta)=\dfrac{L_{\gamma}-2 \xi_{m}\zeta}{L_{\gamma}+2c_{\vK,\gamma}}<1.
\end{equation}
Therefore, the optimal stabilization parameter is given by
\begin{equation}\label{opt_stab}
L_{\gamma,\textn{opt}}=\dfrac{L_{\xi}}{2(1-\zeta^{*})}.
\end{equation}\ese
\end{cor}
\begin{proof}
Plugging $L_{\gamma}={L_{\xi}}/{2(1-\zeta)}$ in the  contraction
estimate~\eqref{contraction_lscheme} leads to\FIX{
$||\delta_{\vecu,h}^{k}||_{\gamma}^{2}
\leq \rho(\zeta)||\delta_{\vecu,h}^{k-1}||_{\gamma}^{2}$, where}{}
\begin{equation}\label{rate}
\FIX{}{||\delta_{\vecu,h}^{k}||_{\gamma}^{2}
\leq \rho(\zeta)||\delta_{\vecu,h}^{k-1}||_{\gamma}^{2}\quad\text{where}\quad}
 \rho(\zeta)=\dfrac{L_{\xi}-4(1-\zeta) \xi_{m}\zeta}{L_{\xi}+4(1-\zeta)c_{\vK,\gamma}}<1,
\end{equation}
which clearly can be minimal  when choosing the optimal value of $\zeta$.
To calculate $\zeta^{*}$, we  differentiate~\eqref{rate} with respect to
$\zeta$ and infer
the resulting roots and we find that the  minimum of~\eqref{rate} is obtained
for the optimal choice given by~\eqref{choice:gamma}. Replacing back the resulting value into
$L_{\gamma}(\zeta)$ \FIX{delivers the optimal stabilization parameter}{gives}~\eqref{opt_stab}.
\end{proof}
\begin{lem}[Well-posedness of the mixed-dimensional problem]
 There exists a unique solution to the mixed-dimensional problem~\eqref{md_discrete_system_fracture}.
\end{lem}
\begin{proof}
 Problem~\eqref{discrete_system_fracture}
 is  equivalent to~\eqref{md_discrete_system_fracture}. Since we know from Theorem~\ref{thm:cvlscheme}
that~\eqref{discrete_system_fracture} has a unique
solution, this equivalence implies that~\eqref{md_discrete_system_fracture} is uniquely solvable.
\end{proof}
\FIX{We  continue  with  some  important  remarks  concerning
the results above and the implications to the convergence rate
of MoLDD-scheme.}{}
\FIX{\begin{rem}[Dependence of the convergence rate] \end{rem}}{}
 \FIX{Obviously, t}{T}he rate of convergence~\eqref{rate_mono} depends only on the strength of
 the non-linearity by means of the Lipschitz constant $L_{\xi}$,
 the  lower bound $\xi_{m}$ and the fracture permeability $\vK_{\gamma}$.
 Particularly, the rate is \textbf{independent} of  the fracture-matrix coupling parameter
 $\alpha_{\gamma}$, the mesh size $h$ and the time step  $\tau $. Moreover,
\FIX{\begin{rem}[Global convergence]}{}
 the convergence of MoLDD-scheme is global, i.e. independent of the initialization and
 particularly of the inner DD solver\FIX{ (like GMRes)}{}.
 Nevertheless, it is\FIX{obviously}{} beneficial if one starts  MoLDD-scheme iterations with the solution
 of the last time step.

\section{Analysis of  ItLDD-scheme}\label{sec:analysis_itldd}
We turn now to the analysis of the iterative LDD-scheme
(Algorithm~\ref{splitting_ldd}). In contrast to MoLDD-scheme, in which two levels
of calculations (Linearization+DD) are necessary to achieve the required
solution, the iterative LDD-scheme treats simultaneously   the non-linearity and DD.  Thus, the next result is to be
understood as the convergence for the combined  Linearization-DD processes. \FIX{}{We again denote
$\delta_{\vecu,h}^{k}\eq\vecu^{n,k}_{\hg}-\vecu^{n}_{\hg}$ and
$\delta_{\lambda,h}^{k}\eq\lambda^{n,k}_{\hg}-\lambda^{n}_{\hg}$.}
\begin{thm}[Convergence of  ItLDD-scheme]\label{thm:ldd}
Assuming  that Assumptions~\textn{\Assum} hold true and that
$L_{\gamma,u}(\zeta)={L_{\xi}}/{2(1-\zeta)}$, where
$\zeta$ is a parameter to be optimized  in $[0,1)$, and
$L_{\gamma,p}\geq\alpha_{\gamma}$,
the ItLDD-scheme given by
Algorithm~\ref{splitting_ldd} is linearly convergent.
There holds
\begin{align} \label{lddscheme_contraction}
\left(1+\tau
\dfrac{L_{\gamma,p}}{2}\right)||\delta_{\lambda,h}^{k}||_{\gamma}^{2}\!+\!
\dfrac{\tau
}{2}||\delta_{\lambda,h}^{k}||_{s,\gamma}^{2}&+\left(\dfrac{L_{\gamma,u}}{2}\!+\!
c_{\vK,\gamma}\right)\tau ||\delta_{\vecu,h}^{k}||_{\gamma}^{2}
\!\leq\! \left(\dfrac{L_{\gamma,u}}{2}
-\zeta \xi_{m}\right)\tau ||\delta_{\vecu,h}^{k-1}||_{\gamma}^{2}\!+\!
\tau \dfrac{L_{\gamma,p}}{2}||\delta_{\lambda,h}^{k-1}||_{\gamma}^{2}.
\end{align}
\end{thm}
\begin{proof}
By subtracting~\eqref{lscheme_system_fracture} at the iteration $k$ from~\eqref{discrete_system_fracture}, we obtain
  \bse\label{lddscheme_error_fracture}\begin{alignat}{4}
\label{lddscheme_error_d_f}&(\xi(\vecu_{h, \gamma}^{n,k-1})-\xi(\vecu_{h, \gamma}^{n}),\vecv)_{\gamma}+
L_{\gamma,u}(\delta_{\vecu,h}^{k}-\delta_{\vecu,h}^{k-1},\vecv)_{\gamma} +
a_{\gamma}(\delta_{\vecu,h}^{k},\vecv)-b_{\gamma}(\vecv,\delta_{\lambda,h}^{k})=
0  &&\quad \forall  \vecv\in \mathbf{V}_{\hg},&\\
\label{lddscheme_error_c_f}&c_{\gamma}(\delta_{\lambda,h}^{k},\mu)+\tau L_{\gamma,p}(\delta_{\lambda,h}^{k}-\delta_{\lambda,h}^{k-1},\mu)_{\gamma}
+ \tau s_{\gamma}(\delta_{\lambda,h}^{k-1},\mu)+\tau
b_{\gamma}(\delta_{\vecu,h}^{k},\mu)=0 &&\quad \forall  \mu\in M_{\hg}.&
\end{alignat}\ese
Taking $\vecv=\tau \delta_{\vecu,h}^{k}$ in \eqref{lddscheme_error_d_f} and $\mu=\delta_{\lambda,h}^{k}$ in \eqref{lddscheme_error_c_f}, and summing up the equations gives
\begin{align}
 ||\delta_{\lambda,h}^{k}||_{\gamma}^{2} &+
 \tau L_{\gamma,p}(\delta_{\lambda,h}^{k}-\delta_{\lambda,h}^{k-1},\delta_{\lambda,h}^{k})_{\gamma}
 +\tau s_{\gamma}(\delta_{\lambda,h}^{k-1},\delta_{\lambda,h}^{k}) \nonumber \\
 &+\tau (\xi(\vecu_{h, \gamma}^{n,k-1})-\xi(\vecu_{h, \gamma}^{n}),\delta_{\vecu,h}^{k})_{\gamma}
\label{s_original}+L_{\gamma,u}\tau (\delta_{\vecu,h}^{k}-\delta_{\vecu,h}^{k-1},\delta_{\vecu,h}^{k})_{\gamma} +\tau a_{\gamma}(\delta_{\vecu,h}^{k},\delta_{\vecu,h}^{k})
=0.
\end{align}
 For any $\zeta\in[0,1)$, this is  equivalent to
\begin{align}
||\delta_{\lambda,h}^{k}||_{\gamma}^{2}&+
 \tau L_{\gamma,p}(\delta_{\lambda,h}^{k}-\delta_{\lambda,h}^{k-1},\delta_{\lambda,h}^{k})_{\gamma}
 +\tau s_{\gamma}(\delta_{\lambda,h}^{k},\delta_{\lambda,h}^{k})+\tau \zeta(\xi(\vecu_{h, \gamma}^{n,k-1})-\xi(\vecu_{h, \gamma}^{n}),\delta_{\vecu,h}^{k-1})_{\gamma} \nonumber\\
&
+\tau (1-\zeta)(\xi(\vecu_{h, \gamma}^{n,k-1})-\xi(\vecu_{h, \gamma}^{n}),\delta_{\vecu,h}^{k-1})_{\gamma}
 +L_{\gamma,u}\tau (\delta_{\vecu,h}^{k}-\delta_{\vecu,h}^{k-1},\delta_{\vecu,h}^{k})_{\gamma} +\tau a_{\gamma}(\delta_{\vecu,h}^{k},\delta_{\vecu,h}^{k}) \nonumber\\
& \quad =-\tau s_{\gamma}(\delta_{\lambda,h}^{k-1}-\delta_{\lambda,h}^{k},\delta_{\lambda,h}^{k})-\tau (\xi(\vecu_{h, \gamma}^{n,k-1})-\xi(\vecu_{h, \gamma}^{n}),\delta_{\vecu,h}^{k}-\delta_{\vecu,h}^{k-1})_{\gamma}.
\end{align}
We apply the lower bound in the
last term of the left-hand side and
then use the
monotonicity and Lipschitz continuity of the
operator $\xi$,  followed by
Cauchy-Schwarz and Young's inequalities
in the second term of the right-hand side, to get
\begin{gather}\label{lscheme_error_estimate44}
\Big( 1 + \tau \dfrac{L_{\gamma,p}}{2}\Big)||\delta_{\lambda,h}^{k}||_{\gamma}^{2}+\tau ||\delta_{\lambda,h}^{k}||_{s,\gamma}^{2}
+\tau \dfrac{L_{\gamma,p}}{2}||\delta_{\lambda,h}^{k}-\delta_{\lambda,h}^{k-1}||_{\gamma}^{2}
+c_{\vK,\gamma}\tau ||\delta_{\vecu,h}^{k}||_{\gamma}^{2}
+\zeta \xi_{m}||\delta_{\vecu,h}^{k-1}||_{\gamma}^{2} +\dfrac{L_{\gamma,u}}{2}\tau ||\delta_{\vecu,h}^{k}||_{\gamma}^{2} \nonumber \\
+\dfrac{(1-\zeta)}{L_{\xi}}\tau ||\xi(\vecu_{\hg}^{n,k-1})-\xi(\vecu_{\hg}^{n})||_{\gamma}^{2}
+\dfrac{L_{\gamma,u}}{2}\tau ||\delta_{\vecu,h}^{k}-\delta_{\vecu,h}^{k-1}||_{\gamma}^{2} \nonumber
\leq \dfrac{L_{\gamma,u}}{2}\tau
||\delta_{\vecu,h}^{k-1}||_{\gamma}^{2} +\dfrac{L_{\gamma,p}}{2}\tau
||\delta_{\lambda,h}^{k-1}||_{\gamma}^{2}\\
+\dfrac{L_{\gamma,u}}{2}\tau ||\delta_{\vecu,h}^{k}-\delta_{\vecu,h}^{k-1}||_{\gamma}^{2}  -\tau s_{\gamma}(\delta_{\lambda,h}^{k-1}-\delta_{\lambda,h}^{k},\delta_{\lambda,h}^{k})
+\dfrac{1}{2L_{\gamma,u}}\tau ||\xi(\vecu_{\hg}^{k-1})-\xi(\vecu^{n}_{\hg})||_{\gamma}^{2},
\end{gather}
\FIX{{Using \eqref{def_inter_oper_disc} and}} the continuity of $s_{\gamma}$ we get
\begin{alignat}{2}
\label{cont_s_gamma}| s_{\gamma}(\delta_{\lambda,h}^{k-1}-\delta_{\lambda,h}^{k},\delta_{\lambda,h}^{k})|&&\leq ||\delta_{\lambda,h}^{k}||_{s,\gamma}||\delta_{\lambda,h}^{k}-\delta_{\lambda,h}^{k-1}||_{s,\gamma}
\leq\alpha_{\gamma}^{1/2}||\delta_{\lambda,h}^{k}||_{s,\gamma}||\delta_{\lambda,h}^{k}-\delta_{\lambda,h}^{k-1}||_{\gamma}.
\end{alignat}
Applying Young's inequality to~\eqref{cont_s_gamma}, plugging the result
in~\eqref{lscheme_error_estimate44} and
choosing $L_{\gamma,u}={L_{\xi}}/{2(1-\zeta)}$ gives
\begin{align}  \label{lscheme_error_estimate66}
\left(1+\tau \dfrac{L_{\gamma,p}}{2}\right)&||\delta_{\lambda,h}^{k}||_{\gamma}^{2}
+\tau ||\delta_{\lambda,h}^{k}||_{s,\gamma}^{2}
+\tau \dfrac{L_{\gamma,p}}{2}||\delta_{\lambda,h}^{k}-\delta_{\lambda,h}^{k-1}||_{\gamma}^{2}
+\left(\dfrac{L_{\gamma,u}}{2}+c_{\vK,\gamma}\right)\tau ||\delta_{\vecu,h}^{k}||_{\gamma}^{2} \nonumber\\
&\leq \left(\dfrac{L_{\gamma,u}}{2}-\zeta \xi_{m}\right)\tau ||\delta_{\vecu,h}^{k-1}||_{\gamma}^{2}+
\tau \dfrac{L_{\gamma,p}}{2} ||\delta_{\lambda,h}^{k-1}||_{\gamma}^{2}
+\tau \dfrac{\alpha_{\gamma}}{2}||\delta_{\lambda,h}^{k-1}-\delta_{\lambda,h}^{k}||_{\gamma}^{2}+\dfrac{\tau }{2}||\delta_{\lambda,h}^{k}||^{2}_{s,\gamma}.
\end{align}
We  let
$L_{\gamma,p}\geq\alpha_{\gamma}$, to obtain the
estimate~\eqref{lddscheme_contraction}
\FIX{which is clearly a  contraction}{}. We \FIX{finally}{}
repeat the same techniques as
in~\eqref{div_estimate2},  to  get that
$||\nabla_{\tau}\cdot\delta_{\vecu,h}^{k}||_{\gamma}$ tends  to $0$ in
$L^{2}(\gamma)$. This \FIX{altogether}{} shows that
$\delta_{\lambda,h}^{n,k}$ tends to 0 in $L^{2}(\gamma)$ and $\delta_{\vecu,h}^{k}$ tends to $0$
in $\mathbf{H}(\textn{div}_{\tau}, \gamma)$.
\end{proof}
\FIX{\begin{rem}[Contraction factor] }{}Our  contraction  estimate  shows
 that  the  strength of the non-linearity and the matrix-fracture (DD) coupling
 controls the  convergence rate.  In practice,
 the contraction factor   is better if we take
 into account the energy norm ${\tau}||\delta_{\lambda,h}^{k}||_{s,\gamma}^{2}/2$ using the
 bound~\eqref{def_inter_oper_disc}. \FIX{As the stabilization term is such that}{Since we assume} $L_{\gamma,p}\geq \alpha_{\gamma}$,
 \FIX{thus, }{}we have to study the robustness of the algorithm when $\alpha_{\gamma}\rightarrow\infty$\FIX{,  corresponding physically to the case of continuous pressure across the fracture}{}.
 \begin{lem}[Contraction-robustness]\label{thm:cvlscheme_limit_case}
Assuming continuous pressure across $\gamma$ ($\alpha_{\gamma}\rightarrow\infty$), then
let
$L_{\gamma,u}(\zeta)={L_{\xi}}/{2(1-\zeta)}$ with $\zeta$ to be chosen in
$[0,1)$,  and
$L_{\gamma,p}\geq {C_{\textn{dTr}}^2}/({c_{\vK}}h)$,
the contraction~\eqref{lddscheme_contraction}  holds true for
the ItLDD-scheme in Algorithm~\ref{splitting_ldd}.
\end{lem}
\begin{proof}
 Recall the estimate~\eqref{lscheme_error_estimate44}
 which holds true in that case. We then estimate the coupling term
 $| s_{\gamma}(\delta_{\lambda,h}^{k}-\delta_{\lambda,h}^{k-1},\delta_{\lambda,h}^{k})|$  with the help of~\eqref{def_inter_oper_cont},
 \begin{alignat}{2}
| s_{\gamma}(\delta_{\lambda,h}^{k}-\delta_{\lambda,h}^{k-1},\delta_{\lambda,h}^{k})|
\leq ||\delta_{\lambda,h}^{k}||_{s,\gamma}||\delta_{\lambda,h}^{k}-\delta_{\lambda,h}^{k-1}||_{s,\gamma},
\leq  C_{\textn{dTr}} c_{\vK}^{-1/2}h^{-1/2} ||\delta_{\lambda,h}^{k}||_{s,\gamma}||\delta_{\lambda,h}^{k}-\delta_{\lambda,h}^{k-1}||_{\gamma}.
\end{alignat}
We apply  Young's inequality to~\eqref{cont_s_gamma} and replace the result in~\eqref{lscheme_error_estimate44},
while choosing $L_{\gamma}={L_{\xi}}/{2(1-\zeta)}$,
\begin{align}  \label{lscheme_error_estimate99}
\left(1+\tau \dfrac{L_{\gamma,p}}{2}\right)&||\delta_{\lambda,h}^{k}||_{\gamma}^{2}
+\tau ||\delta_{\lambda,h}^{k}||_{s,\gamma}^{2}
+\tau \dfrac{L_{\gamma,p}}{2}||\delta_{\lambda,h}^{k}-\delta_{\lambda,h}^{k-1}||_{\gamma}^{2}
+\left(\dfrac{L_{\gamma,u}}{2}+c_{\vK,\gamma}\right)\tau ||\delta_{\vecu,h}^{k}||_{\gamma}^{2} \nonumber\\
&\leq \left(\dfrac{L_{\gamma,u}}{2}-\zeta \xi_{m}\right)\tau ||\delta_{\vecu,h}^{k-1}||_{\gamma}^{2}
+\dfrac{L_{\gamma,p}}{2}\tau ||\delta_{\lambda,h}^{k-1}||_{\gamma}^{2}
+\dfrac{C_{\textn{dTr}}^2}{c_{\vK}}h^{-1}\dfrac{\tau }{2}||\delta_{\lambda,h}^{k-1}-\delta_{\lambda,h}^{k}||_{\gamma}^{2}
+\dfrac{\tau }{2}||\delta_{\lambda,h}^{k}||^{2}_{s,\gamma}.
\end{align}
\FIX{We choose}{Choosing} $L_{\gamma,p}\geq {C_{\textn{dTr}}^2}/({c_{\vK}}h)$
\FIX{we end up with}{gives} the contraction~\eqref{lddscheme_contraction}. The
rest of the proof is as in Theorem~\ref{thm:ldd}.
\end{proof}
\FIX{We complete our analysis of~Algorithm~\ref{splitting_ldd} by giving}{Finally, we give} alternative
convergence results when $(h,{1}/{\alpha_{\gamma}})\rightarrow0$, leading
to extremely large stabilization parameter $L_{\gamma,p}$ which  deteriorates
the convergence rate of ItLDD scheme. These results are important to show
the robustness of the ItLDD scheme for extreme physical or
discretization situations.
\begin{prop}[Alternative convergence results]
 If $L_{\gamma,p}=0$, and $L_{\gamma,u}={L_{\xi}}/{2(1-\zeta)}$ with $\zeta\in[0,1)$,
Algorithm~\ref{splitting_ldd} is convergent under the
constraint on the  time
step
$\tau \leq {1}/{\alpha_{\gamma}}$.
The following estimate for~Algorithm~\ref{splitting_ldd} holds true and defines
a contraction
\begin{alignat}{2}
\label{estimate1_prop}\left(1-\dfrac{\alpha_{\gamma}}{2}\tau \right)||\delta_{\lambda,h}^{k}||_{\gamma}^{2}+
\left(\dfrac{L_{\gamma,u}}{2}+c_{\vK,\gamma}\right)\tau ||\delta_{\vecu,h}^{k}||_{\gamma}^{2}
\leq \left(\dfrac{L_{\gamma,u}}{2}-\zeta \xi_{m}\right)\tau ||\delta_{\vecu,h}^{k-1}||_{\gamma}^{2}+
\dfrac{\alpha_{\gamma}}{2}\tau ||\delta_{\lambda,h}^{k-1}||_{\gamma}^{2}.
\end{alignat}
 Moreover, if $\alpha_{\gamma}\rightarrow\infty$,   Algorithm~\ref{splitting_ldd} is convergent as \FIX{the ratio}{}
${\tau }/{h}\leq {c_{\vK}}/{C_{\textn{dTr}}^{2}}(=\vcentcolon{C^{-1}_{\gamma,s}})$ holds true, and the resulting
estimate is  a contraction given by
\begin{alignat}{2}
\label{estimate2_prop} \left(1-\dfrac{C_{\gamma,s}}{2}\dfrac{\tau }{h}\right)||
\delta_{\lambda,h}^{k}||^{2}+\left(\dfrac{L_{\gamma,u}}{2}+c_{\vK,\gamma}\right)\tau ||\delta_{\vecu,h}^{k}||^{2}
\leq \left(\dfrac{L_{\gamma,u}}{2}-\zeta \xi_{m}\right)\tau ||\delta_{\vecu,h}^{k-1}||^{2}+ \dfrac{C_{\gamma,s}}{2}\dfrac{\tau }{h}||\delta_{\lambda,h}^{k-1}||^{2}.
\end{alignat}
\end{prop}
\begin{proof}
We let $L_{\gamma,p}=0$ in the estimate~\eqref{s_original} to get
\begin{align*}
||\delta_{\lambda,h}^{k}||_{\gamma}^{2}+\tau (b_{\gamma}(\vecu_{h, \gamma}^{n,k-1})-b_{\gamma}(\vecu_{h, \gamma}^{n}),\delta_{\vecu,h}^{k})_{\gamma}
+L_{\gamma,u}\tau (\delta_{\vecu,h}^{k}-\delta_{\vecu,h}^{k-1},\delta_{\vecu,h}^{k})_{\gamma} +\tau a_{\gamma}(\delta_{\vecu,h}^{k},\delta_{\vecu,h}^{k})
=-\tau s_{\gamma}(\delta_{\lambda,h}^{k-1},\delta_{\lambda,h}^{k}).
\end{align*}
With the same techniques
used to get~\eqref{lscheme_error_estimate44}, we get for
$L_{\gamma,u}={L_{\xi}}/{2(1-\zeta)}$ with $\zeta\in[0,1)$,
 \begin{align}\label{lscheme_error_estimate66_pro}
\left(1+\tau \dfrac{L_{\gamma,p}}{2}\right)||\delta_{\lambda,h}^{k}||_{\gamma}^{2}
\!+\!\tau ||\delta_{\lambda,h}^{k}||_{s,\gamma}^{2}
\!+\!\left(\dfrac{L_{\gamma,u}}{2}+c_{\vK,\gamma}\right)\tau
||\delta_{\vecu,h}^{k}||_{\gamma}^{2}
\!\leq \!\left(\dfrac{L_{\gamma,u}}{2}-\zeta \xi_{m}\right)\tau ||\delta_{\vecu,h}^{k-1}||_{\gamma}^{2}
-\tau s_{\gamma}(\delta_{\lambda,h}^{k-1},\delta_{\lambda,h}^{k}).
\end{align}
The coupling term in the
right-hand side is now estimated as $ s_{\gamma}(\delta_{\lambda,h}^{k-1},\delta_{\lambda,h}^{k})|\leq \alpha_{\gamma}||\delta_{\lambda,h}^{k-1}||_{\gamma}||\delta_{\lambda,h}^{k}||_{\gamma} $,
where we used~\eqref{def_inter_oper_disc}. Applying Young inequality
and inserting the result in~\eqref{lscheme_error_estimate66_pro},
we \FIX{infer}{get}~\eqref{estimate1_prop}. \FIX{That of t}{T}he second estimate~\eqref{estimate2_prop},
when $\alpha_{\gamma}\rightarrow\infty$, is \FIX{obtained similarly}{similar} to~\eqref{estimate1_prop},
but with using~\eqref{def_inter_oper_cont} to bound the
coupling term.
\end{proof}
\FIX{\begin{rem}[Time step vs stabilization]}{}
The constraint on \FIX{the ratio}{} ${\tau}/{h}$ is less restrictive than the
constrain on \FIX{the stabilization parameter}{} $L_{\gamma,p}$ in
Lemma~\ref{thm:cvlscheme_limit_case}. \FIX{We also note that}{Moreover,} the constraint\FIX{on the
time step}{} $\tau\leq {1}/{\alpha_{\gamma}}$ may have the same implication on
the convergence rate as taking $L_{\gamma,p}\geq \alpha_{\gamma}$ in
Theorem~\ref{thm:ldd}. In practice,  the choice between the two constraints may
depend on the\FIX{physical}{} situation. All the results show a strong correlation
between the Robin parameter $\alpha_{\gamma}$, \FIX{and }{}the time step $\tau$ (or
${\tau}/{h}$) \FIX{ or}{and} the stabilization parameter $L_{\gamma,p}$.

{
\section{\FIX{The LDD Iterations with Multiscale Flux Basis Implementation}{Implementation of LDD schemes using Multiscale Flux Basis}}\label{sec:Mufbi}
\FIX{}{
Lastly, we propose an implementation\FIX{ of application}{} of the inter-dimensional map
$\mathcal{S}^{\textn{RtN}}_{\gamma}$ based on the construction of a multiscale
mortar flux basis \cite{ahmed2018multiscale,MR3577939}. We want to solve  the reduced scheme in
Definition~\ref{def:reduced scheme} for different  physical and $L$-scheme parameters, for variation of non-Darcy flow problems by
changing the non-linearity $\xi$, and to compare the two LDD
schemes. In Section \ref{sec:algorithms}, we have seen that the dominant computational cost of both Algorithm~\ref{monolithic_ldd} and
Algorithm~\ref{splitting_ldd} comes from the subdomain solves by evaluating the action
of $\mathcal{S}_{\gamma}^{\textn{RtN}}$  using Algorithm~\ref{Eval_op} (step
2(b)). We have also seen that the condition
number~\eqref{condition_numb_alpha}-\eqref{condition_numb_noalpha} of the
linearized interface problem grows with refining the grids or increasing
$\alpha_{\gamma}$. Therefore, in the case of large number of iterations, we want to have an efficient method to evaluate the action
of $\mathcal{S}_{\gamma}^{\textn{RtN}}$}.

The  construction of the inter-dimensional mapping is achieved by  pre-computing and storing the
flux subdomain responses, called \textit{multiscale flux basis} (MFB), associated with each
fracture pressure degree of freedom on each subdomain \cite{MR2557486}.
\FIX{Define}{Let}
$(\Phi^{\ell}_{\hg})^{\mathcal{N}_{\hg}}_{\ell=1}$ be a set of basis
functions on\FIX{the  interface pressure space }{} $M_{\hg}$, where $\mathcal{N}_{\hg}$
is \FIX{the number of pressure degrees of freedom on $\gamma$}{the dimension of $M_{\hg}$}. \FIX{}{Then, each function $ \lambda_{\hg} \in M_{\hg} $
can be represented as $
 \lambda_{\hg} = \sum_{\ell=1}^{\mathcal{N}_{\hg}} \lambda^{\ell}_{\hg}\Phi^{\ell}_{\hg}$.}
\FIX{and compute  the   }{}The MFB functions corresponding to
$(\Phi^{\ell}_{\hg})^{\mathcal{N}_{\hg}}_{\ell=1}$ \FIX{using the following
algorithm:}{ are computed as follows.}
\begin{algo}[Assembly of the multiscale flux basis]\label{Eval_MS}~
{
\setlist[enumerate]{topsep=0pt,itemsep=-1ex,partopsep=1ex,parsep=1ex,leftmargin=1.5\parindent,font=\upshape}
\begin{enumerate}
    \item Enter the basis $(\Phi^{\ell}_{\hg})^{\mathcal{N}_{\hg}}_{\ell=1}$. Set $\ell\eq0$.
    \item \textn{\textbf{Do}}
    \begin{enumerate}
        \item Increase $\ell\eq \ell +1$.
        \item Project $\Phi^{\ell}_{\hg}$ on the  subdomain boundary,
            $\lambda_{h,i}^{\ell} =\mathcal{Q}_{h,i}(\Phi^{\ell}_{\hg})$.
        \item Solve  problem~\eqref{subdo1_weak_mixed_formulation_disc} in each subdomain $\Omegai$.
        \item Project the resulting flux onto the pressure space on the
            fracture,
            $ \Psi_{\hg,i}^{\ell} \eq - \mathcal{Q}^{\textn{T}}_{h,i} \vecu^{*}_{\hi}(\lambda_{h,i}^{\ell})\cdot\vecn_{i}. $
    \end{enumerate}
    \textn{\textbf{While} $\ell\leq \mathcal{N}_{\hg}$.}
\end{enumerate}
}
\end{algo}
\FIX{Once the \FIX{multiscale flux basis functions}{MFB} is constructed for each subdomain,
the action of
$\mathcal{S}^{\textn{RtN}}_{\gamma}$  is replaced by a linear combination of the multiscale flux basis functions $\Psi_{\hg,i}^{\ell}$. Specifically, at any time step $n\geq 1$, and at any iteration $m\geq 1$ of any of the algorithms,  for an interface datum $\lambda^{n,m}_{\hg}\in M_{\hg}$, we have $\lambda^{n,m}_{\hg}=\sum_{\ell=1}^{\mathcal{N}_{\hg}} \lambda^{n,m,\ell}_{\hg}\Phi^{\ell}_{\hg}$, and for $i\in\{1,2\}$,}{}
\FIX{}{Hence, the action of $\mathcal{S}^{\textn{RtN}}_{\gamma}$ is given by
\begin{equation}
\label{S_g_action}
    \mathcal{S}^{\textn{RtN}}_{\gamma}(\lambda_{\hg}) = \sum_{i=1}^2 \mathcal{S}^{\textn{RtN}}_{\gamma,i}(\lambda_{\hg}) = \sum_{i=1}^2 \sum_{\ell=1}^{\mathcal{N}_{\hg}} \lambda^{\ell}_{\hg}\mathcal{S}^{\textn{RtN}}_{\gamma,i}(\Phi^{\ell}_{\hg}) = \sum_{i=1}^{2} \sum_{\ell=1}^{\mathcal{N}_{\hg}} \lambda^{\ell}_{\hg}\Psi^{\ell}_{\hg,i}.
\end{equation}
This holds true at any time step and at any iteration of MoLDD or
ItLDD. The use of MFB removes the dependence between the total number
of subdomain solves and the number of their iterations.}
\FIX{We continue with some important remarks on the applicability of the MFB.}{}

\FIX{\begin{rem}[On the MFB - gain]}{}

\FIX{\begin{rem}[On the MFB - cost]}{}
\begin{rem}[Fracture network]
%
For the case of a fracture network, say $\gamma=\cup_{i\neq j}
\gamma_{ij}$, where $\gamma_{ij}$ is the fracture between the subdomain
$\Omega_{i}$ and $\Omega_{j}$, the previous basis reconstruction is then
applied independently on each fracture.
\end{rem}
\FIX{}{
\begin{rem}[Inner solver]
 Our numerical examples have a single one-dimensional fracture and, thus, we only use the direct methods to solve the interface system. For large-scale problems with many fractures or in three spatial dimensions, we emphasize the need for an iterative solver, such as GMRes. However, we see from Theorem \ref{thm:cond_num} that the condition number of the DD system depends on $ h $ and $ \alpha_{\gamma} $, which in turn influences the number of iterations of the iterative solver. To retain robustness, we can use a preconditioner~\cite{Ana2019precond,Antonietti2019} or a coarse mortar space that is compensated by taking higher order mortars~\cite{MR3869664,MR2306414}.
\end{rem}}

\section{Numerical examples}\label{sec:examples}
{In this section, we present several test cases to show how the schemes behave  (1) for different  values for numerical and
physical parameters (2) with coarsening/refining mortar grids (3) on extensions to other governing
equations.}  We subsequently study the value of $ L_{\gamma,
opt} $ in the MoLDD scheme and the relationship between $ L_{\gamma, u} $ and $
L_{\gamma, p} $ in the ItLDD scheme. The performance of schemes is measured in the
overall number of iterations needed for each scheme to reach the stopping
criteria. In the implementation of both schemes, we consider that the solution
has converged if the relative error of the fracture solution is less than $
10^{-5} $, if the value at the previous iteration step \FIX{is not zero}{is not below machine precision}. Otherwise we
use the absolute error. \FIX{}{We use a direct method (LU decomposition) to solve the linearized interface problem since the size of the system is relatively small.}

To keep the presentation simple, we consider the domain and several parameters in common in
all the examples in relation to the first test case in \cite{Martin2005}. The domain $\Omega \eq (0, 2) \times (0, 1)$
is intersected with a fracture defined as $\gamma \eq \{x = 1\}$.
On the boundaries of the rock matrix $\{x = 0\}$ and $\{x =2\}$ we
impose pressure boundary condition with values $0$ and $1$, respectively. We set zero flux boundary condition on the
rest of $\partial \Omega$. The boundary of the fracture at the tips $\{y=1\}\cap\partial\gamma$ and $\{y=0\}\cap\partial\gamma$ inherits the pressure boundary conditions from the rock matrix. The examples are set on the time interval $ I = (0, 1) $ with homogeneous pressure initial condition. As for the physical parameters, we take the
permeability matrix for the bulk $\vK_{i} = \mathbf{I}$, while the source terms $f_{i}$
and $f_\gamma$ are equal to zero.

First, we consider the Forchheimer flow model
where the non linear term is
$\xi(\vecu_\gamma) = \beta \vert{\vecu_\gamma}\vert\vecu_\gamma$.
{The parameter $ \beta $ is a fluid dependent non-negative scalar known as the Forchheimer
coefficient, and $\vert\cdot\vert$ denotes the Euclidean vector norm $\vert{\vecu_\gamma}\vert^{2}=\vecu_{\gamma}\cdot\vecu_{\gamma}$.} It is straightforward to see that $ \xi $ is a simply increasing function and satisfies condition \ref{ass_xi}. For more details see \cite{MR2425154,MR3264361} and references therein.
\subsection{Stability with respect to  the user-given parameters}\label{subsec:stability}
We first study the performances of MoLDD and ItLDD solvers by
 varying the time step $ \tau $, the mesh size $ h $, and the $L$-scheme parameters $(L_{\gamma, u},L_{\gamma, p})$.  We  let  $\vK_{\gamma} = 1$, $\alpha_\gamma = 10^4$ and  $ \beta = 1 $ and according to the theoretical results,  the
  L-scheme parameters  are given by $ L_{\gamma, u} \approx 1 $ and  $ L_{\gamma, p} = 10^3 $. Results in Table \ref{tab:example1} report the number of
iterations required by  the two LDD solvers while varying the mesh size $ h $ and time step size $ \tau $. Each column of the tables represent results for a time step $ n$.
\begin{table}[tbp]
    \centering
    \begin{tabularx}{0.293\textwidth}{l|r|r|r|r|r}
        \hline
        ${h\backslash n}$ & 1 & 2 & 3 & 4 & 5 \\ \hline
        $2^{-1}$ &17&12&11&10&9 \\
        $2^{-3}$ &17&11&10&9&8 \\
        $2^{-5}$ &17&11&10&9&8 \\ \hline
    \end{tabularx}
    \begin{tabularx}{0.679\textwidth}{l|l|l|l|l|l|l|l|l|l|l|l|l|l|l|l|l}
        \hline
        ${\tau}$ & \multicolumn{16}{c}{}\\ \hline
        $2^{-2}$   &\multicolumn{4}{c|}{17}&\multicolumn{4}{c|}{11}&
                    \multicolumn{4}{c|}{10}&\multicolumn{4}{c}{9}\\ \hline
        $2^{-3}$   &\multicolumn{2}{c|}{17}&\multicolumn{2}{c|}{10}&
                    \multicolumn{2}{c|}{9}&\multicolumn{2}{c|}{9}&
                    \multicolumn{2}{c|}{8}&\multicolumn{2}{c|}{8}&
                    \multicolumn{2}{c|}{7}&\multicolumn{2}{c}{6} \\ \hline
        $2^{-4}$   &16&10&9&9&9&8&8&8&7&7&7&6&6&6&5&5 \\ \hline
    \end{tabularx}
    \vskip 5mm
    \begin{tabularx}{0.293\textwidth}{l|r|r|r|r|r}
        \hline
        ${h\backslash n}$ & 1 & 2 & 3 & 4 & 5 \\\hline
        $2^{-1}$ & 17 & $\,$ 8 & $\,$ 7 & $\,$ 7 & 6 \\
        $2^{-3}$ & 17 & $\,$ 9 & $\,$ 8 & $\,$ 7 & 7 \\
        $2^{-5}$ & 17 & $\,$ 9 & $\,$ 8 & $\,$ 7 & 7 \\ \hline
    \end{tabularx}
    \begin{tabularx}{0.679\textwidth}{l|l|l|l|l|l|l|l|l|l|l|l|l|l|l|l|l}
        \hline
        ${\tau}$ & \multicolumn{16}{c}{}\\\hline
        $2^{-2}$   &\multicolumn{4}{c|}{17}&\multicolumn{4}{c|}{9}&
                    \multicolumn{4}{c|}{8}&\multicolumn{4}{c}{7}\\ \hline
        $2^{-3}$   &\multicolumn{2}{c|}{17}&\multicolumn{2}{c|}{10}&
                    \multicolumn{2}{c|}{9}&\multicolumn{2}{c|}{8}&
                    \multicolumn{2}{c|}{8}&\multicolumn{2}{c|}{7}&
                    \multicolumn{2}{c|}{6}&\multicolumn{2}{c}{6} \\ \hline
        $2^{-4}$   &16&10&9&9&9&8&8&8&7&7&7&6&6&6&5&5 \\ \hline
    \end{tabularx}
    \caption{Results for the example of Subsection \ref{subsec:stability}. Top two tables correspond to solving with the MoLDD scheme, while bottom two correspond to solving with the ItLDD scheme.
        On the left we report the number of iterations by varying the mesh size
        $h$ for a fixed time step $\tau=2^{-4}$, while on the right depending on the time step size $ \tau $ for a fixed mesh size $h=2^{-5}$.}%
    \label{tab:example1}
\end{table}

Regardless of the choice of scheme, we can observe that the number of
iterations is independent from the mesh size and slightly dependent on the time
step size. The reason for the latter might be related to the fact that we
consider the solution at previous iteration as the initial guess for the next iteration. Thus, by decreasing the time step size, the variation of the solution
between steps varies less and so the number of iterations.  Overall, the sequential
 ItLDD and the monolithic   MoLDD solvers behave  similarly; one can also see a slightly better results  for the  iterative   solver in Table~\ref{tab:example1} (left).  Note that   any
comparison of the two solvers does not make sense for the simple reason
that   the amounts of stabilization fixed by  $L_{\gamma,p}$ and $L_{\gamma,u}$
are  not   yet optimal. Another explanation\FIX{, may also be,}{ is that} the
amount of stabilization in the monolithic solver MoLDD is set solely by
$L_{\gamma,u}$, in contrast to the iterative solver ItLDD where \FIX{two
stabilization parameters}{} $L_{\gamma,p}$ and $L_{\gamma,u}$  are used.

Finally, we recall that  with the use of the \FIX{
multiscale flux basis}{MFB}, the  computational  costs of the two solvers is practically
the same.
The main  cost is done  offline\FIX{ using the
multiscale flux basis }{}, which is mostly related to the number of mortar degrees of freedom. As an example, the computational cost needed  to draw the
results for $h=2^{-5}$ of the two right tables in
Table~\ref{tab:example1} is approximately equal to 96 subdomain solves $(\textn{Num. of
DOF} * \textn{Num. of subdo.} + 2 N)$, where two  solves per time step are  required  to form the right-hand side
in~\eqref{lscheme_system_fracture} (for MoLDD)
and~\eqref{lddscheme_system_fracture} (for ItLDD). Without MFB,
the cost should be
$\sum_{n=1}^{N} \sum_{k=1}^{N^{n}_{\textn{Lin}}}N_{\textn{dd}}^{k} +2N$, where $N^{n}_{\textn{Lin}}$ is the number of iterations of the $L$-scheme, and $N_{\textn{dd}}^{k}$ denotes the number of DD iterations \FIX{ (GMRes or any Krylov solver)}{}. Thus, if we assume a fixed $N_{\textn{dd}}^{k}$ along all the simulation, say $N_{\textn{dd}}^{k}=2$, this number will be at least 1012 subdomain solves, so that  with MFB we make a save  of approximately $91\%$ of the total subdomain solves.}

In Figure~\ref{fig:example1_mono}, we plot  the number of iterations
\FIX{with various realizations}{} of the  user-given $ L_{\gamma, u} $ in MoLDD solver. We consider 100 values of $ L_{\gamma, u} $, from $ 0 $ to $ 2.5 $ with uniform step $ 0.025
$. The other parameters are fixed as follows,  $ \beta = 1$,  $ h = 0.125 $ and
$ \tau = 0.2$. The graph in this figure behaves very similarly to what is
usually observed for the $L$-type schemes (a typical V-shape graph),
highlighting  a numerically optimal value  $ L_{\gamma, opt} $ between $ 0.5 $
and $ 1 $. \FIX{By increasing $L_{\gamma, u}$, the number of iterations slowly
increases, while they increase more drastically for small value of $ L_{\gamma,
u} $. }{The number of iterations increases for small and large values of $ L_{\gamma,
u} $}. This behavior is common for all time steps. We expect such a behavior
when choosing $ L_{\gamma, u} $ close to zero because it directly influences the
contraction factor in \eqref{contraction_lscheme}.
We can also see
that the identified parameter $ L_{\gamma, opt}\approx 1 $ is close to the
optimal one. 
\begin{figure}[tbp]
    \centering
    \subfloat[\FIX{}{MoLDD scheme}]{
    \includegraphics[width=0.33\textwidth]{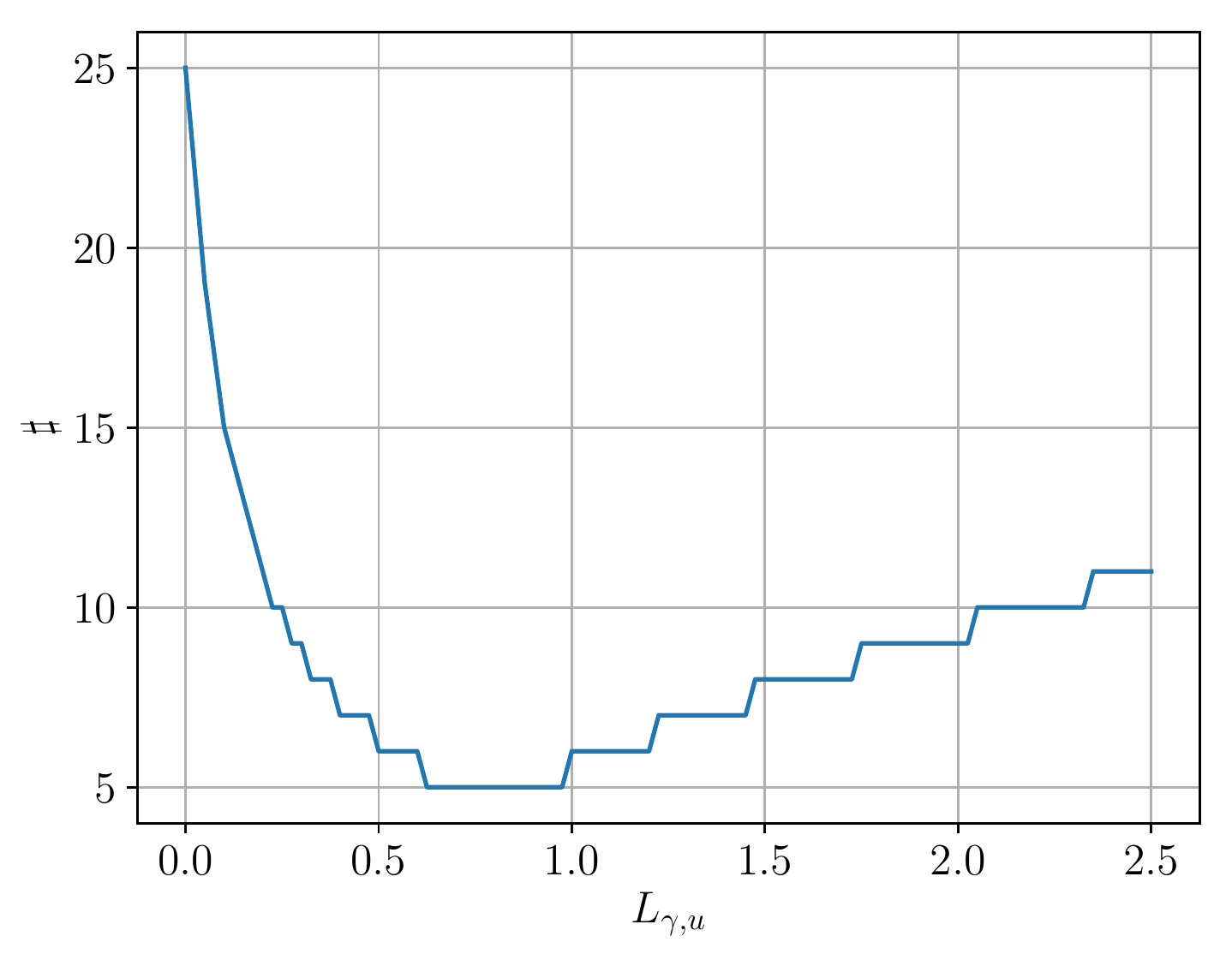}\label{fig:example1_mono}}%
    \hspace{0.1\textwidth}%
    \subfloat[\FIX{}{ItLDD scheme}]{\includegraphics[width=0.33\textwidth]{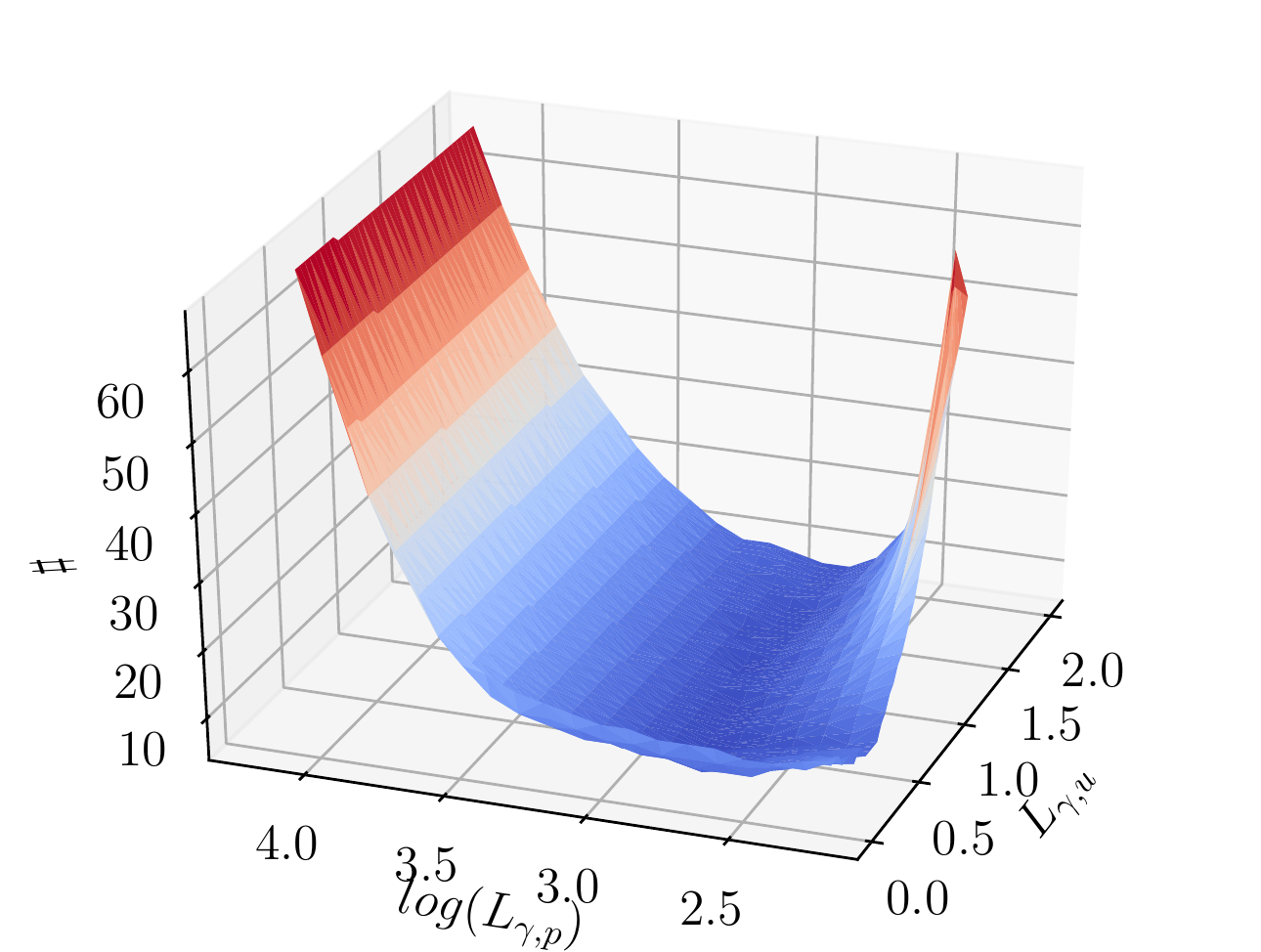}\label{fig:example1_iter}}%
    \caption{Results for the example of Subsection \ref{subsec:stability}\FIX{
        using MoLDD scheme}{}.
        We report the number of iterations $\sharp$ for different values of
        $L_{\gamma,u}$ \FIX{}{ and $ L_{\gamma,p} $}.
        \FIX{On the left for the first time step, in the centre for the third, and on
        the right for the last time step.}{In both cases we report the third
        time step.}}%
\end{figure}
\FIX{
\begin{figure}[tbp]
    \centering
    \includegraphics[width=0.33\textwidth]{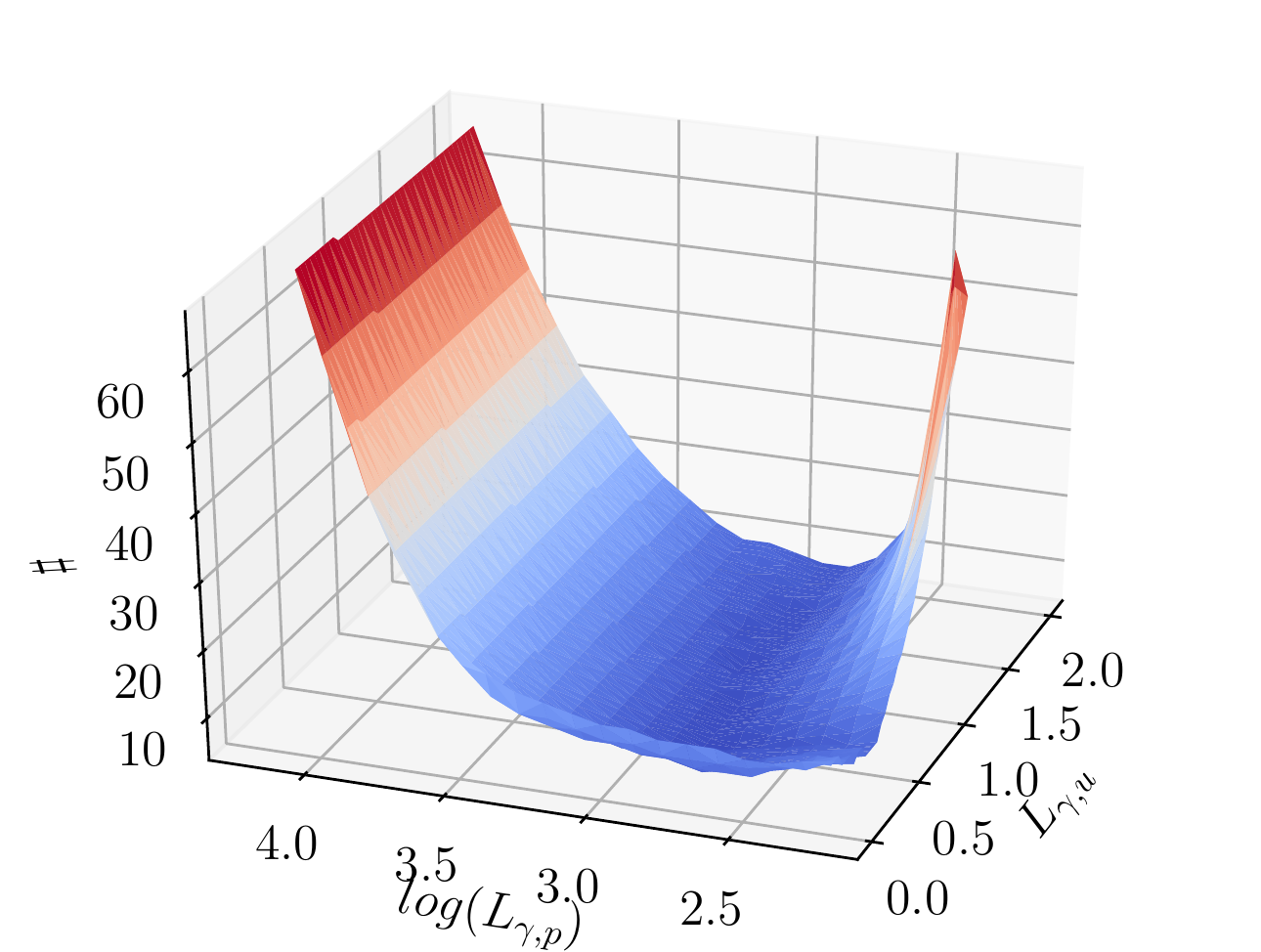}%
    \includegraphics[width=0.33\textwidth]{forchheimer_L_Lp_3}%
    \includegraphics[width=0.33\textwidth]{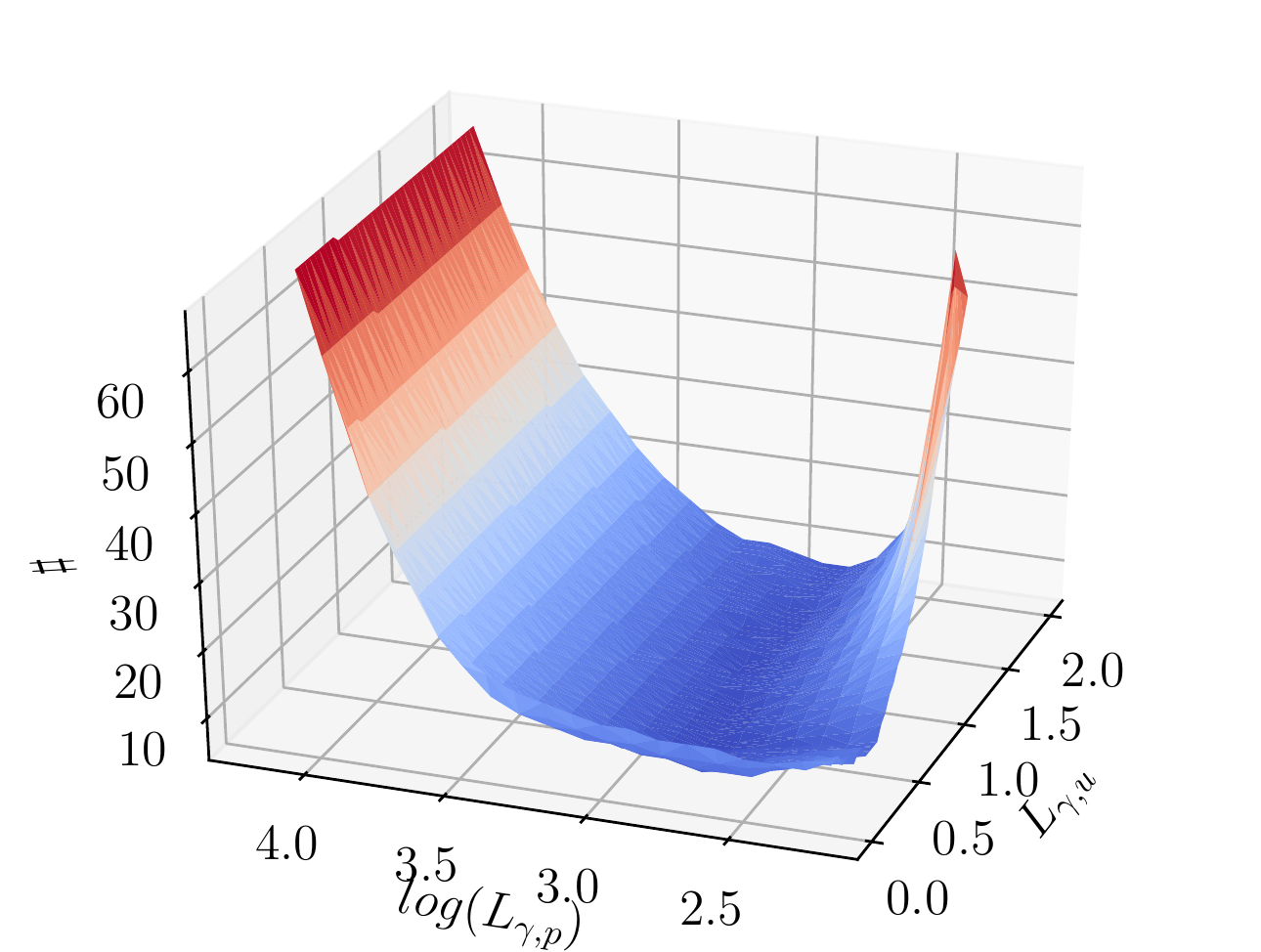}%
    \caption{Results for the example of Subsection \ref{subsec:stability} using ItLDD scheme.
        We report the number of iterations $\sharp$ for different values of $ L_{\gamma,u} $ and $ L_{\gamma,p} $.
        On the left for the first time step, in the centre for the third, and on
        the right for the last time step.}%
    \label{fig:example1_iter}
\end{figure}
}{}In Figure \ref{fig:example1_iter}, we show the performance of the ItLDD solver with
regards to changing parameters $ L_{\gamma, u} $ and $ L_{\gamma, p} $. We
consider $ L_{\gamma, u} $ taking 50 values uniformly distributed on the
interval $ (0, 2.5) $, while $ L_{\gamma, p} = 10^x $, where $ x $ are 21
equidistant values on the interval $ (2.2, 4.2) $ with step $ 0.1 $. We can observe on the surface plots that there is a global
minimum that determines the optimal choice for  $ L_{\gamma, u} $ and
$ L_{\gamma, p} $. For example, the minimum number of iterations for this flow
model is $ 5 $ for $ L_{\gamma, u} $ between $ 0.59 $ and $ 1.1 $ and $
\log(L_{\gamma, p}) $ between $ 2.8 $ and $ 3 $, in all time steps.  Similar to
\FIX{the monolithic approach}{MoLDD}, the number of iterations required by the ItLDD solver increases when the L-scheme parameters assume low values. Particularly, the scheme diverges when $ L_{\gamma, p} \leq 10^2 $. In the analysis of the
scheme we require that $  L_{\gamma, p} \geq \alpha_{\gamma} $, but the lower
values also allow a good convergence behavior, concluding that the theoretical
lower bound is possibly too strict, but it certainly exists. Therefore, in
practice, we can slightly relax the bounds on the L-scheme parameters to still
obtain good performance of the solver. It is also relevant to mention that the
normal permeability constant $ \alpha_{\gamma} = 10^4 $ is sufficiently  large
 to apply the limit case results in \Cref{thm:cvlscheme_limit_case}.

Crucially,\FIX{we want to mention that}{} the computational cost\FIX{of the realizations}{}
needed to draw Figure~\ref{fig:example1_mono} and~\ref{fig:example1_iter},  is
exactly equal to only one realization with fixed  $(L_{\gamma,u}, L_{\gamma,p}
)$,\FIX{permitting easier calculation of  these parameters, and}{} confirming the
utility of the MFB on fixing the total cost and avoiding any computational
overhead if these parameters are not optimal.

\subsection{Robustness with  respect to the physical
parameters}\label{subsec:robust}%
\FIX{In this set of test examples, w}{W}e want to show \FIX{}{now} the robustness of the algorithms
with respect to $\alpha_{\gamma}$ and $\beta$\red{;} \FIX{Note that}{} $\alpha_{\gamma}$
controls the strength of the fracture-matrix coupling, while $\beta$ controls
the strength of the non-linearity. We fix the mesh size $ h = 0.125 $ and the
time step $\tau=2^{-3}$.
\begin{table}[tbp]
    \centering
    \begin{tabular}{l|r|r|r|r|r}
        \hline
        ${\beta\backslash n}$ & 1 & 2 & 3 & 4 & 5 \\ \hline
        0.1 & 17 & 9 & 8 & 7 & 7   \\
        1   & 17 & 9 & 8 & 7 & 7    \\
        100 & 9 & 8 & 7 & 6 & 5  \\ \hline
    \end{tabular}
    \hspace*{0.05\textwidth}
    \begin{tabular}{l|r|r|r|r|r}
        \hline
        ${\beta\backslash n}$ & 1 & 2 & 3 & 4 & 5 \\ \hline
        0.1 & 17 & 11 & 10 & 9 & 8    \\
        1   & 17 & 11 & 10 & 9 & 8    \\
        100 & 14 & 10 & 9 & 9 & 8 \\ \hline
    \end{tabular}
    \\
        \vspace*{0.05\textwidth}
      \begin{tabular}{l|r|r|r|r|r}
        \hline
        ${\alpha_\gamma\backslash n}$ & 1 & 2 & 3 & 4 & 5 \\ \hline
            $10^2$    & 17 & 9 & 8 & 7 & 7 \\
            $10^4$    & 17 & 9 & 8 & 7 & 7 \\
            $10^6$    & 17 & 9 & 8 & 7 & 7 \\
            $10^8$    & 17 & 9 & 8 & 7 & 7 \\\hline
    \end{tabular}
    \hspace*{0.05\textwidth}
    \begin{tabular}{l|r|r|r|r|r}
        \hline
        ${\alpha_\gamma\backslash n}$ & 1 & 2 & 3 & 4 & 5 \\ \hline
            $10^2$    & 17 & 11 & 10 & 9 & 9 \\
            $10^4$    & 17 & 11 & 10 & 9 & 8 \\
            $10^6$    & 17 & 11 & 10 & 9 & 8 \\
            $10^8$    & 17 & 11 & 10 & 9 & 8 \\\hline
    \end{tabular}
    \caption{Results for the example of Subsection \ref{subsec:robust} reporting the number of iterations by varying the parameter $\beta$ (top) and by varying $\alpha_{\gamma}$ (bottom). Left tables correspond to solving with the MoLDD solver, while the right ones correspond to solving with the ItLDD solver. }%
    \label{tab:example1_param}
\end{table}


In Table \ref{tab:example1_param} (top), we study the dependency of the number of iterations on the Forchheimer coefficient $ \beta $. The LDD solvers show a weak dependency of the number of
iterations on the values of $ \beta $, giving slightly better results for larger
values. Overall, the MoLDD solver performs slightly better
then the ItLDD. Bear in mind that changing $ \beta $ directly influences $ L_{\gamma, u} $. This shows that this parameter should be optimized in accordance to the given value of $ \beta$. Again, we suggest that the decrease in number of iterations over time steps may
be due to using the previous iteration solution as the initial guess in the
subsequent iteration. Moreover, all the simulations in Table \ref{tab:example1_param} (top) are run with a fixed computational cost, due to using MFB. Thus, strengthening or
changing the non-linearity effects that may increase the number of iterations
\FIX{increases if the amount of stabilization via}{if} $L_{\gamma,u}$ or
$L_{\gamma,p}$ are not carefully set, has no practical effects on the total
computational cost.
\FIX{Clearly, we can}{We can} conclude that the two solvers remain
robust when strengthening  the non-linearity effects.
\begin{figure}[tbp]
    \centering
    \includegraphics[width=0.2\textwidth]{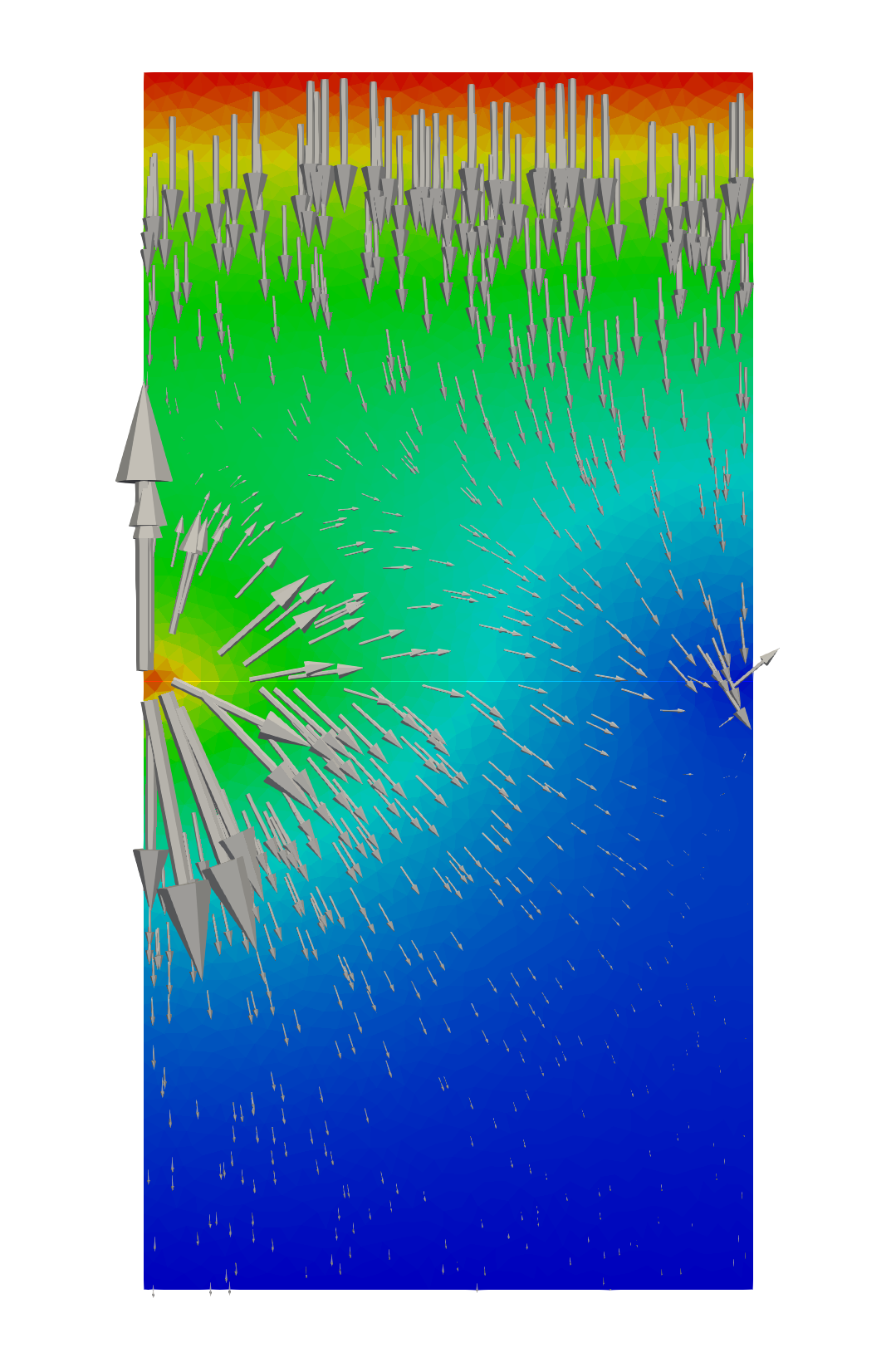}%
    \includegraphics[width=0.2\textwidth]{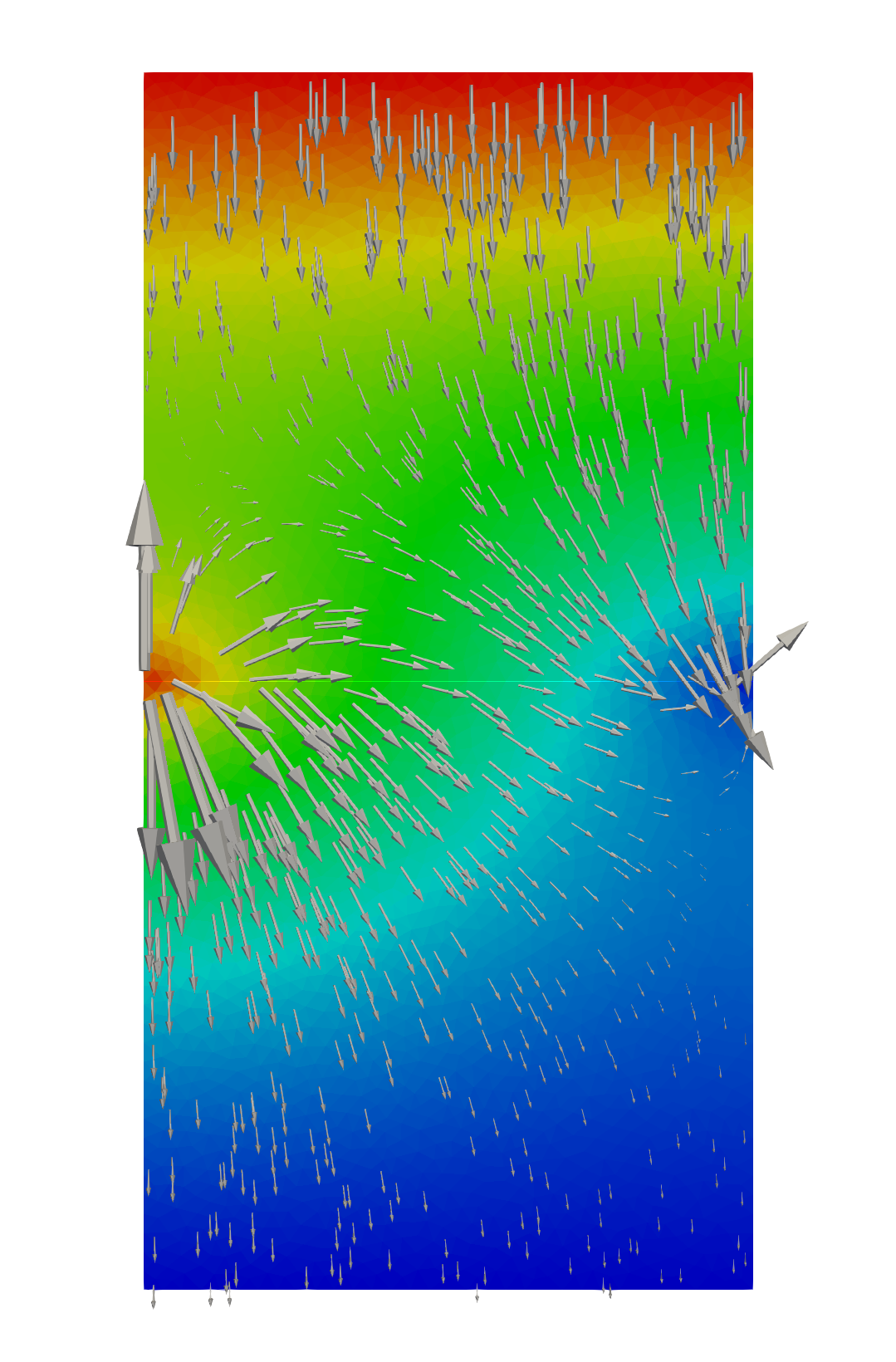}%
    \includegraphics[width=0.2\textwidth]{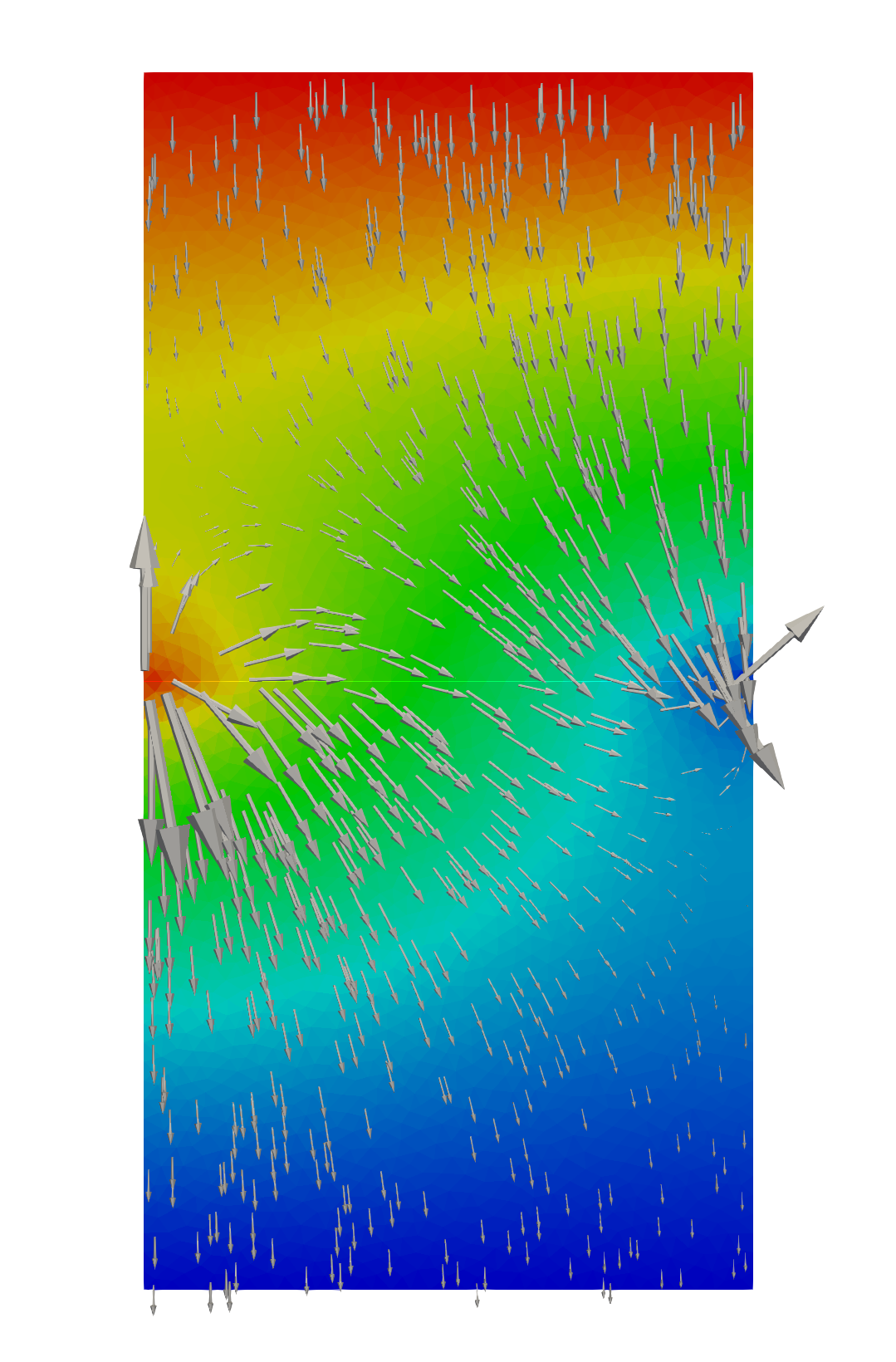}%
    \includegraphics[width=0.2\textwidth]{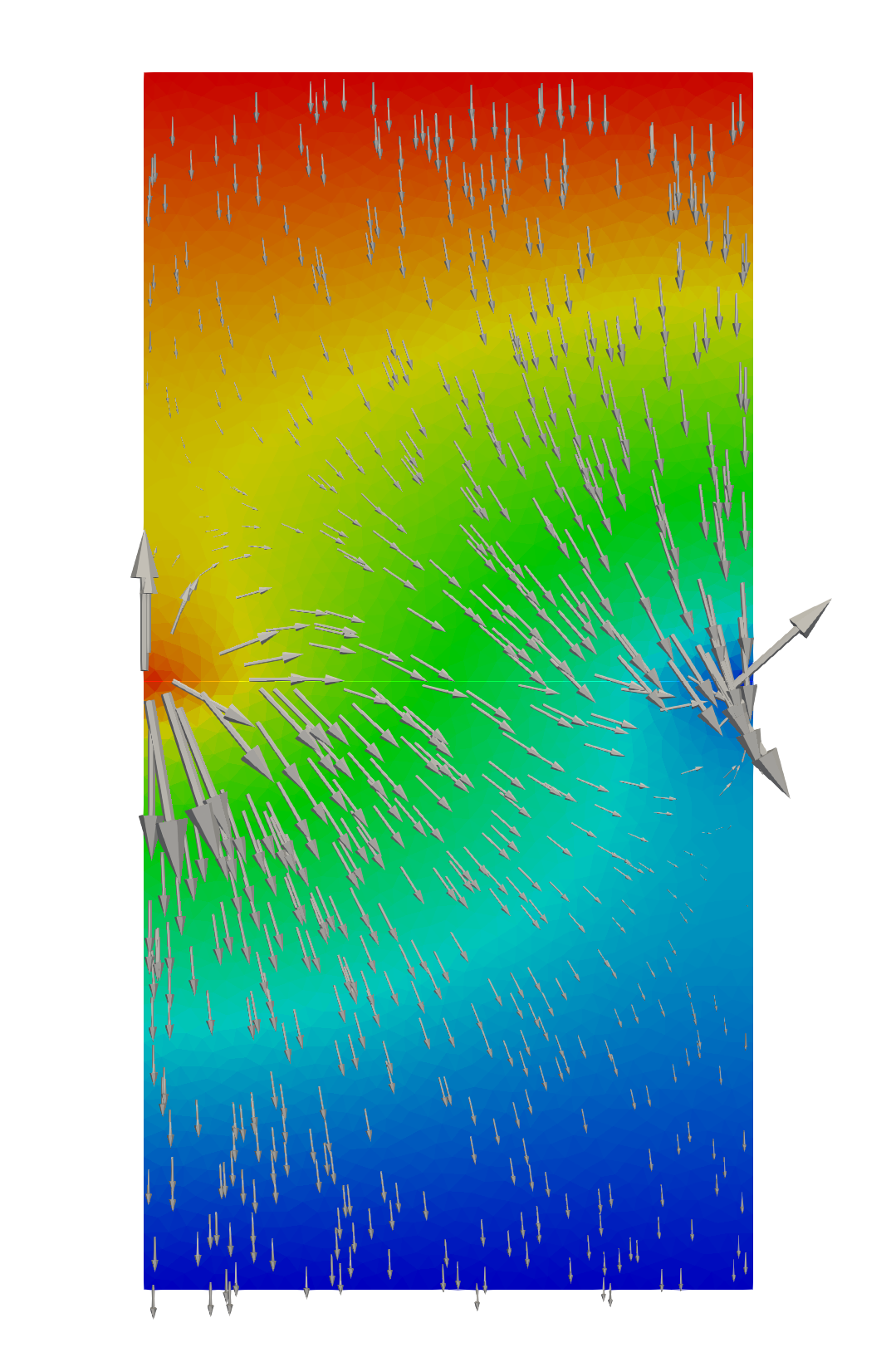}%
    \includegraphics[width=0.2\textwidth]{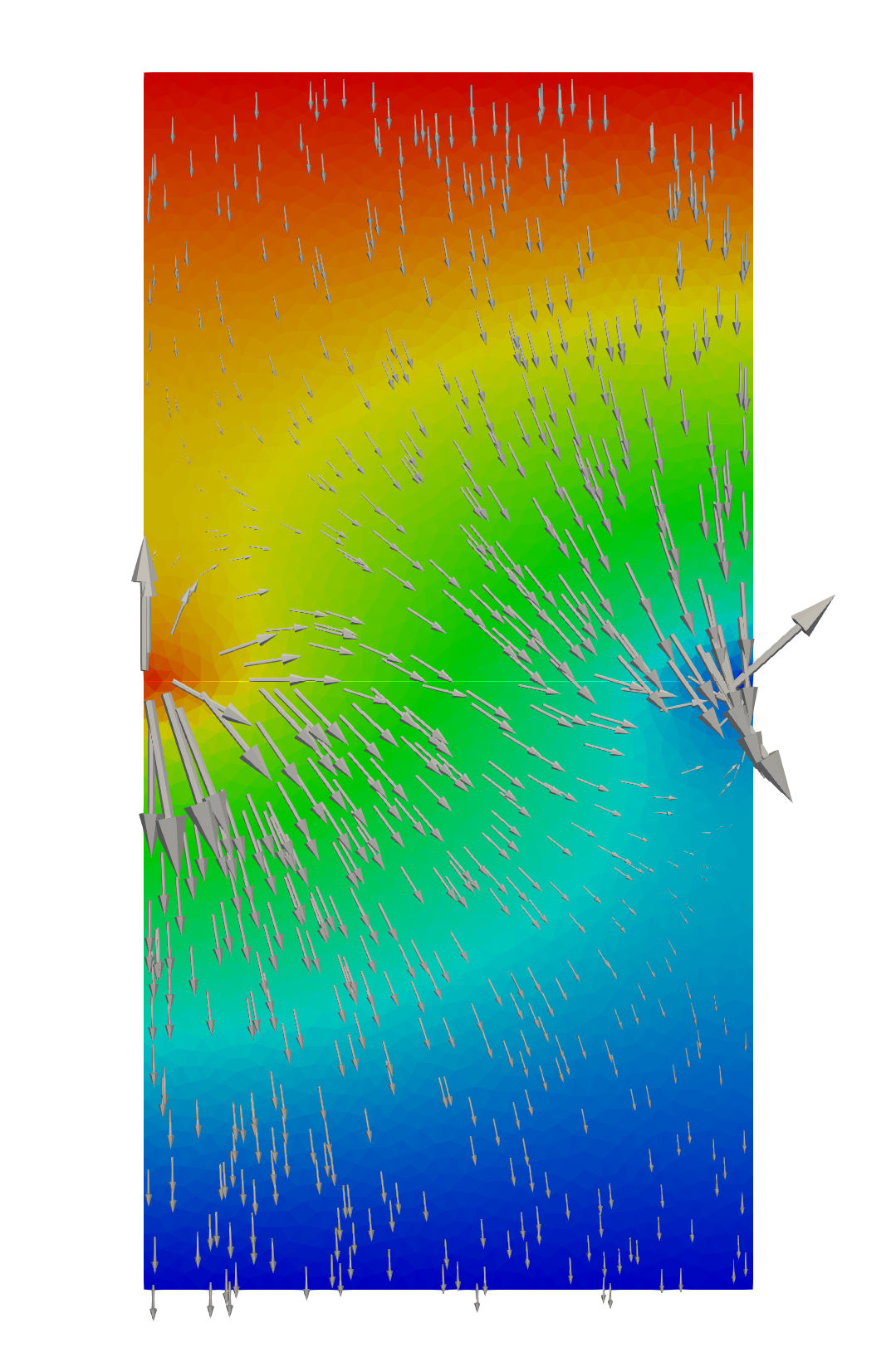}%
    \caption{$p \in [0,1]$ and $\mathbf{u}$ for the example in
    Subsection \ref{subsec:robust} with $\beta=10^2$ and $\alpha_{\gamma}=10^4$.}%
    \label{fig:example1}
\end{figure}
\FIX{Turning now to the effect of the fracture-matrix coupling on the two LDD
solvers}{In  Table \ref{tab:example1_param} (bottom)},  we show the dependency
of the number of iterations on the Robin parameter $\alpha_{\gamma}$. Clearly,
the number of iterations remains stable when strengthening or weakening
matrix-fracture coupling, confirming and concluding the robustness of both
schemes with respect to $\alpha_{\gamma}$. \FIX{Example of s}{A s}olution is reported in
Figure \ref{fig:example1}. \FIX{For t}{T}he computational cost \FIX{of the
results}{} in
Table~\ref{tab:example1_param} (bottom),  any change of $\alpha_{\gamma}$
requires re-computing the MFB. However, this cost remains
fixed  when running and comparing the two LDD solvers for a fixed
$\alpha_{\gamma}$.

{\subsection{Flexibility of coarsening/refining the mortar grids}\label{subsec:multiscale}
In this set of simulations, we consider the case
of  weak inter-dimensional coupling by fixing $\alpha_\gamma = 1$, with  a low
permeable fracture with $\vK_{\gamma}=10^{-4}\II$.   We fix the following
parameters: $h=2^{-5}$ (on the matrix), $L_{\gamma,u}=1$, $L_{\gamma,p} = 2\cdot
10^{2}$ and $\beta= 1$. We allow for a coarse scale of the mortar grids on
the fracture; $h_{\gamma}= 2^{-3},\,h_{\gamma}=2^{-4},\,h_{\gamma}=2^{-5}$, where  the last choice corresponds to matching grids on the fracture.   In Table~\ref{tab:example1_mortar}, we plot the resulting number of
iterations required by each LDD solver. Particularly, we can see that
the sequential ItLDD solver in the matching grids has more difficulty to converge, so the effectiveness of the  MFB is more pronounced.  The monolithic solver MoLDD seems to be more robust with refining the mortar grids. Here,  the  computational cost of the construction of the inter-dimensional operator benefits from fewer mortar degrees on the fracture.
 \begin{figure}[tbp]
    \centering
    \includegraphics[width=0.2\textwidth]{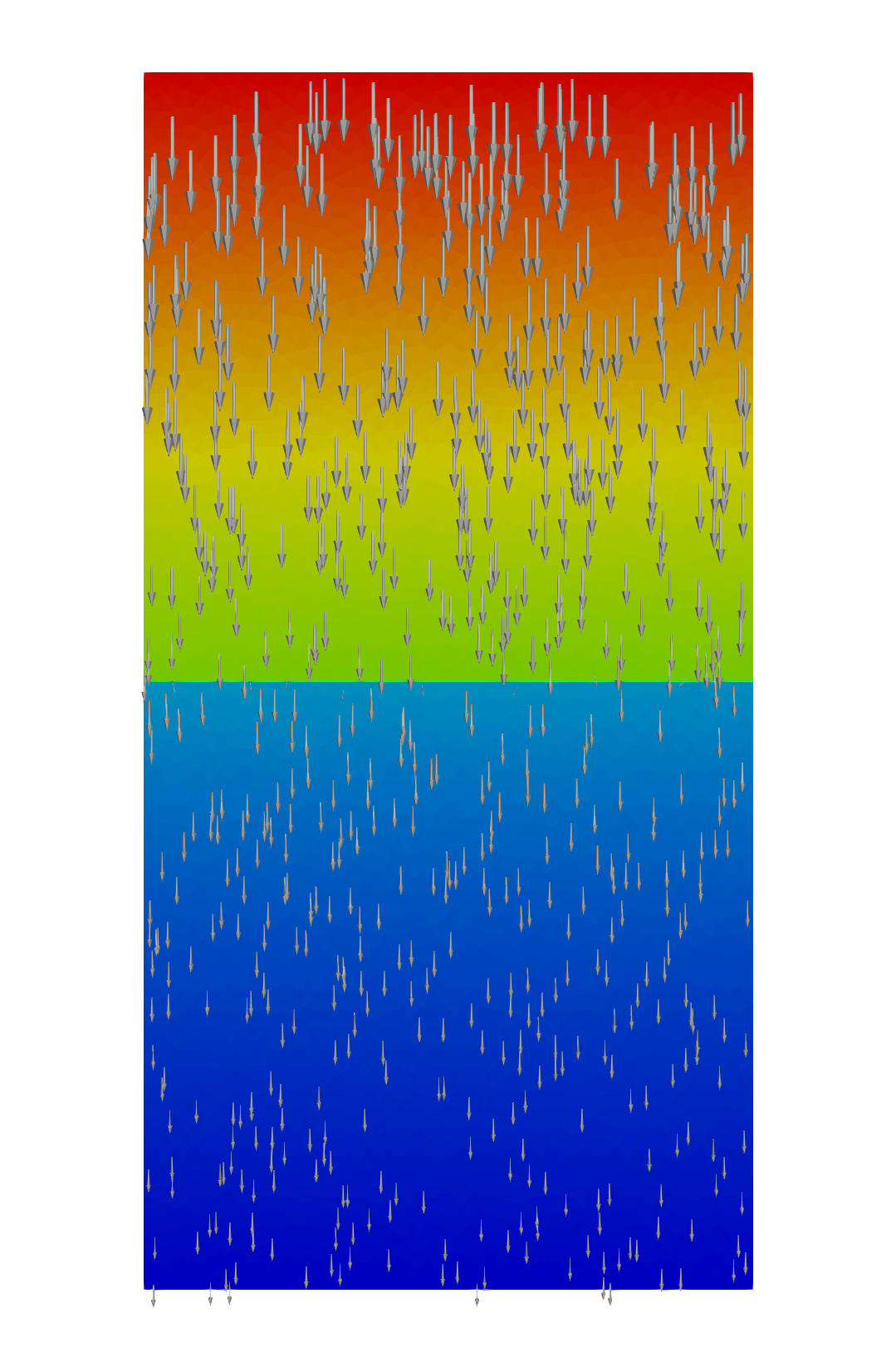}%
    \includegraphics[width=0.2\textwidth]{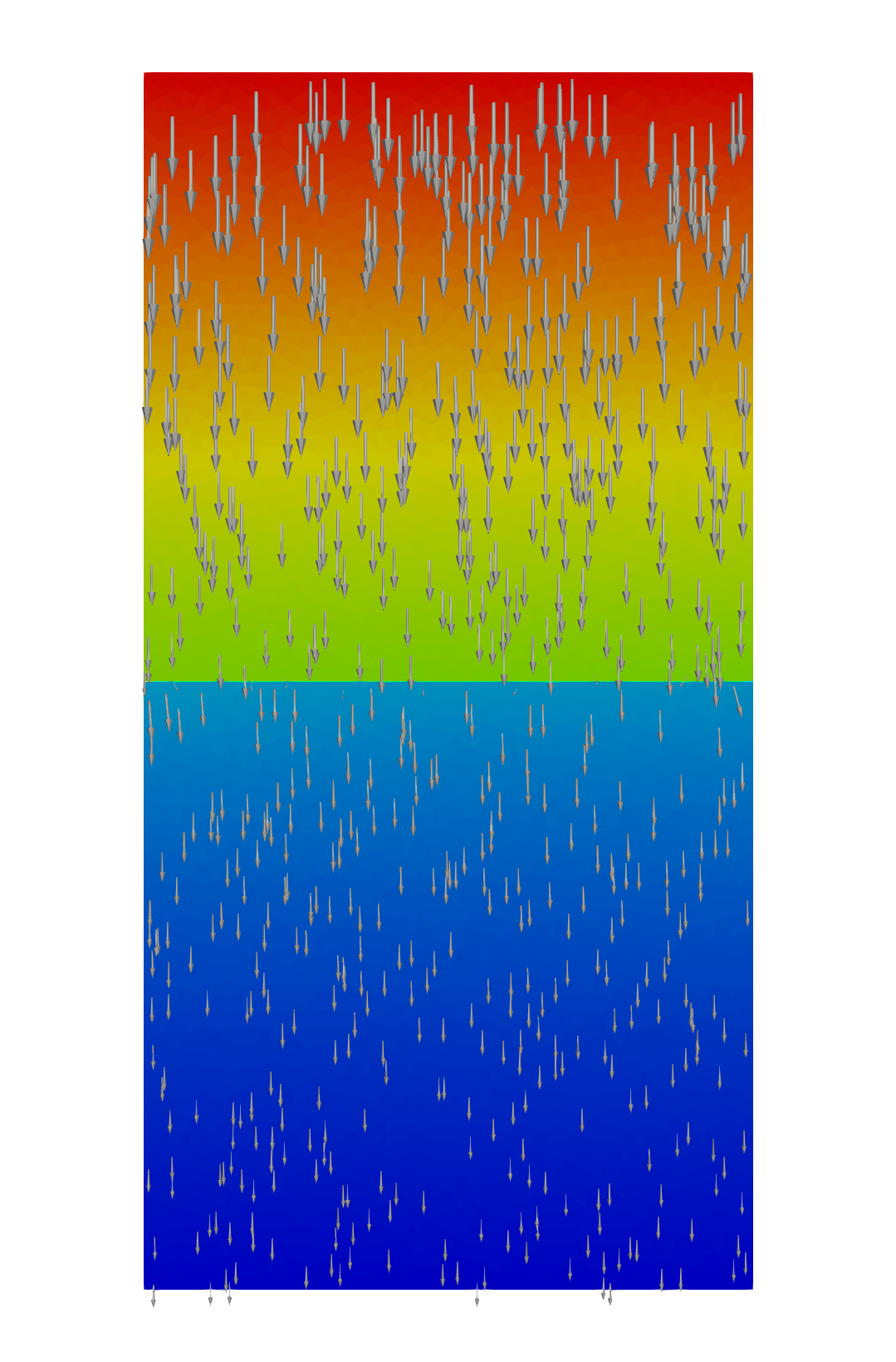}%
    \caption{$p \in [0,1]$ and $\mathbf{u}$ for the example in
    Subsection \ref{subsec:multiscale}  with $\vK_{\gamma}=10^{-4}\II$, $\beta=1$ and $\alpha_{\gamma}=1$. With fine (left) and coarse (right) mortar grids.}%
    \label{fig:example1_coarse}
\end{figure}
  Example of a
solution is depicted in Figure~\ref{fig:example1_coarse}, where we can see that
conforming and non-conforming grids (with $h_{\gamma}= 2^{-3}$) on the fracture give  indistinguishable results.
\begin{table}[tbp]
    \centering
    \begin{tabular}{l|r|r|r|r|r}
        \hline
        ${\sharp \text{cells} \backslash n}$ & 1 & 2 & 3 & 4 & 5 \\ \hline
            8   & 3 & 3 & 3 & 3 & 2\\
            16  & 3 & 3 & 3 & 3 & 2\\
            64  & 3 & 3 & 3 & 3 & 2\\
    \end{tabular}
    \hspace*{0.05\textwidth}
    \begin{tabular}{l|r|r|r|r|r}
        \hline
        ${\sharp \text{cells} \backslash n}$ & 1 & 2 & 3 & 4 & 5 \\ \hline
            8   & 11 & 10 & 10 & 9  & 9 \\
            16  & 11 & 10 & 10 & 9  & 9 \\
            64  & 18 & 15 & 15 & 15 & 14\\
    \end{tabular}
    \caption{Results for the example of Subsection \ref{subsec:multiscale}
    reporting the number of iterations for conforming and non-conforming (coarse scale) grids on the fracture.  Left
    table corresponds to solving with the MoLDD scheme, while the right one
    corresponds to solving with the ItLDD scheme.}%
    \label{tab:example1_mortar}
\end{table}

\subsection{Extension to other flow models: the Cross model}\label{subsec:extension}
The aim of this test case is to show that our LDD solvers can be applied to more general flow models.  On the fracture, we assume  the Cross flow model  to relate $p_\gamma$ and
$\vecu_\gamma$. We have the non-linear term given by
\begin{gather*}
    \xi(\vecu_\gamma) = \dfrac{(\omega_{0}-\omega_{\infty})\vecu_\gamma}{1+\FIX{\zeta}{c_{\omega}} \vert{\vecu_\gamma}
    \vert^{2-r}}.
\end{gather*}
The parameters $0\leq \omega_{\infty}<\omega_{0}$, $\FIX{\zeta}{c_{\omega}}$ and $r$ are positive
scalars related to the rheology of the considered liquid.
In~\eqref{Initial_system_d_f}, $\vK_{\gamma}$ is now replaced by
$\omega_{\infty}$. We let $\omega\eq\omega_{\infty}-\omega_{0}$ and set   $
\omega_{0}=2$, $\omega_{\infty}=1$, $\FIX{\zeta}{c_{\omega}} = 1$, and $ r = 1.5$. It is
\FIX{easy}{possible} to
verify that $ \xi $ satisfies the assumption \ref{ass_xi}. \FIX{For more details
s}{S}ee
\cite{MR3388812,MR2842139} and the references therein.

We choose  the iterative solver ItLDD and re-compute the simulations of
Subsection~\ref{subsec:stability} and~\ref{subsec:robust} for the  Cross flow
model. We set then $ L_{\gamma, u} = L_{\xi}/2=0.5 $ and $ L_{\gamma,
p}=\alpha_{\gamma}= 10^3$ as derived from the theory. The results (not shown)
demonstrate first the stability of the ItLDD solver with respect to  the
parameters $ h $ and $\tau$.  Crucially,  all the simulations in this example do
not require additional computational cost (except  fracture solves), as we use
the same MFB inherited from the  Forchheimer model. We set $h=2^{-5}$ with
slightly coarse grids on the fracture $h_{\gamma}=2^{-4}$ and $\tau=2^{-4}$.

\begin{figure}[tbp]
    \centering
    \subfloat[ItLDD scheme]{
    \includegraphics[width=0.33\textwidth]{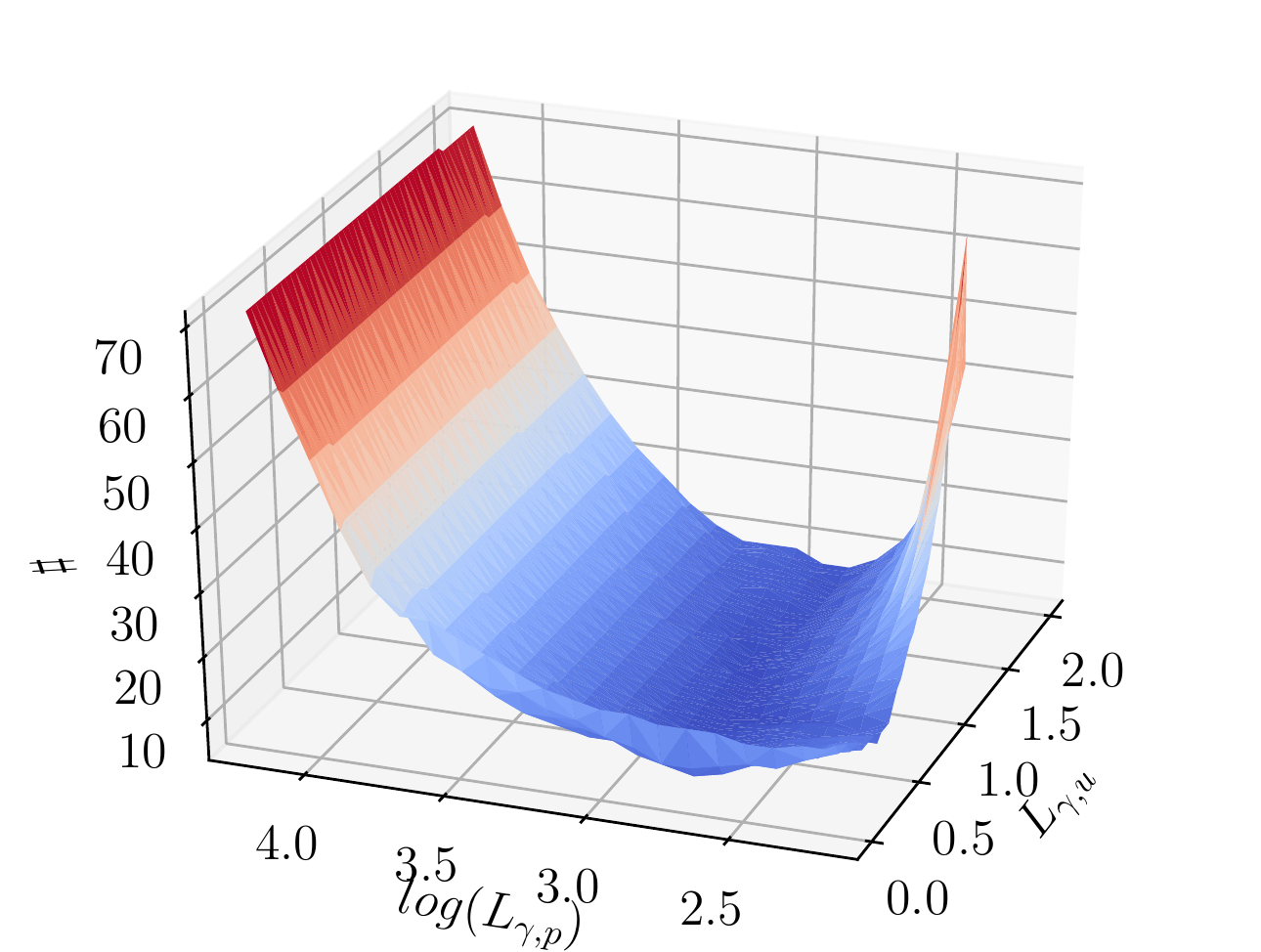}\label{fig:example2_iter}}%
    \subfloat[$p$ and $\mathbf{u}$]{
    \includegraphics[width=0.2\textwidth]{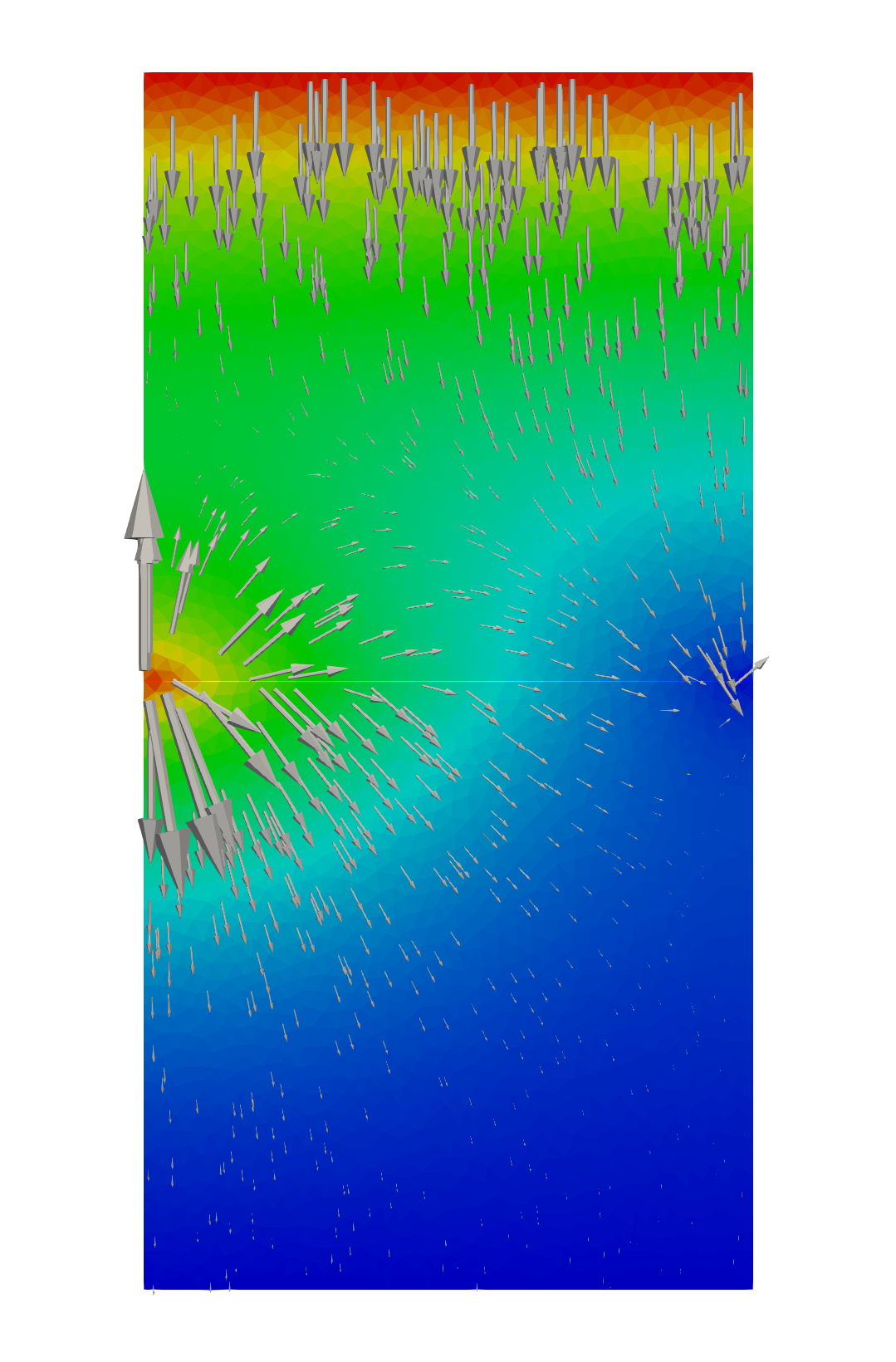}%
    \includegraphics[width=0.2\textwidth]{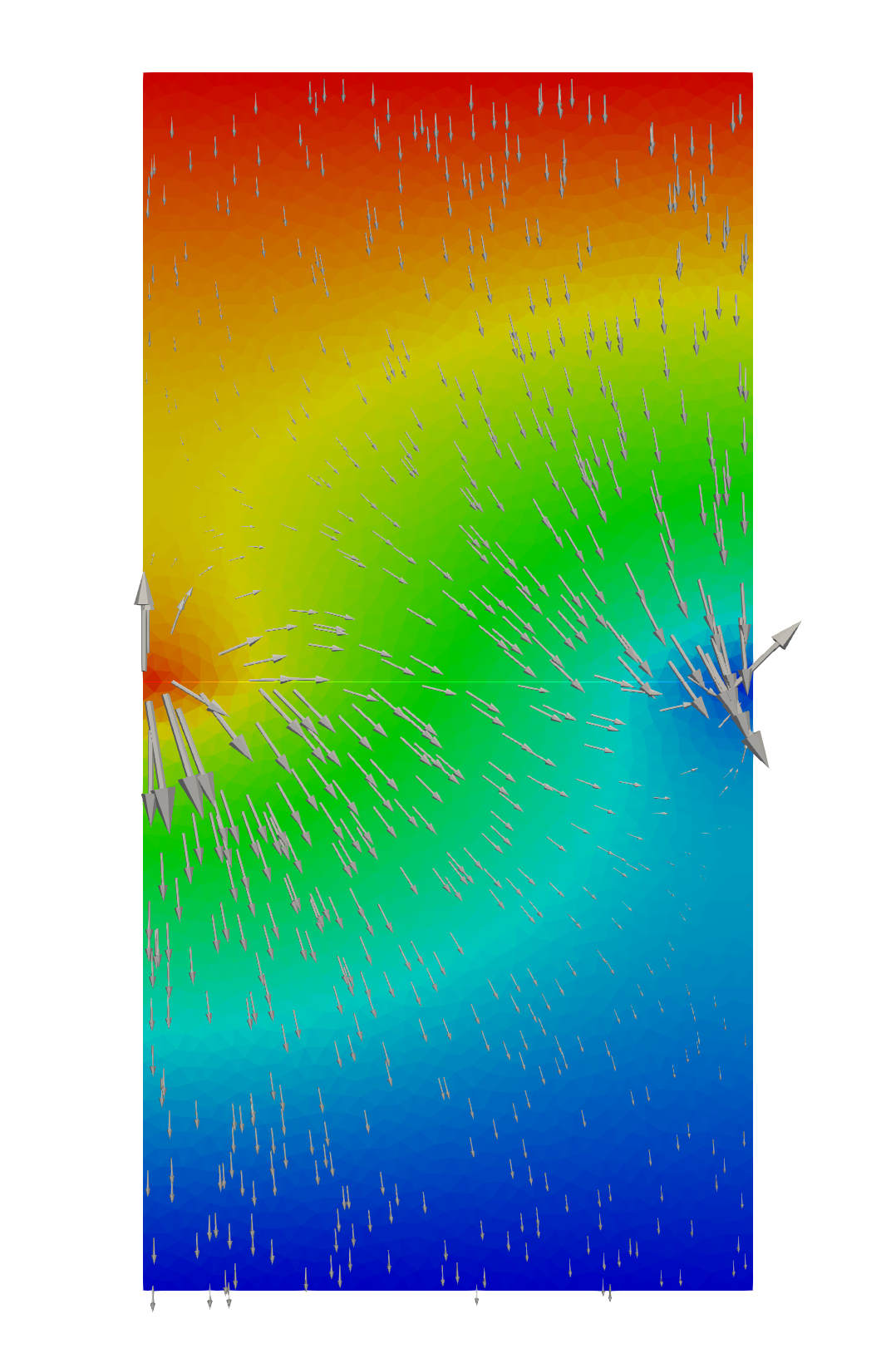}%
    \label{fig:example2}
    }
    \caption{Results for the example of Subsection \ref{subsec:extension}\FIX{
    using ItLDD scheme}{}.
    \FIX{W}{On the left, w}e report the number of iterations $\sharp$ for different values of $
        L_{\gamma,u} $ and $ L_{\gamma,p} $\FIX{}{, for the third time step}.
        \FIX{
        On the left for the first time step, in the centre for the third, and on
        the right for the last time step.}{On the right, $p \in [0,1]$ and $\mathbf{u}$ for the example in
    Subsection \ref{subsec:extension} with $\omega=1$, $c_{\omega}=1$, and $r=1.5$.}}%
\end{figure}


In \Cref{fig:example2_iter}, we show the results for the ItLDD
solver on a set of realizations of  $( L_{\gamma, u}, L_{\gamma, p})$. The results do not differ greatly comparing to the case of Forchheimer's flow
model. The convexity of the surface plots in all time steps is clear giving away
an optimal combination of values for $ L_{\gamma, u} $ and $ L_{\gamma, p} $.
For example, we can find minimum of $ 5 $ iterations for $ L_{\gamma, u} $
between approximately $ 0.73 $ and $ 1.57 $, and $ L_{\gamma, p} $ between $
10^{2.8} $ and $ 10^3 $. Note that the parameters prescribed by the theoretical
results, $ L_{\gamma, u} =0.5 $ and $ L_{\gamma, p}= 10^3 $, form  a good
candidate in    this simulation. Finally, we mention that we can use an
optimization process, as detailed in~\cite{storvik2018optimization}, in order to
get the optimal values. Precisely, the fact that the choice of the stabilization
parameters is independent of the mesh size, one can then run the LDD solver on a
coarse spatial mesh and one time step, and study the stabilization parameters in
specific intervals centred around the theoretical values. The parameters that
give the lowest number of iterations are  then used for the real computations.
This ``brute-force'' optimization is simple to do in practice when using the MFB.}

\begin{table}[tbp]
    \centering
    %
    \begin{tabular}{l|r|r|r|r|r}
        \hline
        ${\omega\backslash n}$ & 1 & 2 & 3 & 4 & 5 \\\hline
        0.1 & 15 & 11 & 10 & 9 & 8     \\
        1   & 10 & 9 & 8 & 8 & 8      \\
        2.5 & 30 & 19 & 16 & 14 & 12 \\ \hline
    \end{tabular}
    \hspace*{0.04\textwidth}
    \begin{tabular}{l|r|r|r|r|r}
        \hline
        ${\FIX{\zeta}{c_{\omega}}\backslash n}$ & 1 & 2 & 3 & 4 & 5 \\\hline
        1   & 10 & 9 & 8 & 8 & 7 \\
        10  & 11 & 9 & 9 & 8 & 7 \\
        100 & 16 & 11 & 10 & 9 & 8  \\ \hline
    \end{tabular}
    \hspace*{0.04\textwidth}
    \begin{tabular}{l|r|r|r|r|r}
        \hline
        ${r\backslash n}$ & 1 & 2 & 3 & 4 & 5 \\\hline
        1   & 10 & 9 & 9 & 8 & 7  \\
        1.5 & 10 & 9 & 8 & 8 & 7  \\
        4.5 & 17 & 11 & 10 & 9 & 8  \\ \hline
    \end{tabular}
    \caption{Results for the example of Subsection \ref{subsec:extension}.
        On the left the number of iterations by varying the values of $\omega$.
        In the center when $\FIX{\zeta}{c_{\omega}}$ changes, while on the right for different
        values of $r$.}%
    \label{tab:example2_param}
\end{table}

In Table~\ref{tab:example2_param}, we consider to test the dependency of the
number of iterations on the rheology parameters of the flow model.   We provide results of
several tests on $\omega$, $\FIX{\zeta}{c_{\omega}}$, and $r$. While testing for one of the
parameters, the other two are fixed to either  $ \omega= 1 $, $ \FIX{\zeta}{c_{\omega}} = 1 $ or
$ r = 1.5 $.    We can observe that $ \omega$ strongly
influences the performance of both methods making it difficult to converge when
$ \omega$ gets larger, that is, when the non-linearity is stronger. For larger
values of $\omega$ the number of iterations increases drastically, suggesting the necessity  to adjust the $L$-scheme parameters as well as to use the MFB. The number of iterations was  less dependent of the parameter $ \FIX{\zeta}{c_{\omega}} $. This parameter itself contributes less to the strength of the non-linearity in comparison to $ \omega$, and, thus,
influencing less the performance of the  solver. Finally, we can again
notice a moderate dependency of number of iterations on parameter $ r $. This is
especially shown when $ r > 2 $ and the exponent on the vector norm of $
\vecu_{\gamma} $ becomes negative. Thus, the non-linear flow function $ \xi $ is
exponential in the values of $ \vecu_{\gamma} $ and accounts for the very fast
flow in the fractures. We finally recall that the robustness study drawn in Table~\ref{tab:example2_param} has the cost of one realization with  fixed-parameters,  confirming the role  of the MFB in our solvers. For the robustness of LDD solvers with respect to the matrix-fracture coupling effects induced by the parameter $\alpha_{\gamma}$, we have seen  that both solvers are robust when strengthening or weakening the coupling effects (results not shown). Example of solution is reported in Figure \ref{fig:example2}.

\FIX{
\begin{figure}[tbp]
    \centering
    \includegraphics[width=0.2\textwidth]{case2_0}%
    \includegraphics[width=0.2\textwidth]{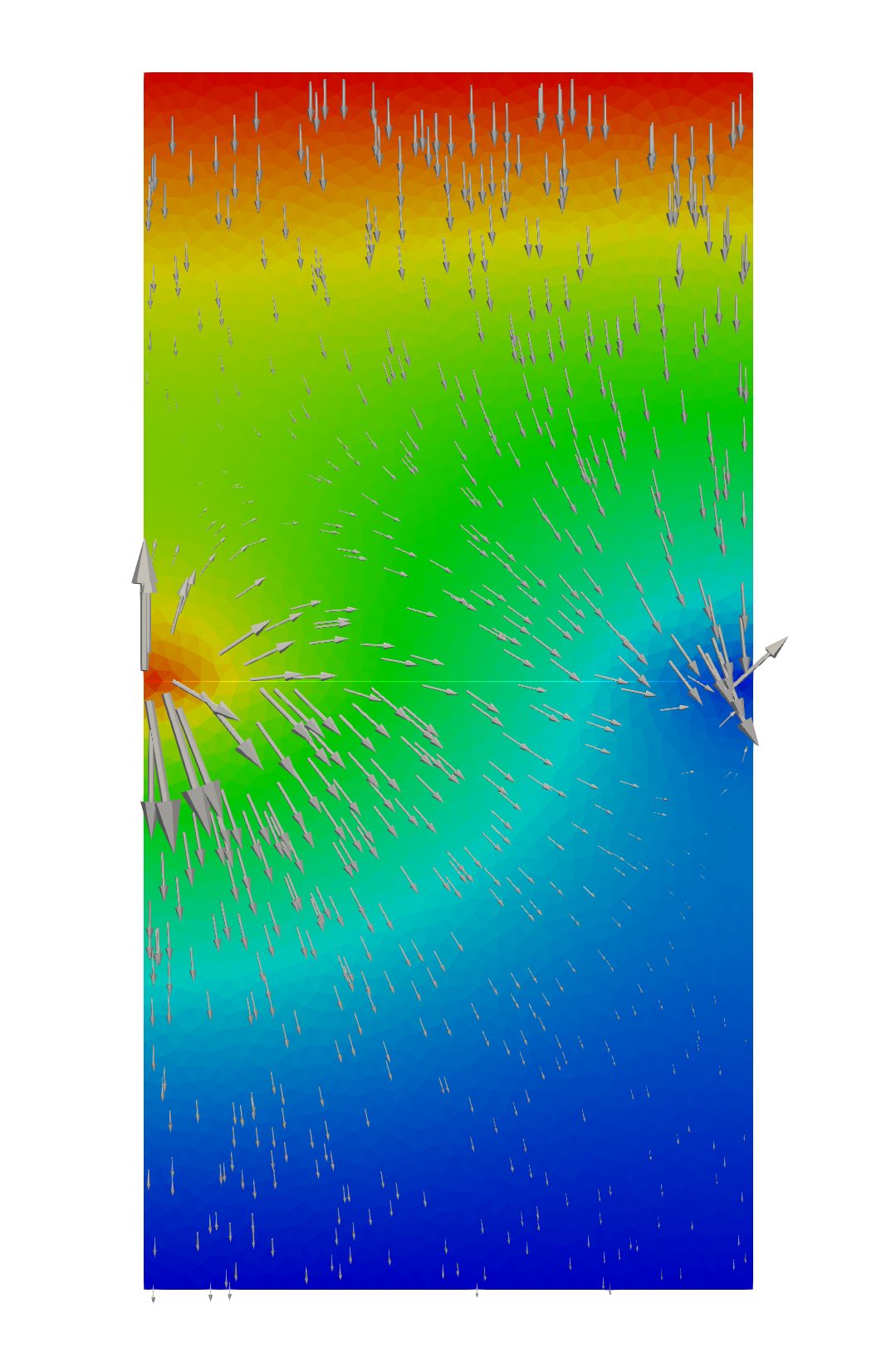}%
    \includegraphics[width=0.2\textwidth]{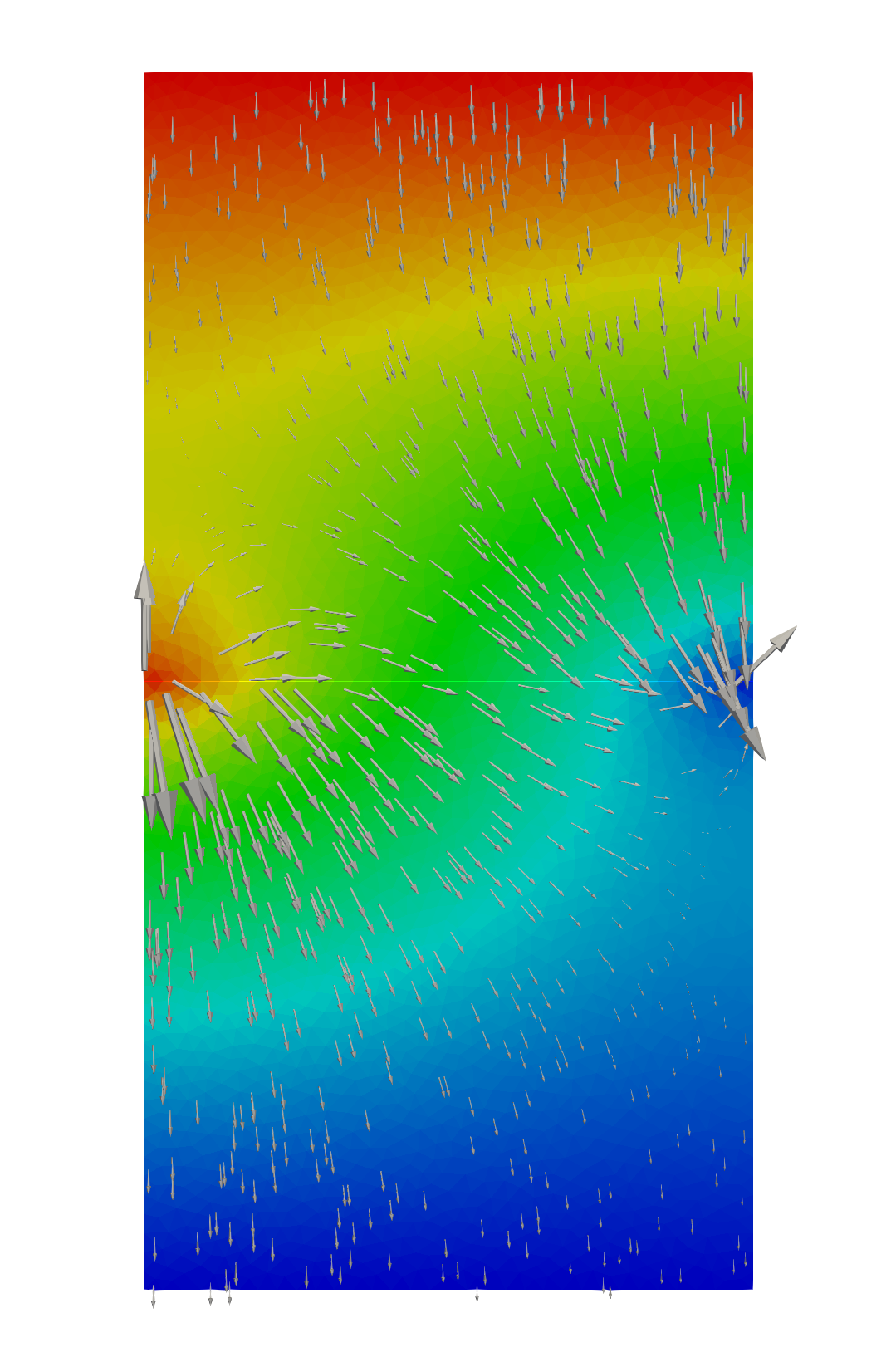}%
    \includegraphics[width=0.2\textwidth]{case2_3}%
    \includegraphics[width=0.2\textwidth]{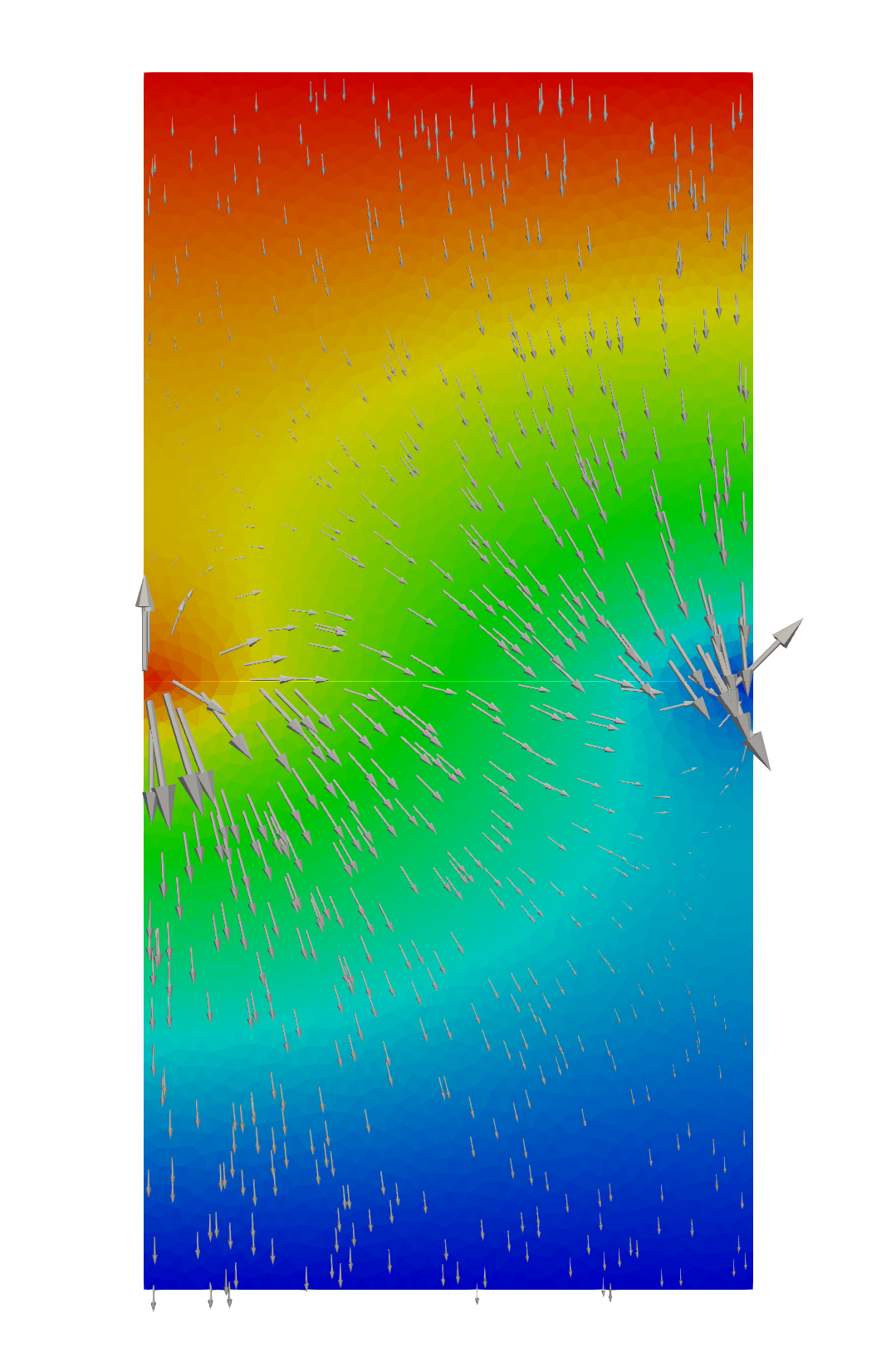}%
    \caption{$p \in [0,1]$ and $\mathbf{u}$ for the example in
    Subsection \ref{subsec:extension} with $\omega=1$, $\zeta=1$, and $r=1.5$.}%
    \label{fig:example2}
\end{figure}}{}

\section{Conclusions}\label{sec:conclusion}
In this study, we have presented two new strategies to solve a compressible
single-phase flow problem in a porous medium with a fracture. In the porous medium, we have considered the
classical Darcy relation between the velocity and the pressure while, in the
fracture, a general non-linear law. We employ the L-scheme to
handle the non-linearity term, but also to treat the
inter-dimensional coupling in the second proposed algorithm. To further achieve computational speed-up, the linear
Robin-to-Neumann co-dimensional map is constructed in an offline phase resulting
in a problem reduced only to the fracture interface. This approach allows to change the
fracture parameters, or the fracture flow model in general, without the need to recompute the problem associated with the
rock matrix. We have shown the existence of  optimal values for the L-scheme parameters, which are validated through several numerical tests.
Future developments can be explored towards domain decomposition in time, where fast and
slow fractures are solved asynchronously.

\section*{Acknowledgements}

We acknowledge the financial support from the Research Council of Norway for the
TheMSES (no. 250223) and ANIGMA projects (no. 244129/E20) through the ENERGIX program.
\bibliographystyle{siamplain}
\bibliography{BIOTbib,biblio}

\end{document}